\documentclass[10pt,a4paper]{article}  

\usepackage[left=2.5cm,right=2.5cm,top=2cm,bottom=2cm]{geometry}
\usepackage[dvipsnames]{xcolor}
\usepackage[small]{caption}
\usepackage{subcaption}
\usepackage{amsmath}
\usepackage[backend=biber,url=false]{biblatex}
\usepackage[hidelinks]{hyperref} 

\usepackage{amsthm}
\newtheorem*{theorem}{Theorem}
\newtheorem*{lemma}{Lemma}

\usepackage[utf8]{inputenc}
\usepackage[english]{babel}
\usepackage[T1]{fontenc}
\usepackage{amssymb}
\usepackage{graphicx}
\usepackage{svg}
\usepackage{float}
\usepackage{tabularx}
\usepackage{multirow}
\usepackage{empheq}
\usepackage[bottom]{footmisc}
\usepackage{url}
\usepackage{csquotes}
\usepackage{eso-pic}
\usepackage{gensymb}
\usepackage{frcursive}
\usepackage{mathrsfs}  
\usepackage{mathtools}
\usepackage[shortlabels]{enumitem}
\usepackage{tikz}
\usepackage{authblk}  
\usepackage{array}
\newcolumntype{x}[1]{>{\centering\let\newline\\\arraybackslash\hspace{0pt}}p{#1}}

\newcommand*\circled[1]{\tikz[baseline=(char.base)]{
            \node[shape=circle,draw,inner sep=2pt] (char) {\hspace{-0.05cm}\small{#1}};}}
\usepackage{marvosym}
\usetikzlibrary{arrows.meta}
\usepackage{romanbar}
\usepackage[nottoc]{tocbibind}
\usepackage{array}
\newcolumntype{M}[1]{>{\centering\arraybackslash}m{#1}}
\newcolumntype{C}[1]{>{\centering\arraybackslash}p{#1}}
\newcolumntype{L}[1]{>{\raggedright\let\newline\\\arraybackslash\hspace{0pt}}p{#1}}

\usepackage{titlesec}
\titleclass{\subsubsubsection}{straight}[\subsection]
\newcounter{subsubsubsection}[subsubsection]
\renewcommand\thesubsubsubsection{\thesubsubsection.\arabic{subsubsubsection}}
\titleformat{\subsubsubsection}
  {\normalfont\normalsize\bfseries}{\thesubsubsubsection}{1em}{}
\titlespacing*{\subsubsubsection}
{0pt}{3.25ex plus 1ex minus .2ex}{1.5ex plus .2ex}
\makeatletter
  {\normalfont\normalsize\bfseries}
\def\toclevel@subsubsubsection{4}
\def\l@subsubsubsection{\@dottedtocline{4}{7em}{4em}}
\makeatother
\definecolor{darkblue}{rgb}{0,0.2,0.6}
\newcommand{\lk}[1]{{\textcolor{black}{#1}}}

\title{Pulled, pushed or failed: the demographic impact of a gene drive can change the nature of its spatial spread}

\author[,1]{Léna Kläy\thanks{Corresponding author: \texttt{lena.klay@sorbonne-universite.fr}}}
\author[2]{Léo Girardin}
\author[2]{Vincent Calvez}
\author[1]{Florence Débarre}
 
\affil[1]{\small{Institute of Ecology and Environmental Sciences Paris (IEES Paris), Sorbonne Université, CNRS, IRD, INRAE, Université Paris Est Creteil, Université de Paris, Paris Cedex 5, France.}}
\affil[2]{\small{Institut Camille Jordan, UMR 5208 CNRS and Universite Claude Bernard Lyon 1, France}}

\definecolor{darkbluelight}{rgb}{0,0.2,0.5}
\definecolor{darkgray}{gray}{0.3}
\definecolor{bleu}{rgb}{0,0.5,0.8}
\definecolor{vert}{rgb}{0,0.42,0.24}
\definecolor{rose}{rgb}{0.8,0,0.5}
\definecolor{orange}{rgb}{0.99,0.69,0.07}

\newcommand{\n}[1]{n_{_\mathrm{#1}}}
\newcommand{\p}[1]{p_{_\mathrm{#1}}}
\newcommand{\F}[1]{\mathscr{F}_{_\mathrm{#1}}}

\begin{document} 




\maketitle

\vspace{0.5cm}


\section*{Abstract}

Understanding the temporal spread of gene drive alleles -- alleles that bias their own transmission -- through modelling is essential before any field experiments. In this paper, we present a deterministic reaction-diffusion model describing the interplay between demographic and allelic dynamics, in a one-dimensional spatial context. We focused on the traveling wave solutions, and more specifically, on the speed of gene drive invasion (if successful). We considered various timings of gene conversion (in the zygote or in the germline) and different probabilities of gene conversion (instead of assuming 100$\%$ conversion as done in a previous work). We compared the types of propagation when the intrinsic growth rate of the population takes extreme values, either very large or very low. When it is infinitely large, the wave can be either successful or not, and, if successful, it can be either pulled or pushed, in agreement with previous studies (extended here to the case of partial conversion). In contrast, it cannot be pushed when the intrinsic growth rate is vanishing. In this case, analytical results are obtained through an insightful connection with an epidemiological SI model. We conducted extensive numerical simulations to bridge the gap between the two regimes of large and low growth rate. We conjecture that, if it is pulled in the two extreme regimes, then the wave is always pulled, and the wave speed is independent of the growth rate. This occurs for instance when the fitness cost is small enough, or when there is  stable coexistence of the drive and the wild-type in the population after successful drive invasion. Our model helps delineate the conditions under which demographic dynamics can affect the spread of a gene drive.

\newpage

\tableofcontents

\newpage

\section{Introduction}


A highly accurate, cost-effective and easy-to-use technology, the CRISPR-Cas genome editing system has been favoring the development of promising innovations \cite{jinek2012}. Among them, CRISPR-Cas9 gene drive \cite{alphey2020}, which aims to spread a trait of interest in a wild type population in a relatively short number of generations \cite{kyrou2018}. Application fields are numerous, and include i) the eradication of insect-borne diseases \cite{buchman2020, gantz2015a, kyrou2018}; ii) the elimination of herbicide and pesticide resistance in pest populations \cite{neve2018}; iii) the control of destructive invasive species \cite{gantz2015, grunwald2019}; iv) the conservation of biodiversity by spreading beneficial traits in endangered species \cite{esvelt2014, rode2019}.

Targeting sexually reproducing species, CRISPR-Cas9 gene drive biases the transmission of an allele from a parent to its offspring. This biased inheritance occurs through gene conversion (also called ``homing'' \cite{deredec2008}): in a heterozygous cell, the gene drive cassette present on one chromosome induces a double-strand break at a specific target site on the homologous chromosome, and the repair process duplicates the cassette. Overall, this process increases the chances of transmitting the gene drive cassette compared to its wild-part counterpart, and the mechanism repeats through the generations. Gene conversion can potentially take place at different timings of the life cycle: from very early on, in the zygote, meaning that potentially every single cell of the individual could become homozygous for the gene drive, to, in the germline, where only the gametes are converted. 


Gene drives can be classified into two main categories depending on the purpose of their use \cite{dhole2020, girardin2021}. A ``replacement drive'' is aimed at spreading a genetic modification in order to introduce an important and durable feature in the natural population. Population size is then not significantly affected and the drive construct may in principle persist indefinitely in the environment. A ``suppression drive'' on the other hand is meant to reduce population size by spreading a detrimental trait, such as a sex ratio distorter \cite{zotero-193} or by altering fertility \cite{kyrou2018}, for example. The term ``eradication drive'' can be used for the extreme case where population extinction is the aim. 
As with any new tool, it is essential to balance risks (safety) and benefits (efficacy) of the technique before running any field trials. Experiments currently conducted in laboratories provide small- to medium-scale information; mathematical models can help to extend these empirical results and identify the features that are the most important in determining the dynamics at larger scales \cite{committeeongenedriveresearch2016}.


Early gene drive models \cite{burt2003, deredec2008, unckless2015} used classical population genetics frameworks, and considered discrete non-overlapping generations in a well-mixed population. These simplifications helped to draw general conclusions, but it is important to challenge them. First of all, most of the species targeted in the context of gene drive do not have synchronous generations (for instance mosquitoes \cite{gantz2015a, hammond2015, buchman2020, kyrou2018}, flies \cite{gantz2015}, mice \cite{grunwald2019}). Secondly, the assumption of a single well-mixed collection of individuals living across a uniform space is usually not realistic. In fact, most of the natural landscapes are heterogeneous. Individuals are also more likely to interact with others that are in closer proximity, which might result in local genetic variations. Finally, releases of transgenic individuals are limited in range, which is another factor of spatial heterogeneity.

Taking into account spatio-temporal dynamics of the population size is another key step towards more realistic models. For the sake of simplicity, most early models focused on allele frequencies and considered a constant population density. However in the context of gene drive, the introduction of maladapted transgenic individuals can lead to the reduction (or even extinction) of the population \cite{dhole2020}. 
When considering a spatially structured population, variations in population density naturally generate a demographic flux from denser to less dense areas. This demographic flux is directed in opposition to the spread of the drive allele. It was previously shown \cite{girardin2021} that the advantage conferred by gene conversion may nevertheless counteract the demographic effect linked to the fitness cost. 

The main goal of this paper is to clarify the impact of variations in population density over the course of drive propagation over space. 

We study partial differential equations which follow the propagation of the drive in space and time. We explore numerically and analytically two models: a first model based on perfect conversion in the zygote, already introduced in \cite{girardin2021} in a spatially structured population, corresponding to an idealized case where gene conversion always succeeds; second, a more realistic model with partial conversion and presence of heterozygous individuals, already studied in \cite{rode2019} in a well-mixed, non spatial population. In order to investigate the possible spreading of gene drives through space after local introduction, we focus on the description of traveling waves solutions, that is, particular solutions which are stationary in a frame moving at constant speed. Our analysis goes beyond \cite{girardin2021} by several means: we extend it to the case of partial conversion, and we systematically analyze the case where the demographic effects are the strongest, in the regime of vanishing growth rate. The latter is possible through an insightful connection with an epidemiological SI model.

\newpage 

\section{Methodology}\label{sec:methodo}


\subsection{Models}\label{subsec:models}
We present our model step-by-step. For a genetically and spatially homogeneous population, we consider the following (non-dimensionalized) equation: \begin{equation}
    \partial_t n(t)  =  \Big( 1 + r \left(1-n(t)\right) \Big) f \ n(t)  - n(t)  \quad \quad (\forall t>0).
\end{equation} 

\lk{where the unit of time is generations}. Fecundity is density-dependent, and parametrized by the fitness $f$ and the rate $r$ at which the ($f-$dependent) carrying capacity is restored. 
When $f=1$, the carrying capacity is $1$, and we recover the logistic equation $  \partial_t n(t) =  r \left(1-n(t)\right) n(t)$. Other modelling options are discussed in \cite{girardin2021}.

Then, we add genetic diversity in the population. We still denote by $n$ the total density, and by $n_i$ the density of individuals with genotype $i$. The population we consider is diploid, sexually reproducing, and the fitness $f_i$ depends on the genotype. The dynamics are given by the following equations: 
\begin{equation}
    \partial_t n_i(t)  =  \Big(1 + r (1-n(t)) \Big) f_i  \ n(t) \underbrace{ \sum\limits_{l,k}  \pi_{l,k}^i \ \dfrac{n_l(t)}{n(t)} \ \dfrac{n_k(t)}{n(t)} }_{\text{Mating term}}  - n_i(t) \quad \quad (\forall t>0) \ (\forall i).
\end{equation}

The mating term takes into account the probability for each couple of parents $l$,$k$ to have offspring of type $i$ ($\pi_{l,k}^i$), multiplied by the probability of a mating event $l$,$k$ ($ \frac{n_l(t) n_k(t)}{n(t)^2}$), assuming random mating.

Last but not least, we consider a spatially structured population. We assume that the movement of individuals is described by a diffusion term with equal diffusion coefficients, normalized to 1. Since we focus on traveling wave solutions, we restrict our analysis to a one-dimensional space. We obtain the following equations:



\begin{equation}\label{eq:mastereq}
    \partial_t n_i(t,x) - \partial_{xx}^2 n_i  (t,x) =  \Big(1 + r (1-n(t,x)) \Big) f_i  \ n(t)  \sum\limits_{l,k}  \pi_{l,k}^i \dfrac{n_l(t)}{n(t)} \ \dfrac{n_k(t)}{n(t)}   - n_i(t,x) \quad \quad (\forall t >0) \ (\forall x \in \mathbb{R}) \ (\forall i).
\end{equation}

There are two possible alleles at the locus that we consider: the wild-type allele ($W$) and the drive allele ($D$). We have three genotypes: wild-type homozygotes ($i = WW$), drive homozygotes ($i = DD$) and heterozygotes ($i = DW$).  Wild-type homozygotes have fitness \lk{$f_{WW} = 1$}, drive homozygotes have fitness $f_{DD} = 1 - s$, where $s$ is the fitness cost of the drive, and drive heterozygotes have fitness $f_{DW} = 1 - s h$, where $h$ is the dominance parameter (see Table \ref{tab:variables}).

\begin{table}[H]
\centering
\begin{tabular}{rccc}
     & \textbf{Density} & \textbf{Adult genotype} & \textbf{Fitness}  \\
    Drive Homozygote & $\n{DD}$ &  D D &  $1-s$ \\
    Heterozygote  & $\n{DW}$ &  W D &  $1-sh$ \\
    Wild-type Homozygote & $\n{WW}$ & W W &  $1$ \\
\end{tabular}
\caption{\label{tab:variables} Population characteristics (D: Drive allele, W: Wild-type allele).}
\end{table}

All along the paper, we assume  $s \in (0,1)$, corresponding to a fitness cost carried by the drive alleles. Furthermore, we assume that the fitness of heterozygotes cannot be greater than the fitness of either homozygote ($h \in [0,1]$).


Gene conversion turns a heterozygous cell into a drive homozygous cell. To determine the probability $ \pi_{l,k}^i $ (probability for a couple $l$,$k$ to have offspring of type $i$), we need to take into account both the probability $c\in[0,1]$ with which gene conversion occurs in heterozygotes, and the stage of the life cycle at which it occurs: either in the zygote, or in the germline. This last feature modifies significantly the probabilities: for example, a couple $W$, $D$ of gametes has a probability $1-c$ to lead to heterozygous offspring if conversion occurs in the zygote, whereas this probability becomes one if conversion occurs in the germline. We detail all $ \pi_{l,k}^i $ values in Appendix \ref{an:growthterm}. For the sake of clarity, we now omit variables in the notation ($n_i = n_i(t,x)$).

The parameters are summarized in  Table~\ref{tab:parameters}.


\begin{table}[H]
\centering
\begin{tabular}{lll}
    \textbf{Parameters} & \textbf{Range values} & \textbf{Description}\\
    $r$  & $(0,+ \infty)$ &  Intrinsic growth rate \\
   $c$  & $[0,1]$ &  Conversion rate \\
   $s$ & $(0,1)$ & Fitness cost of drive homozygotes  \\
   $h$  & $[0,1]$ & Drive dominance \\
\end{tabular}
\caption{\label{tab:parameters} Model parameters.}
\end{table}

In this article, we will analyse the three following versions/variations of model~\eqref{eq:mastereq}:

\textbf{Partial conversion occurring in the zygote:}\begin{equation} \label{eq:par_zyg}
\left\{ \small
    \begin{array}{ll}
       \partial_t \n{DD} - \partial_{xx}^2 \n{DD}  = (1-s) (r \ (1-n)+1) \ \dfrac{ c   \ \n{WW} \n{DW} + 2\  c \  \n{WW} \n{DD} + (\frac{1}{2} \ c + \frac{1}{4}) \  \n{DW}^2 +  (c+1) \   \n{DW} \n{DD} + \n{DD}^2 }{n}  - \n{DD},\\ 
         \\
       \partial_t \n{DW} - \partial_{xx}^2 \n{DW} =  (1-sh) (r \ (1-n)+1) \  \ (1-c)  \ \dfrac{  \n{WW} \n{DW} + 2 \ \n{WW} \n{DD} + \frac{1}{2} \  \n{DW}^2 +  \n{DW} \n{DD}}{n} - \n{DW}, \\ 
         \\
        \partial_t \n{WW} - \partial_{xx}^2 \n{WW} = (r \ (1-n)+1) \ \dfrac{ \n{WW}^2 +  \n{WW} \n{DW} + \frac{1}{4} \ \n{DW}^2}{n} - \n{WW} . 
    \end{array}
    \right.
\end{equation}

\textbf{Partial conversion occurring in the germline:}\begin{equation}
\label{eq:par_ger}
 \left\{ \small
    \begin{array}{ll}
      \partial_t \n{DD} - \partial_{xx}^2 \n{DD} =  (1-s) (r \ (1-n)+1) \ \dfrac{ \frac{1}{4} \ (1+c)^2 \  \n{DW}^2 +  (1+c) \ \n{DW} \n{DD} + \n{DD}^2  }{n}  - \n{DD} ,\\
        \\
      \partial_t \n{DW} - \partial_{xx}^2 \n{DW} =  (1-sh) (r \ (1-n)+1) \ \dfrac{ (1+c) \ \n{WW} \n{DW} + 2 \ \n{WW} \n{DD} + \frac{1}{2} \ (1-c^2) \ \n{DW}^2 +  (1-c)  \ \n{DW} \n{DD} }{n}  - \n{DW}, \\
        \\
      \partial_t \n{WW} -  \partial_{xx}^2 \n{WW} =  (r \ (1-n)+1) \  \dfrac{ \n{WW}^2 +  (1-c) \  \n{WW} \n{DW} + \frac{1}{4} \ (1-c)^2 \  \n{DW}^2 }{n}  -  \n{WW} .
    \end{array}
\right. 
\end{equation}


\textbf{Perfect conversion occurring in the zygote (no heterozygotes):}

For a perfect conversion occurring in the zygote ($c=1$), model (\ref{eq:par_zyg}) reduces to the following set of two equations, which was introduced in \cite{girardin2021}:

\begin{equation} \label{eq:per_zyg}
\left\{
    \begin{array}{ll}
       \partial_t \n{DD} - \partial_{xx}^2 \n{DD} & = (1-s) \ \Big(r  \ (1-\n{DD}-\n{WW}) + 1 \Big) \ \dfrac{\n{DD}^2 + 2 \ \n{WW} \n{DD}}{\n{WW}+\n{DD}} - \n{DD} = F_D(\n{DD}, \n{WW})\\
       \\
       \partial_t \n{WW} - \partial_{xx}^2 \n{WW} & =  \Big(r \ (1-\n{DD}-\n{WW}) + 1 \Big)  \ \dfrac{\n{WW}^2}{\n{WW}+\n{DD}} - \n{WW} = F_W(\n{DD}, \n{WW}).
    \end{array}
\right.
\end{equation}

This last model only follows the two homozygous genotypes, drive and wild-type. Due to perfect gene conversion ($c=1$), no heterozygous individuals are ever produced: heterozygous eggs are all transformed into homozygotes. Further assuming that there are no heterozygotes initially, we only need to follow the densities of homozygotes. 


Note that system \eqref{eq:per_zyg} can also be obtained from model~\eqref{eq:par_ger} by assuming perfect conversion in the germline ($c = 1$) and drive dominance ($h = 1$). In this case, heterozygotes and drive homozygotes have the same fitnesses, and both only produce gametes with the drive allele. We can then group them together and follow their density $\n{DW}+\n{DD}$, whose dynamics are given by the first line of \eqref{eq:per_zyg}.



\subsection{Setting of the problem}\label{approach}

\subsubsection*{Traveling waves}

We seek stationary solutions in a reference frame moving at speed $v$, where $v$ is some unknown: 

\begin{equation} \label{eq:trav_waves}
\left\{
    \begin{array}{ll}
       \n{DD}(t,x) = \n{DD}(x-vt) \quad \quad (\forall t >0) \ (\forall x \in \mathbb{R}), \\
       \n{DW}(t,x) = \n{DW}(x-vt) \quad \quad (\forall t >0) \ (\forall x \in \mathbb{R}), \\
       \n{WW}(t,x) = \n{WW}(x-vt) \quad \ \  (\forall t >0) \ (\forall x \in \mathbb{R}).\\
    \end{array}
\right.
\end{equation}



Traveling wave solutions contain important information for the biological interpretation of the results, such as the speed of invasion $v$, the genetic composition of the expanding population, or the final equilibrium. In this paper, we focus our study on this mathematical object and detail below the vocabulary we use. Key to our analysis are the notions of monostable or bistable systems, and whether the traveling wave is pulled or pushed. There may be confusion around these concepts in the literature, so we clarify their definitions below.

\subsubsection*{Numerical simulations}

We complement our mathematical analysis with numerical simulations of the Cauchy problem, with initial conditions for each genotype specified as in Figure \ref{fig:CI}. The outcomes of the simulations are heatmaps of the expansion speed over a wide range of parameters.

Initial conditions for numerical simulations are as follows: the left half of the domain is full of drive ($\n{DD} = 1$), and the right half is full of wild-type  ($\n{WW} = 1$) (see Figure \ref{fig:CI}). 

\begin{figure}[H]
        \centering
     \includegraphics[scale = 0.7]{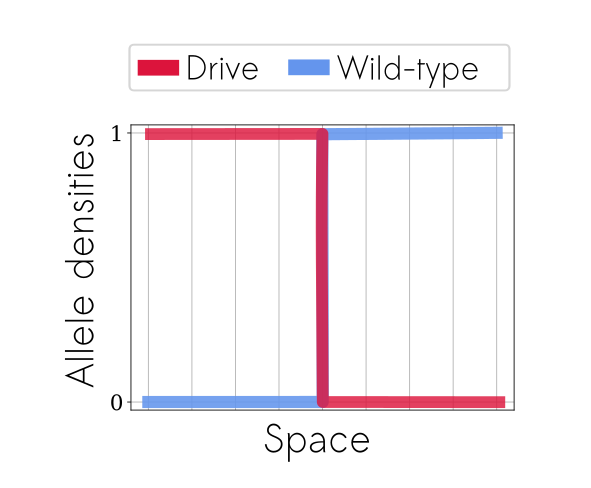}
        \caption{Initial conditions for numerical simulations. The left half of the domain is full of drive ($\n{DD} = n_{_{D}} = 1$), and the right half is full of wild-type  ($\n{WW} = n_{_{W}} = 1$).}
        \label{fig:CI}
\end{figure}

The code is available on GitHub (\url{https://github.com/LenaKlay/gd_project_1}). We ran our simulations in Python 3.6, with the Spyder environment. Heavy heatmaps \ref{fig:heatmap_zyg_coex}, \ref{fig:heatmap_zyg_bist}, \ref{fig:heatmap_ger_coex}, \ref{fig:heatmap_ger_bist}; \ref{fig:continuity} have been computed thanks to the INRAE Migale bioinformatics facility (doi: 10.15454/1.5572390655343293E12). We are grateful to them for providing these computing resources.

\subsection{Glossary}


\subsubsection*{Allelic densities and frequencies}

For our analysis, it is convenient to introduce the allelic (half-) densities $(\n{D}, \n{W})$. The precise definition depends on the model, and more specifically on the timing of conversion. In fact, we have $\n{D} = \n{DD} + \alpha \  \n{DW}$ and $\n{W} = \n{WW} + (1-\alpha) \ \n{DW}$, with $\alpha = \frac12$ when conversion occurs in the zygote, and $\alpha = \frac{1+c}2$ when conversion occurs in the germline (see section \ref{subsec:partial}). Depending on the regime of parameters, it may be more appropriate to study the allelic frequencies $\p{D} = \frac{\n{D}}{\n{D} + \n{W}}$,  $\p{W} = \frac{\n{W}}{\n{D} + \n{W}}$. 

\subsubsection*{Classification of the dynamics}

It can happen that the dynamics lead to the decay of the drive allele uniformly in space. In this case, there cannot exist a traveling wave for the drive population: we use the term \textit{gene drive clearance} to describe this situation. Then, the  problem boils down to the standard Fisher-KPP traveling wave problem for the expansion of the wild-type in the absence of a drive (see \cite{girardin2021}).




When traveling waves do exist, we distinguish between two cases depending on the sign of the speed. When $v > 0 $, the wave moves to the right: it is a \textit{drive invasion}.  When $v < 0 $, the wave moves to the left: it is a \textit{wild-type invasion}.
In some specific cases, drive and wild-type invasions can happen simultaneously: the waves decompose into two sub-traveling wave solutions over half of the domain. They move in opposite directions and lead to the  coexistence of both alleles in-between. 

In case of drive invasion, we distinguish several cases depending on the state of the population in the wake of the front(s): i) {\em eradication} drives are those for which the population vanishes in the wake of the front(s); ii) {\em suppression} drives are those for which population persists in the wake of the front(s). In the latter case, two scenarios are possible: persistence of drive homozygotes only; persistence of all genotypes.   




\subsubsection*{Monostable / Bistable systems}

To illustrate useful concepts in the theory of reaction-diffusion equations, we consider the following standard equation of population genetics \cite{otto2011} describing the dynamics of the frequency $p$ of an allele of interest: 
\begin{equation}\label{eq:standard_genetics}
    \partial_t p - \partial_{xx}^2 p = p \ (1-p) \ \sigma(p)  \quad \quad \text{with} \ p \in [0,1],
\end{equation}
where $\sigma(p)$ is the selection term, which we consider frequency-dependent (i.e., function of $p$). 

If $\sigma$ is of constant sign, say $\sigma>0$, this equation is referred to as a monostable case. Then, the solution converges locally to the unique stable equilibrium $p=1$ (or $p=0$ if $\sigma<0$). If $\sigma$ is changing signs once in $(0,1)$, being negative below some threshold, and positive above, it is referred to as a bistable case. In the latter case, the solution converges locally to one of the two stable equilibria $p=0$ or $p=1$, depending on the initial condition. Moreover, each equilibria has a basin of attraction and there is a threshold effect -- hence the name ``threshold-dependent drives'' in the gene drive literature to describe this kind of case (for example in reference \cite{tanaka2017}). 

In both cases, there exist traveling waves connecting the two equilibria $p=0$ and $p=1$. A straightforward integration by parts shows that, whatever the stability, the sign of the wave speed satisfies
\begin{equation}\label{eq:integral}
    \operatorname{sign}(v)=\operatorname{sign}\left(\int_0^1 p(1-p)\sigma(p) dp \right).
\end{equation}
In monostable cases with $\sigma>0$, this sign is positive; in bistable cases, however, it depends on the details of the frequency-dependence $\sigma$. 
Moreover, under some circumstances (bistable case, or degenerate monostable case), the invasion outcome for the Cauchy problem can be changed by modifying the inoculum size. Even if traveling waves exist such that $p=1$ is invading $p=0$, small initial conditions may not succeed in propagating in space, see the discussion in  \cite{tanaka2017,turelli2017,nadin2018}.

\lk{By analogy with the scalar case, we consider that a system is monostable if it has exactly one stationary stable state, and bistable if it has exactly two stationary stable states.}

\subsubsection*{Pulled and pushed waves}\label{pulled_pushed} 
Usually, a wave is said to be pulled if the wave speed coincides with the minimal speed of the linearized problem at low density (resp.\ low frequency). This occurs when the population at low density (resp.\ low frequency) has sufficient reproductive success to determine the dynamics of the full invasion. 


Conversely, a wave is said to be pushed if the wave speed is strictly larger than the minimal speed of the linearized problem. In contrast with pulled waves, the whole population contributes to the dynamics of invasion.


A bistable wave is clearly pushed \cite{hadeler1975}. However, a monostable wave can be either pulled or pushed, see \cite{an2022, holzer2022, birzu2018} and discussion therein. Nonetheless, a monostable wave is necessarily pulled if the per-capita growth rate is maximal at low density (resp. low frequency). In the particular case of the scalar problem \eqref{eq:standard_genetics}, this criterion simply writes:
\begin{equation}\label{eq:pull_cri}
     \sigma(0) \geq (1-p) \ \sigma(p) \quad\quad(\forall p\in [0,1]).
\end{equation}

\newpage

\section{Results}\label{sec:results}

In part \ref{subsec:perfect}, we study the model with perfect conversion in the zygote (\ref{eq:per_zyg}) and compare the qualitative behavior of the solution when $r = 0$ and $r = + \infty$. In part \ref{subsec:partial}, we proceed the same way on models with partial conversion (\ref{eq:par_zyg}) and (\ref{eq:par_ger}), obtaining more general results.

\subsection{Model with perfect conversion in the zygote}\label{subsec:perfect}

\subsubsection{Preliminary statements on the model}\label{subsubsubsec:preliminary_perfect_zygote}


We introduce a few general results on model (\ref{eq:per_zyg}) when $r>0$, which will be useful in the study.\\

When $s \leq \frac{1}{2}$, system (\ref{eq:per_zyg_p}) is monostable: the only stable state is $(\n{DD}= \n{DD}^* , \n{WW}=0)$ with $ \n{DD}^* = \min(0,1 - \frac{s}{r (1-s)})$ \cite{girardin2021}, leading to a drive invasion if any. We introduce the minimal speed of problem \eqref{eq:per_zyg} linearized at low drive density, i.e. the speed of any pulled wave in case of a drive invasion: 

\begin{equation}\label{eq:lin_speed_zyg_per_drive}
    2 \sqrt{ \partial_{\n{DD}} F_D(0, 1) } = 2 \sqrt{1-2s}.
\end{equation}


When $s>\frac12$, system (\ref{eq:per_zyg_p}) is bistable. Consequently traveling waves are either semi-trivial ($\n{DD}=0$ identically, standard Fisher-KPP problem for $\n{WW}$) or pushed.\\

For our analysis, it will be convenient to rewrite model (\ref{eq:per_zyg}) so that it follows the frequency of the drive $\p{D} = \frac{\n{D}}{\n{D} + \n{W}} = \frac{\n{DD}}{\n{WW}+\n{DD}}$ (because $\n{DW}=0$) and total population density $n = \n{WW}+\n{DD} $ (details in \ref{an:rewrite1}): \begin{equation} \label{eq:per_zyg_p}
\left\{
    \begin{array}{ll}
       \partial_t \p{D} - \partial_{xx}^2 \p{D} &=  2 \ \partial_x ( \log n) \ \partial_x \p{D} + \left( r \ (1-n)+1 \right) \ s \ \p{D} \  (1-\p{D}) \ \left(\p{D}- \dfrac{2s -1}{s} \right),
       \\
   \partial_t n - \partial_{xx}^2 n  &=  \left( r \ (1-n)+1 \right) \left( 1-s+s(1-\p{D})^2 \right) \ n   - n.
    \end{array}
\right.
\end{equation} 

System~\eqref{eq:per_zyg_p} differs from standard equations often used in populations genetics as it contains an advection term $2 \ \partial_x ( \log n) \ \partial_x \p{D}$. \lk{This term appears when calculating $ \partial_{xx}^2 \p{D} =   \partial_{xx}^2 \dfrac{\n{DD}}{n}$ (details in \ref{an:rewrite1}) and} represents a demographic flux from denser to less dense areas, due to variations in population density. It is opposed to the spread of the (costly) drive allele (see Figure 2 \cite{girardin2021}). \lk{We observe a singularity for $n=0$ in both formulations of the system: in \eqref{eq:per_zyg} due to $\frac{1}{\n{DD}+\n{WW}}$ and in \eqref{eq:per_zyg_p} due to $\log(n)$. This should be handle with care.}

\subsubsection{$r = + \infty$}\label{subsubsec:perfect_rinf}

The limit of system \eqref{eq:per_zyg_p} when $r \rightarrow + \infty$ has already been determined in \cite{girardin2021}. Using the Strugarek-Vauchelet rescaling \cite{strugarek2016}, the following limit equation is obtained, which was also previously introduced in \cite{tanaka2017}: 

\begin{equation} \label{eq:per_zyg_rinf}
\partial_t \p{D} - \partial_{xx}^2 \p{D} =  \dfrac{s \ \p{D} \  (1-\p{D}) \ \left(\p{D}- \dfrac{2s -1}{s} \right)}{1-s+s(1-\p{D})^2}.
\end{equation}

Interestingly, equation \eqref{eq:per_zyg_rinf} is independent of the population density $n$ and it does not contain the advection term $2 \ \partial_x ( \log n) \ \partial_x \p{D}$. This is due to the fact that the population size $n(t,x)$ remains spatially homogeneous after the introduction of drive individuals, when $r\to + \infty$. Intuitively, so many offspring are produced at each generation that the carrying capacity is instantaneously restored, and losing a fraction $s$ of these offspring by selection has no consequence. Therefore the variations in population density ($n$), and consequently the demographic flux, are negligible.


Equation \eqref{eq:per_zyg_rinf} has a single parameter, the fitness cost of the drive $s$. The numerical value of the threshold for the transition from positive to negative speed ($\approx 0.70$) was already known \cite{tanaka2017,girardin2021}, and can be computed to arbitrary precision by the formula \eqref{eq:integral}. The numerical value of the threshold for the transition from pulled to pushed ($\approx 0.35$ up to two digits) was \lk{numerically} computed by a continuation method following \cite{avery2022,holzer2022}.




%
%
%

\vspace{-0.2cm}
\begin{table}[H]
\centering 
\begin{tabular}{|C{1.4cm}|C{3.4cm}|C{2.3cm}|C{2.8cm}|C{3.5cm}|}
    \hline
   \textbf{\textit{s} value} &$0<s \lesssim 0.35$ & $0.35 \lesssim s <1/2$ & $1/2<s \lesssim 0.70$ & $0.70 \lesssim s<1$ \\
    \hline
     \textbf{Stability} & \multicolumn{2}{c|}{Monostable} &   \multicolumn{2}{c|}{Bistable} \\
     \hline
   \textbf{Speed} & $v= 2 \sqrt{1-2s}$ & $v > 2 \sqrt{1-2s}$ & $v > 0$ & $v < 0$ \\
    \hline
  \textbf{Wave} &  Pulled wave &   \multicolumn{2}{c|}{Pushed wave} & Pushed wave \\
     \hline
   & \multicolumn{3}{c|}{Drive invasion}  & Wild-type invasion \\
  \vspace{-1.7cm} \textbf{Invasion}  & \multicolumn{3}{c|}{\includegraphics[scale = 0.5]{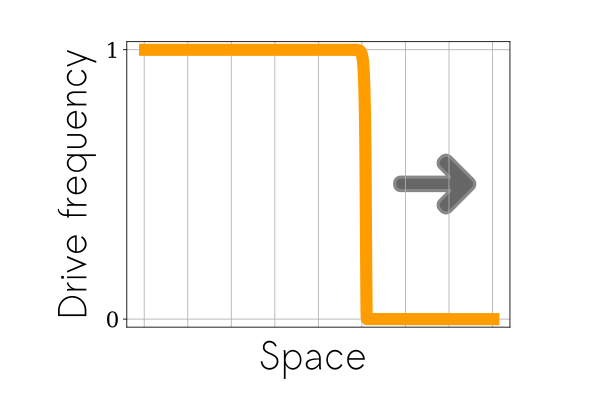}}  & \includegraphics[scale = 0.5]{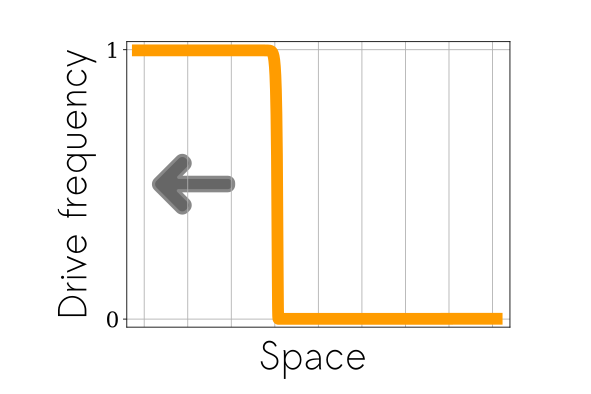} \\
   \hline
\end{tabular} 
\caption{Traveling waves study for Model \eqref{eq:per_zyg_rinf}, \lk{limit of system \eqref{eq:per_zyg} when $r = +\infty$}. All statements in the table are proved in Appendix \ref{ann:tab_perfect}.}
\label{tab:per_zyg_rinf}
\end{table}

Note that equation \eqref{eq:per_zyg_rinf} shows qualitative similarities with a common equation in the population genetics literature \cite{barton1979} (and which is actually an approximation of \eqref{eq:per_zyg_rinf} under weak selection, i.e. $s \to 0$): 


\begin{equation}
\partial_t \p{D} - \partial_{xx}^2 \p{D} =  s \ \p{D} \  (1-\p{D}) \ \left(\p{D}- \dfrac{2s -1}{s} \right).
\end{equation}

Quantitatively, the thresholds are $\frac{2}{5}$ instead of $0.35$, and $\frac{2}{3}$ instead of $0.70$ (analytical values) \cite{hadeler1975}.

\subsubsection{$r = 0$}\label{subsubsec:perfect_rzero}




When the intrinsic growth rate $r$ is finite, it is expected that the final population density after the invasion of the drive (if any) is strictly below $1$, because of the fitness cost. The smaller $r$, the lower the final size. The spatial effect of demography on gene drive expansion is expected to be maximal as $r$ vanishes, when the population can hardly restore its carrying capacity, leading to a high amplitude of the population size gradient $2 \ \partial_x ( \log n)$. In this section we focus on the limit $r = 0$, which maximizes the demographic impact of the fitness cost on drive propagation.

In a purely wild-type population, the case $r=0$ corresponds to a number of births balancing exactly the number of deaths. As soon as the drive allele is introduced, this balance is locally broken, yielding a net decrease in the population size. Then, the drive can either propagate by leaving empty space behind, or disappear. The same conclusion holds as long as $r<\frac{s}{1-s}$, see Section \ref{subsubsubsec:preliminary_perfect_zygote}.

We checked numerically that the wave speed is continuous in the limit $r\to 0$. Therefore, each conclusion on the case $r=0$ sheds some light on the case of small $r$ (see heatmap in Appendix \ref{ann:continuity}). 

As discussed above, we cannot just consider a single equation on the drive frequency $\p{D}$ when $r$ is finite because of the demographic contribution $2 \partial_x \log{n}$. Interestingly, in the case $r=0$, the demographic system \eqref{eq:per_zyg} reduces to the following pair of equations:



\begin{equation}\label{eq:per_zyg_r0}
\left\{
    \begin{array}{ll}
       \partial_t \n{WW} - \partial_{xx}^2 \n{WW} & = \  \dfrac{- \ \n{WW} \n{DD}}{\n{WW}+\n{DD}}, \\
       \\
       \partial_t \n{DD} - \partial_{xx}^2 \n{DD} & = \ (1-s)  \ \dfrac{\n{WW}  \n{DD}}{\n{WW}+\n{DD}} - s \ \n{DD}.
    \end{array}
\right.
\end{equation}\\

Noticeably, the previous system shares some features with density-dependent epidemiological SI models. In particular, the dynamics of $\n{WW}$ is always decreasing. The dynamics of $\n{DD}$ is the balance of creation and linear decay. By changing notations $\n{WW}\leftrightarrow S$ (susceptible individuals), and $\n{DD}\leftrightarrow I$ (infected individuals), \eqref{eq:per_zyg_r0} can be recast as follows:
    
\begin{equation}\label{eq:si}
\left\{
    \begin{array}{ll}
        \partial_t S - \partial_{xx}^2 S & =  \ - \beta_1 \ \dfrac{\ S \ I}{S + I},\\
        \\
       \partial_t I - \partial_{xx}^2 I & = \ \beta_2 \ \dfrac{ S \ I}{S + I} - \gamma I.\\
    \end{array}
\right.
\end{equation}
with $\beta_1 = 1$, $\beta_2 = (1-s)$ (transmission parameters), and $\gamma = s$ (disease clearance). Usually, in SI models, individuals of type $S$ are all transformed into individuals of type $I$ at infection, hence $\beta_1 = \beta_2$. In our case, these two rates are distinct because of the fitness cost of the drive. The existence of traveling waves for model \eqref{eq:per_zyg_r0} with $\beta_1 = \beta_2$ has been studied recently in the literature \cite{zhou2019}. Here, we extend the results of reference \cite{zhou2019} to a more general case $0<\beta_1$ and $0<\beta_2$. This leads to the characterisation in Table \ref{tab:per_zyg_r0} and Appendix \ref{ann:exist}.

\renewcommand{\arraystretch}{1.4}

\vspace{-0.2cm}
\begin{table}[H]
\centering
\begin{tabular}{|C{1.4cm}|C{6.5cm}|C{6.5cm}|}
    \hline
   \textbf{\textit{s} value} & $0<s<1/2$ & $1/2<s<1$ \\
    \hline
   \textbf{Stability} & Monostable & Degenerate case \\
    \hline
    \textbf{Speed} & $v =  2 \sqrt{1-2s} $  &  \multirow{2}{*}{No wave} \\
    \cline{1-2}
    \textbf{Wave} & Pulled wave & \\
   \hline
  &  Drive invasion  & Gene drive clearance \\
     \vspace{-1.7cm} \textbf{Invasion} &  \includegraphics[scale = 0.5]{images/2_maths/drive_invasion_prop.png} & \includegraphics[scale = 0.5]{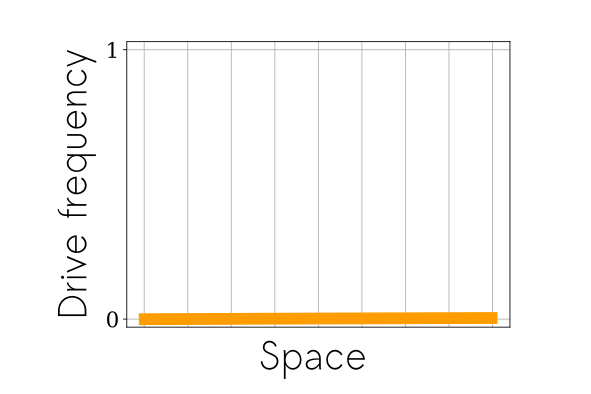}\\
   \hline
\end{tabular} 
\caption{Traveling waves study for Model \eqref{eq:per_zyg_r0}, \lk{limit of system \eqref{eq:per_zyg} when $r=0$}. All statements in the table are proved in Appendix \ref{ann:tab_perfect}.}
\label{tab:per_zyg_r0}
\end{table}

In contrast to the results obtained when $r = +\infty$, when $r = 0$ there is only one threshold value of $s$ determining the outcome of the model (Table \ref{tab:per_zyg_r0}). When $ 0 < s < 1/2$, the system is monostable, the drive necessarily invades. Moreover, the wave is pulled and travels at speed $v =  2 \sqrt{1-2s}$ \eqref{eq:lin_speed_zyg_per_drive}. When $1/2 < s < 1$, the problem is degenerate: there exists a family of steady states, corresponding to homogeneous $\n{WW} \in [0,1]$ and $\n{DD} = 0$. It is a case of gene drive clearance, as $\n{DD}$ converges to zero uniformly in space (at rate at least $1-2s$). However, the final density of wild-type is not clearly determined, as it boils down to diffusion only in the large time asymptotics (details in Appendix \ref{ann:gd_clearance}). Note that this conclusion holds in a well-mixed population (without spatial consideration): the drive decays uniformly and the final density of wild type depends on the initial data. 


\subsubsection{Comparison between the outcomes when $r = +\infty$ and $r = 0$}

The differences between the two regimes are strongest for intermediate values of $s$. When $1/2 < s \lesssim 0.70$, the drive can spread when the demographic consequences are negligible ($r = +\infty$). However, such a costly drive cannot invade when the intrinsic growth rate $r$ is very low ($r = 0$). When $0.35 \lesssim s < 1/2$, the drive wave advances for both $r = +\infty$ and $r = 0$. However, it is of different nature: the wave is pulled when $r = 0$, while it is pushed when $r = \infty$.  

By providing analytical results for $r = 0$, our study is complementary to \cite{girardin2021}, where the invasion outcome was described numerically in \cite[Figure 3.A]{girardin2021} together with a series of analytical estimates of the sign of the speed.

\newpage

\subsection{Models with partial conversion} \label{subsec:partial}

In this section, we extend the study to models \eqref{eq:par_zyg} and \eqref{eq:par_ger} with partial conversion. The models are reformulated in terms of allelic densities ($\n{D}$ and $\n{W}$) rather than genotype densities ($\n{DD}$, $\n{DW}$ and $\n{WW}$). This reformulation enables reducing the number of equations from three to two equations. Even if we are not able to determine the genotypic composition of the population (which individual genotype a gamete comes from, either homozygote or heterozygote), the spreading properties are equivalent. Interestingly, the same roadmap as in the full conversion model can be followed. 
Again, we focus on the two extreme regimes $r = + \infty $ and $r = 0$.





\subsubsection{Conversion occurring in the zygote} \label{subsubsec:partial_zyg}

When conversion occurs in the zygote, we can deduce the following system from model~\eqref{eq:par_zyg}, with $ \n{D} =  \n{DD} + \frac{1}{2} \  \n{DW}  $ and $ \n{W} =  \n{WW} + \frac{1}{2} \  \n{DW}  $: \begin{equation}\label{eq:par_zyg_nD_nW}
    \left\{
    \begin{array}{ll}
      \partial_t \n{D}  - \partial_{xx}^2 \n{D} \ = \  \n{D} \ \left[ \dfrac{ r \ (1-n)+1 }{n} \left[ \ (1-s)  (\n{D} +  2  \  c  \  \n{W} ) +   (1-sh)   \  (1-c)  \    \n{W}  \right]  - 1 \right] = F^z_D(\n{D}, \n{W}), \\
        \\
      \partial_t  \n{W} - \partial_{xx}^2  \n{W} = \   \n{W} \ \left[ \dfrac{ r \ (1-n)+1 }{n} \left[ \     \n{W}   +    (1-sh)  \  (1-c) \ \n{D}  \right]  - 1 \right] = F^z_W(\n{D}, \n{W}).
    \end{array}
\right. 
\end{equation}

The density $ \n{W}$ (resp. $\n{D}$) corresponds to one half of the wild-type (resp. drive) allele density at the time of zygote formation. When conversion happens in the zygote, heterozygous individuals are the result of conversion failures and produce one half of each type of gamete, drive or wild-type. 

\subsubsubsection{Preliminary statements on the model}\label{subsubsubsec:preliminary_partial_zygote}

This model brings more variety in terms of traveling waves than the previous one \eqref{eq:per_zyg}. Cases of monostable wild-type invasion can occur, as well as cases of monostable coexistence. We introduce first all the possible minimal speeds of the problem linearized at low densities and detail later under which parameters they arise.

The minimal speed of the problem linearized at low drive density, i.e. the speed of any pulled monostable wave with a positive speed is given by: \begin{equation} \label{eq:lin_speed_zyg_drive}
    2 \sqrt{\partial_{\n{D}} F^z_D(0,1)} =  2 \sqrt{ 2c \ (1-s) + (1-sh)(1-c)-1} = 2 \sqrt{c ( 1- 2s) - s h (1-c)},
\end{equation}
provided that the quantity is non-negative.
The minimal speed of the problem linearized at low wild-type density depends on the stable steady state of a population only bearing drive alleles (i.e. a population of drive homozygotes):  $\n{D}^* = \n{DD}^* = \min(0, 1 - \frac{s}{r(1-s)})$ \cite{girardin2021}. If $\n{D}^* = 1 - \frac{s}{r(1-s)}$, this minimal speed is given by:\begin{equation}\label{eq:lin_speed_zyg_wt_1}
    -2 \sqrt{\partial_{\n{W}} F^z_W(\n{D}^*, 0)} = - 2 \sqrt{ \frac{(1-sh)(1-c)}{1-s}-1}.
\end{equation}

If $\n{D}^* = 0 $, the minimal speed of the problem linearized at low wild-type density is given by the classical Fisher-KPP formula: \begin{equation}\label{eq:lin_speed_zyg_wt_2}
   - 2 \sqrt{\partial_{\n{W}} F^z_W(0, 0)} = - 2 \sqrt{r}.
\end{equation} 

Note that \eqref{eq:lin_speed_zyg_wt_2} is the only minimal speed depending on parameter $r$: it corresponds to the case of gene drive clearance, the only configuration where the drive allele disappears uniformly in space.\\


For our analysis, it will be convenient to rewrite model (\ref{eq:par_zyg_nD_nW}) so that it follows the frequency of the drive  $\p{D}=\frac{\n{D}}{\n{W} + \n{D}}$ and the total population density  $ n = \n{WW} + \n{DW} + \n{DD}  = \n{W} + \n{D}$  (details in \ref{an:rewrite2}):

\begin{equation}\label{eq:par_zyg_p}
\left\{ \small
    \begin{array}{ll}
      \partial_t n  - \partial_{xx}^2 n \ = & \big( r \ (1-n)+1 \big) \ \Big( (1-s) \ \p{D}^2 + 2 \ \p{D}  \ (1-\p{D}) \ [ c \ (1-s)  + (1-c) \ (1-sh)  ] + (1-\p{D})^2  \Big) n - n , \\
        \\
      \partial_t \p{D}  -  \partial_{xx}^2 \p{D}  \   = &    2 \ \partial_x \log(n) \ \partial_x \p{D}  \\
      & + \big( r \ (1-n)+1 \big) \Big(  \big[ 1 - 2 (1-c) (1-h) \big] \ s \  \p{D} - s [ 1-(1-c)(1-h)] + c (1-s)  \Big)  (1-\p{D})  \p{D} .
    \end{array}
\right. 
\end{equation}

\subsubsubsection{$r = + \infty$} \label{subsubsubsec:par_zyg_rinf}



Similarly as in Section \ref{subsubsec:perfect_rinf}, we compute formally the limiting equation on \eqref{eq:par_zyg_p} when $r \rightarrow +\infty$:\begin{equation}\label{eq:par_zyg_rinf}
    \partial_t \p{D}  -  \partial_{xx}^2 \p{D}  \ = \dfrac{ \Big(  \big[ 1 - 2 (1-c) (1-h) \big] \ s \  \p{D} - s [ 1-(1-c)(1-h)] + c (1-s)   \Big)  \ (1-\p{D}) \ \p{D} }{(1-s) \ \p{D}^2 + 2 \ \p{D}  \ (1-\p{D}) \ [ c \ (1-s)  + (1-c) \ (1-sh)  ] + (1-\p{D})^2}.
\end{equation}

\vspace{0.2cm}

Note that as in Section \ref{subsec:perfect}, this equation does not depend on $n$. We introduce:

\begin{equation}
    \mathscr{A}_z := s \ \big[ 2 (1-c) (1-h) - 1 \big], \quad \quad s_1 := \dfrac{c}{1-h(1-c)}, \quad \quad s_{2,z} := \dfrac{c}{2c + h(1-c)}.
\end{equation}

where $z$ stands for zygote. Note that $ \mathscr{A}_z  > 0 \iff s_1 < s_{2,z}$.\\

We distinguish between two cases, depending on the sign of $ \mathscr{A}_z $. If $\mathscr{A}_z > 0$,  which can only happen if we have both $c<\frac{1}{2}$ and $h<\frac{1}{2}$, the system is always monostable. We observe a drive invasion for $s < s_1$, a coexistence state for $s_1 < s < s_{2,z}$ and a wild-type invasion for $s_{2,z} < s$. Criterion \eqref{eq:pull_cri} is always verified (see Appendix \ref{ann:pull_zyg}), consequently every traveling wave (or sub-traveling wave in case of coexistence) is pulled, moving at speed \eqref{eq:lin_speed_zyg_drive} or \eqref{eq:lin_speed_zyg_wt_1}. These statements are summarized in Table \ref{tab:par_zyg_rinf_coex}. \\

\renewcommand{\arraystretch}{1.4}
 
\vspace{-0.2cm}
\begin{table}[H]
\centering 
\begin{tabular}{|C{1.4cm}|C{4cm}|C{4cm}|C{4cm}|}
    \hline
   \textbf{\textit{s} value} & $s <s_1 $ & $ s_1 < s < s_{2,z}$ & $  s_{2,z} < s $ \\
    \hline
     \textbf{Stability} & \multicolumn{3}{c|}{Monostable}  \\
     \hline
   \textbf{Speed} & $v=v_{lin+}$ & $v=v_{lin+}$ and $v=v_{lin-}$  &  $v=v_{lin-}$ \\
    \hline
  \textbf{Wave} &  \multicolumn{3}{c|}{Pulled Wave}    \\
     \hline
   & Drive invasion & Coexistence & Wild-type invasion \\
  \vspace{-1.7cm} \textbf{Invasion}  & \includegraphics[scale = 0.5]{images/2_maths/drive_invasion_prop.png}  & \includegraphics[scale = 0.5]{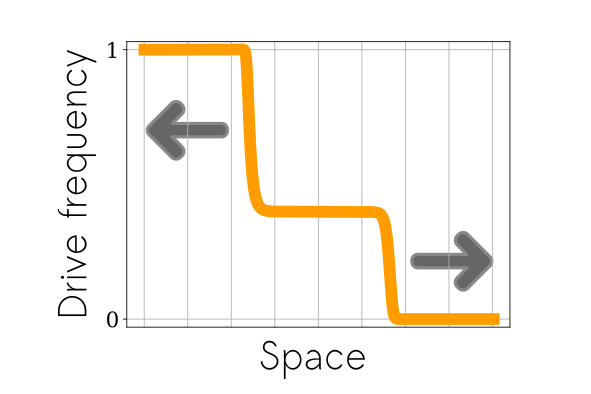}  & \includegraphics[scale = 0.5]{images/2_maths/wild_type_invasion_prop.png} \\
   \hline
\end{tabular} 
\caption{Traveling waves study for Model \eqref{eq:par_zyg_rinf} \lk{(limit of system \eqref{eq:par_zyg_p} when $r=+ \infty$)} when $\mathscr{A}_z > 0$, with $v_{lin+} = 2 \sqrt{c ( 1-2s) - s h (1-c)}$ \eqref{eq:lin_speed_zyg_drive} and $v_{lin-} = - 2 \sqrt{ \frac{(1-sh)(1-c)}{1-s}-1}$ \eqref{eq:lin_speed_zyg_wt_1}. \lk{For all values of $s$, there exists only one stable state (monostability). In particular in the case of coexistence, the stable state (in the center) invades two unstable states (on the right and left)}.}
\label{tab:par_zyg_rinf_coex}
\end{table}

Now we turn to the case $\mathscr{A}_z < 0$. The system is monostable for $s < s_{2,z}$ (drive invasion), bistable for $s_{2,z} < s < s_1$ and monostable for $s_1 < s$ (wild-type invasion). In case of monostable drive invasion, we define a set $\mathscr{S}_z$ of $s$ values: \begin{equation}
    \mathscr{S}_z := \Big\{ s \in (0,1) | \Big( 1 - 2  s [1 - (1-h)(1-c)] \Big) \Big(  (- 2  c - h +  ch) s  + c \Big) +  s \  \big[ 2 (1-c) (1-h) - 1 \big] > 0 \Big\} .
\end{equation}

For all $s$ in $\mathscr{S}_z$ and $\mathscr{A}_z < 0$, criterion \eqref{eq:pull_cri} is verified and consequently, there exists a pulled monostable traveling wave with positive speed (see Appendix \ref{ann:pull_zyg}). Note that such $s$ values are necessarily strictly below $s_{2,z}$, condition for a monostable drive invasion. In case of wild-type invasion, criterion \eqref{eq:pull_cri} is never verified (see Appendix \ref{ann:pull_zyg}). These statements are summarized in Table \ref{tab:par_zyg_rinf_bist}.

\renewcommand{\arraystretch}{1.4}

\vspace{-0.2cm}
\begin{table}[H]
\centering 
\begin{tabular}{|C{1.4cm}|C{3.4cm}|C{2.4cm}|C{2.8cm}|C{3.5cm}|}
    \hline
   \textbf{\textit{s} value} & $s \in \mathscr{S}_z $ & $s \in (0,s_{2,z}) \setminus \mathscr{S}_z  $ & $s_{2,z}<s < s_1$ & $s_1 < s<1$ \\
    \hline
     \textbf{Stability} & \multicolumn{2}{c|}{Monostable} &   Bistable & Monostable \\
     \hline
   \textbf{Speed} & $v=v_{lin+}$ & $v \geq v_{lin+}$ & & $v < 0$ \\
    \hline
  \textbf{Wave} &  Pulled wave &   &  Pushed wave &  \\
     \hline
   & \multicolumn{2}{c|}{Drive invasion} &  & Wild-type invasion \\
  \vspace{-1.7cm} \textbf{Invasion}  & \multicolumn{2}{c|}{\includegraphics[scale = 0.5]{images/2_maths/drive_invasion_prop.png}}  & \vspace{-2.3cm}  Drive or Wild-type invasion & \includegraphics[scale = 0.5]{images/2_maths/wild_type_invasion_prop.png} \\
   \hline
\end{tabular} 
\caption{Traveling waves study for Model \eqref{eq:par_zyg_rinf} \lk{(limit of system \eqref{eq:par_zyg_p} when $r=+ \infty$)} when $\mathscr{A}_z < 0$, with $v_{lin+} = 2 \sqrt{c ( 1- 2s) - s h (1-c)}$ \eqref{eq:lin_speed_zyg_drive}.}
\label{tab:par_zyg_rinf_bist}
\end{table}

 \subsubsubsection{$r = 0$} \label{subsubsubsec:par_zyg_r0}




Using the relation $ n = \n{WW} + \n{DW} + \n{DD}  = \n{W} + \n{D}$, system \eqref{eq:par_zyg_nD_nW} can be rewritten as follows when $r = 0$: 
\begin{equation}\label{eq:par_zyg_r0}
\left\{
    \begin{array}{ll}
      \partial_t \n{D}  - \partial_{xx}^2 \n{D} \ =  \ \Big( c \ (1-s) + s \ (1-c)  \ (1-h) \Big)  \  \dfrac{\n{D} \   \n{W}}{ \n{D} +  \n{W}}   - s \  \n{D}, \\
        \\
      \partial_t  \n{W} - \partial_{xx}^2  \n{W} =  \ -  \Big(  1 - (1-sh) \ (1-c)  \Big) \ \dfrac{\n{D} \   \n{W}}{\n{D} +  \n{W}}.
    \end{array}
\right. 
\end{equation}

We apply the results of Appendix \ref{ann:exist} with $\beta_1 =  1 - (1-sh) \ (1-c) $ and $\beta_2 =  c \ (1-s) + s \ (1-c)  \ (1-h) $. There exists a monostable and pulled drive invasion wave if:
\begin{equation}
    \beta_2 > \gamma \ \ \iff \ \ s <  s_{2,z} = \dfrac{c}{2 c + h(1-c)}  .
\end{equation}

On the other hand when $ \beta_2 < \gamma$, the reaction term of $\n{D}$ in \eqref{eq:par_zyg_r0} is strictly negative. As before, the density $\n{D}$ converges to zero uniformly in space at rate at least $\beta_2 - \gamma$ (gene drive clearance) and the final density of wild-type is not clearly determined: the problem boils down to diffusion only in the large time asymptotics (details in Appendix \ref{ann:gd_clearance}). These statements are summarized in Table \ref{tab:par_zyg_r0}.\\

\renewcommand{\arraystretch}{1.4}

\vspace{-0.2cm}
\begin{table}[H]
\centering
\begin{tabular}{|C{1.4cm}|C{6.5cm}|C{6.5cm}|}
    \hline
   \textbf{\textit{s} value} & $0<s<s_{2,z}$ & $s_{2,z}<s<1$ \\
    \hline
   \textbf{Stability} & Monostable & Degenerate case \\
    \hline
    \textbf{Speed} & $v = v_{lin+} $  &  \multirow{2}{*}{No wave} \\
    \cline{1-2}
    \textbf{Wave} & Pulled wave & \\
   \hline
  &  Drive invasion  & Gene drive clearance \\
     \vspace{-1.7cm} \textbf{Invasion} &  \includegraphics[scale = 0.5]{images/2_maths/drive_invasion_prop.png} & \includegraphics[scale = 0.5]{images/2_maths/gene_drive_clearance_prop.png}\\
   \hline
\end{tabular} 
\caption{Traveling waves study for Model \eqref{eq:par_zyg_r0} \lk{(limit of system \eqref{eq:par_zyg_p} when $r=0$)}, with $v_{lin+} = 2 \sqrt{c ( 1- 2s) - s h (1-c)}$ \eqref{eq:lin_speed_zyg_drive}.}
\label{tab:par_zyg_r0}
\end{table}

Note that, when $\mathscr{A}_z > 0$, a condition for having a pulled wave with positive speed for both $r = 0$ and $r = + \infty$ is $s < s_{2,z}$. When $\mathscr{A}_z < 0$, a condition for having a pulled wave with positive speed for both $r = 0$ and $r = + \infty$ is $ s \in \mathscr{S}_z \subseteq [0, s_{2,z}]$. This suggests that, under those conditions, whatever the value of the demographic parameter $r$ is, the drive invasion wave is always pulled and consequently, travels at a speed which does not depend on $r$ either (speed given by \eqref{eq:lin_speed_zyg_drive}). We verify this intuition numerically (vertical level lines) in the following section.

\subsubsubsection{Numerical illustrations}

$\underline{\mathscr{A}_z > 0}$

In a first example we choose $c=0.25$ and $h=0.1$ such that $\mathscr{A}_z > 0$. The $s$ threshold values are $s_1 \approx 0.27$ and $s_{2,z} \approx 0.43$. As discussed in the previous sections, when $s < s_{2,z}$, all waves are pulled (sub-)traveling waves for $ r = + \infty$ and $ r = 0$. Note that $s > s_1$ is the condition for the existence of pulled (sub-)traveling waves with negative speed only when $ r = + \infty$. 

We show the value of the speed \eqref{eq:lin_speed_zyg_drive} and \eqref{eq:lin_speed_zyg_wt_1} of the pulled waves as a function of $s$ when $ r = + \infty$, with $c=0.25$ and $h=0.1$ (in Figure \ref{fig:lin_speed_ex1}).

\begin{figure}[H]
        \centering
     \includegraphics[scale = 0.4]{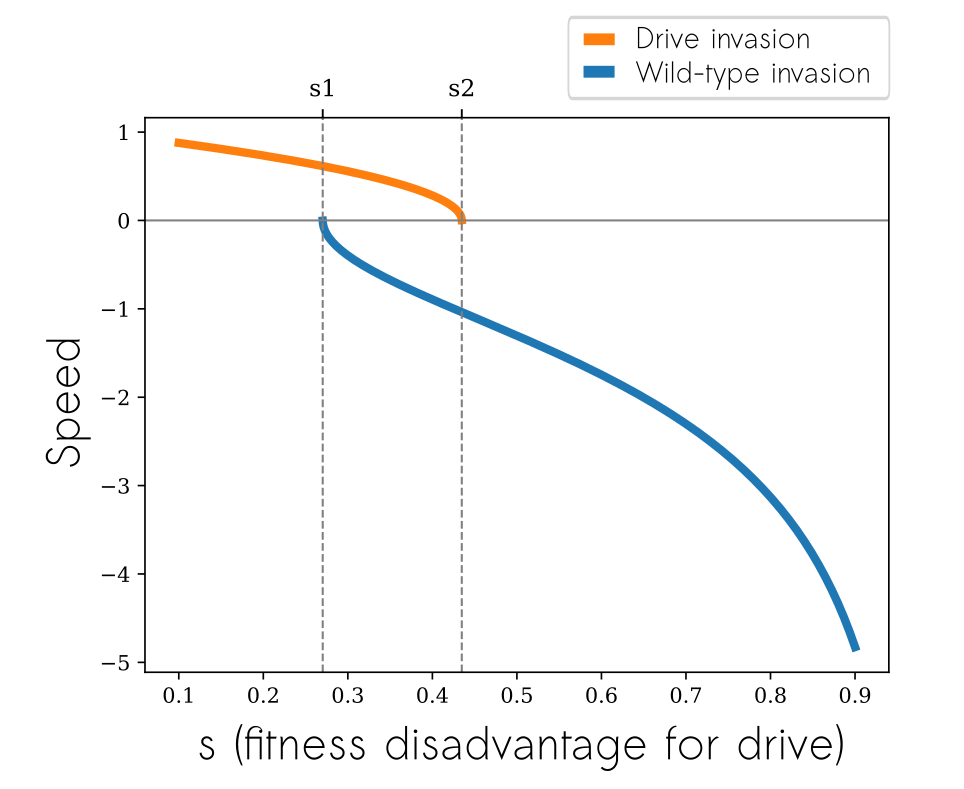}
        \caption{Speed of the drive invasion \eqref{eq:lin_speed_zyg_drive} in orange and speed of the wild-type invasion \eqref{eq:lin_speed_zyg_wt_1} in blue, as a function of $s$ when $ r = + \infty$, with $c=0.25$ and $h=0.1$.}
         \label{fig:lin_speed_ex1}
\end{figure}

Note that for $ 0.27 \approx s_1 < s < s_{2,z} \approx 0.43$, we have both a drive and a wild-type invasion, leading to a stable coexistence state. This case is illustrated in Figure \ref{fig:coex_illu} with $s=0.35$.

\begin{figure}[H]
\centering
\begin{subfigure}{0.328\textwidth}
    \centering
    \includegraphics[scale = 0.7]{images/2_maths/ci.png}
    \caption{$t=0$}
\end{subfigure}
\hfill
\begin{subfigure}{0.328\textwidth}
    \centering
    \includegraphics[scale = 0.7]{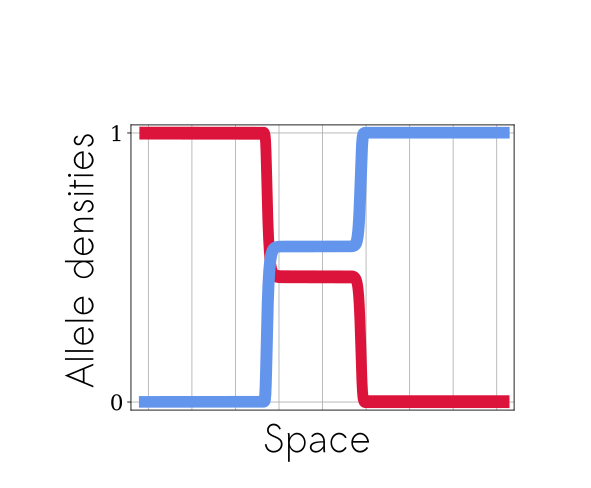}
    \caption{$t=1000$}
\end{subfigure}
\hfill
\begin{subfigure}{0.328\textwidth}
    \centering
    \includegraphics[scale = 0.7]{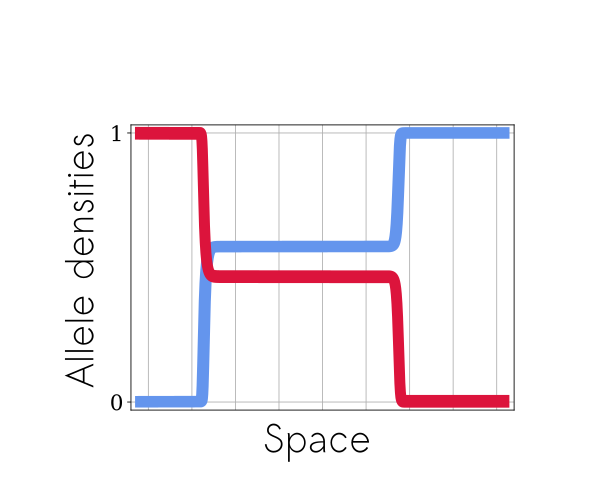}
    \caption{$t=2000$}
\end{subfigure}
\caption{Allele densities as a function of space, at different times, for $s=0.35$, $c=0.25$, $h=0.1$ and $r=3$.}
    \label{fig:coex_illu}
\end{figure}

We now compute numerically the values of the wave speed for intermediate values of $r$ (in Figure \ref{fig:heatmap_zyg_coex}). In case of coexistence, for $ s_1 < s < s_{2,z} $, we choose to show only the positive speed value.


\begin{figure}[H]
     \begin{subfigure}{\textwidth}
    \centering
    \includegraphics[scale = 0.5]{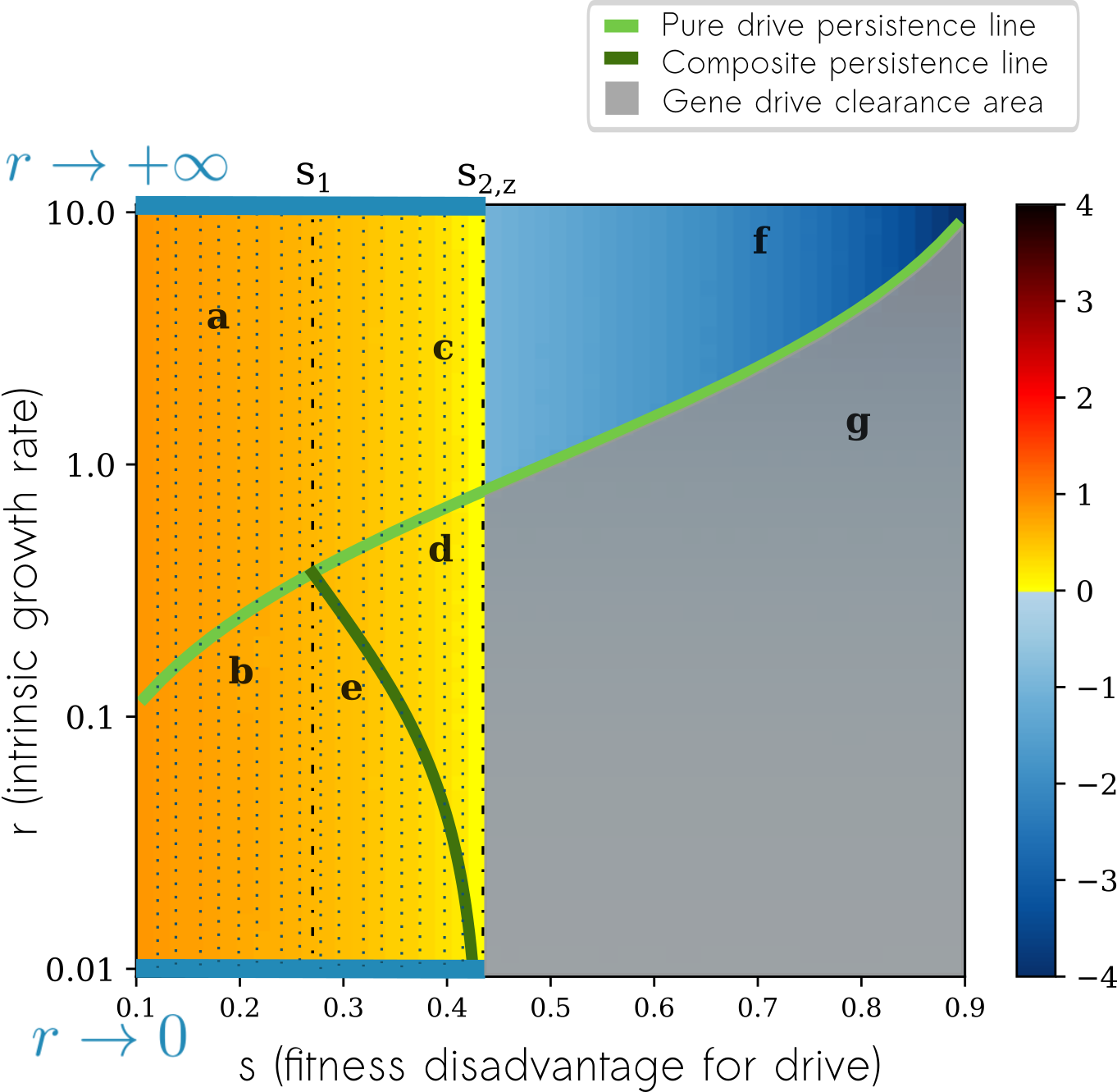}
    \caption{\lk{Heatmap representing the speed of the wave for $c=0.25$, $h=0.1$ when conversion occurs in the zygote. \vspace{0.4cm}}}
\end{subfigure}
\hfill
\begin{subfigure}{\textwidth}
    \centering
    \includegraphics[scale = 0.5]{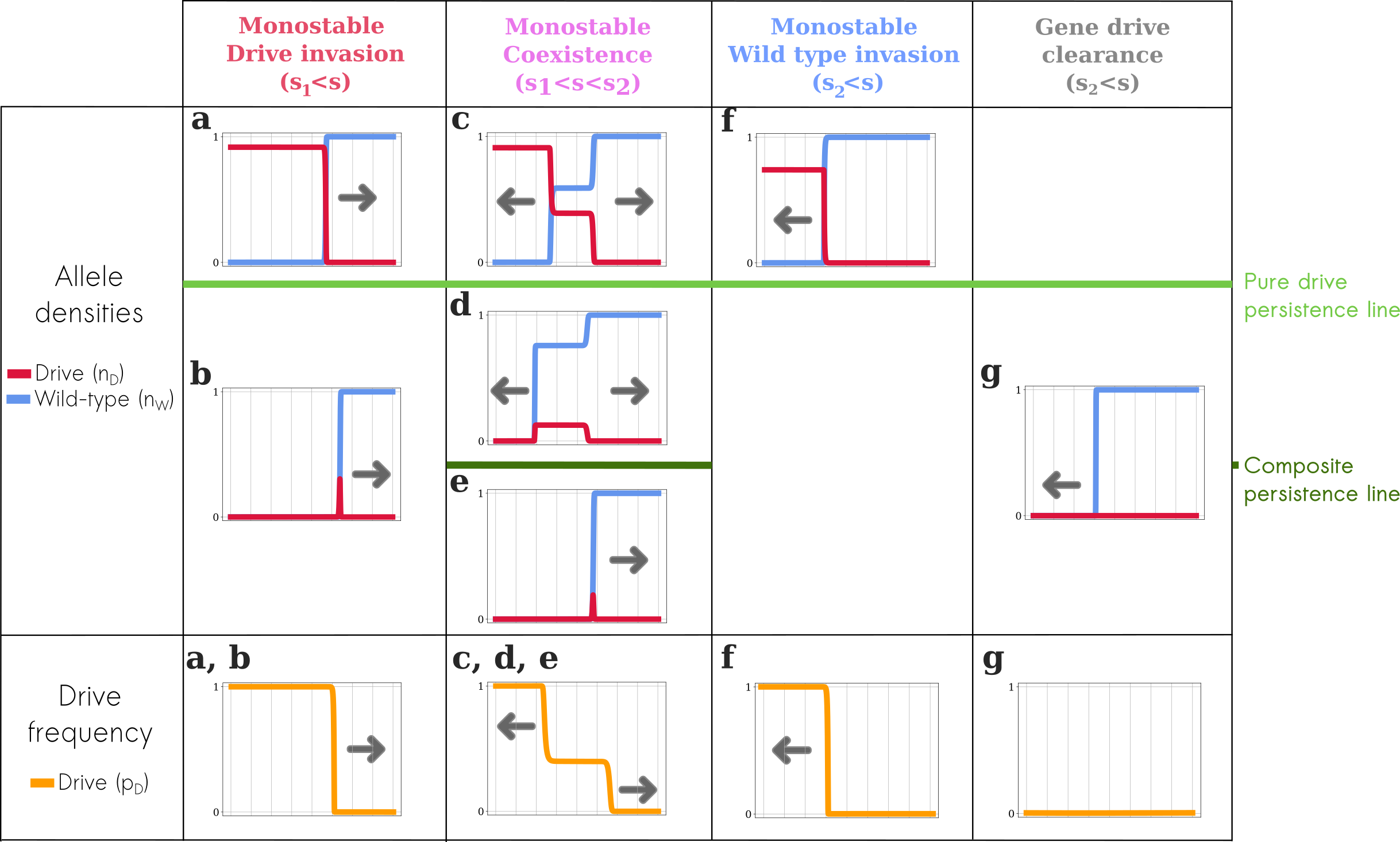}
    \caption{\lk{Illustrations of the shape of the wave (allele densities and drive frequency) for each corresponding case in the heatmap.}}
\end{subfigure}
        \caption{\lk{(A) Heatmap representing the speed of the waves for $c=0.25$, $h=0.1$ when conversion occurs in the zygote.} When the drive invades the population, the speed is positive (in yellow-orange). On the contrary, when the wild-type invades the population, the speed is negative (in blue). When both drive and wild types invade (coexistence), only the speed of the drive is shown in the heatmap, resulting in an apparent discontinuity at $s=s_{2,z}$. As $\mathscr{A}_z > 0$, the system is always monostable for $r = + \infty$: when $ s < s_1$ the drive always invades; when $ s_1 < s < s_{2,z}$ the final state is a coexistence state; when $ s > s_{2,z}$ the wild-type invades or there is gene drive clearance. The turquoise horizontal lines at the bottom and at the top of the heatmap indicate the theoretical values of $s$ such that there exists a pulled wave with positive speed, respectively for $r = + \infty$ and $r = 0$. Below the pure drive persistence line (light green), a well-mixed population containing only drive homozygous individuals will necessarily go extinct. Below the composite persistence line (dark green), it is the whole population that goes extinct (calculations for both lines available in Appendix \ref{ann:persistence}). The gray zone corresponds to the gene drive clearance area. Outside the gray zone, the level lines are apparently vertical, meaning that the wave speed would be independent of $r$. This is in agreement with the fact that the values of the speed coincide when $r = + \infty$ and $r=0$ for $s <s_{2,z}$. If correct, the value of the speed can be found in Figure \ref{fig:lin_speed_ex1}. \lk{(B) Shape of the wave for each case indicated by a letter in the heatmap above. The position of the graphs in the table reflects the position in the heatmap with respect to the persistence lines.}}
        \label{fig:heatmap_zyg_coex}
\end{figure}


For a better understanding of Figure \ref{fig:heatmap_zyg_coex}, we detail the effect of fitness disadvantage $s$ and dominance coefficient $h$ on drive dynamics for $r = + \infty$ and $c=0.25$, without spatial structure, in Appendix \ref{ann:rode_debarre} (in Figure \ref{ann:rode_zyg_c25}).

In Figure \ref{fig:heatmap_zyg_coex}, the speed value for $s<s_{2,z}$ seems not to depend on the demographic parameter $r$: whatever the final equilibrium is, going from population extinction to full replacement of the wild-type genotypes by drive genotypes, the invasion occurs at the same speed. This is in agreement with the fact that the values of the speed coincide when $r = + \infty$ and $r=0$ for $s <s_{2,z}$. If correct, the value of the speed can be found in Figure \ref{fig:lin_speed_ex1}.





\vspace{0.2cm}

$\underline{\mathscr{A}_z < 0}$

In a second example, we choose $c = 0.75$ and $h=0.1$ such that $\mathscr{A}_z < 0$. The $s$ threshold values are $s_1 \approx 0.77 $ and $ s_{2,z} \approx 0.49$. As discussed in the previous section, when $s \in \mathscr{S}_z$ (in our case $ s \lesssim 0.38$), all waves are pulled traveling waves. However, the latter criterion is not a sufficient condition. It is expected that waves are indeed pulled beyond this approximate value of $0.38$. However, this would require to use numerical continuation methods as in Section \ref{subsubsec:perfect_rinf}. We computed numerically the speed values for intermediate values of $r$, as shown in Figure \ref{fig:heatmap_zyg_bist}. We believe that the wave speed is independent of the demographic parameter $r$ when the wave is pulled (visual observation for $s \lesssim 0.38$).

\begin{figure}[H]
    \centering
    \begin{subfigure}{\textwidth}
    \centering
    \includegraphics[scale = 0.5]{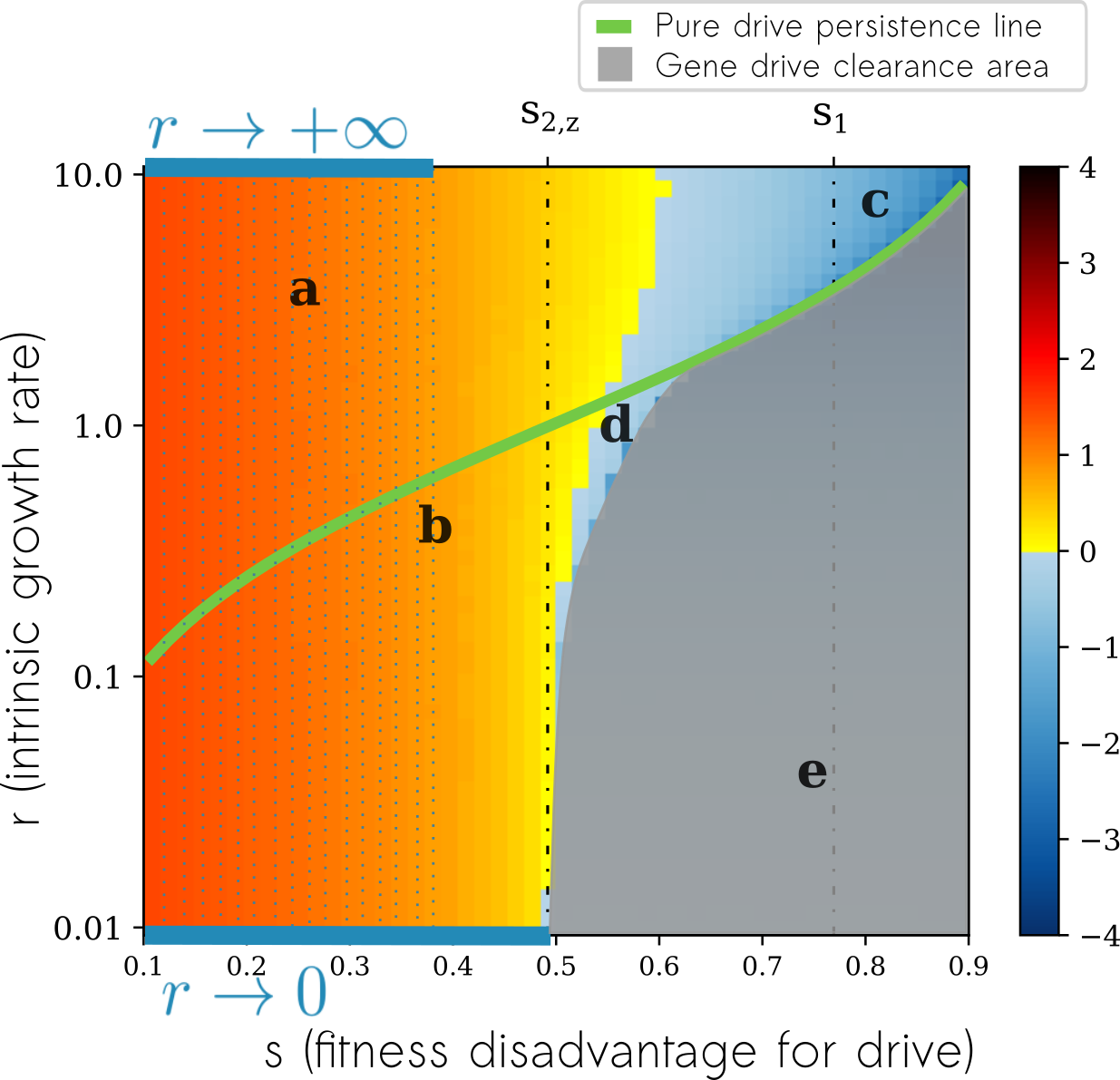}
    \caption{\lk{Heatmap representing the speed of the wave for $c=0.75$, $h=0.1$ when conversion occurs in the zygote. \vspace{0.4cm}}}
\end{subfigure}
\hfill
\begin{subfigure}{\textwidth}
    \centering
    \includegraphics[scale = 0.5]{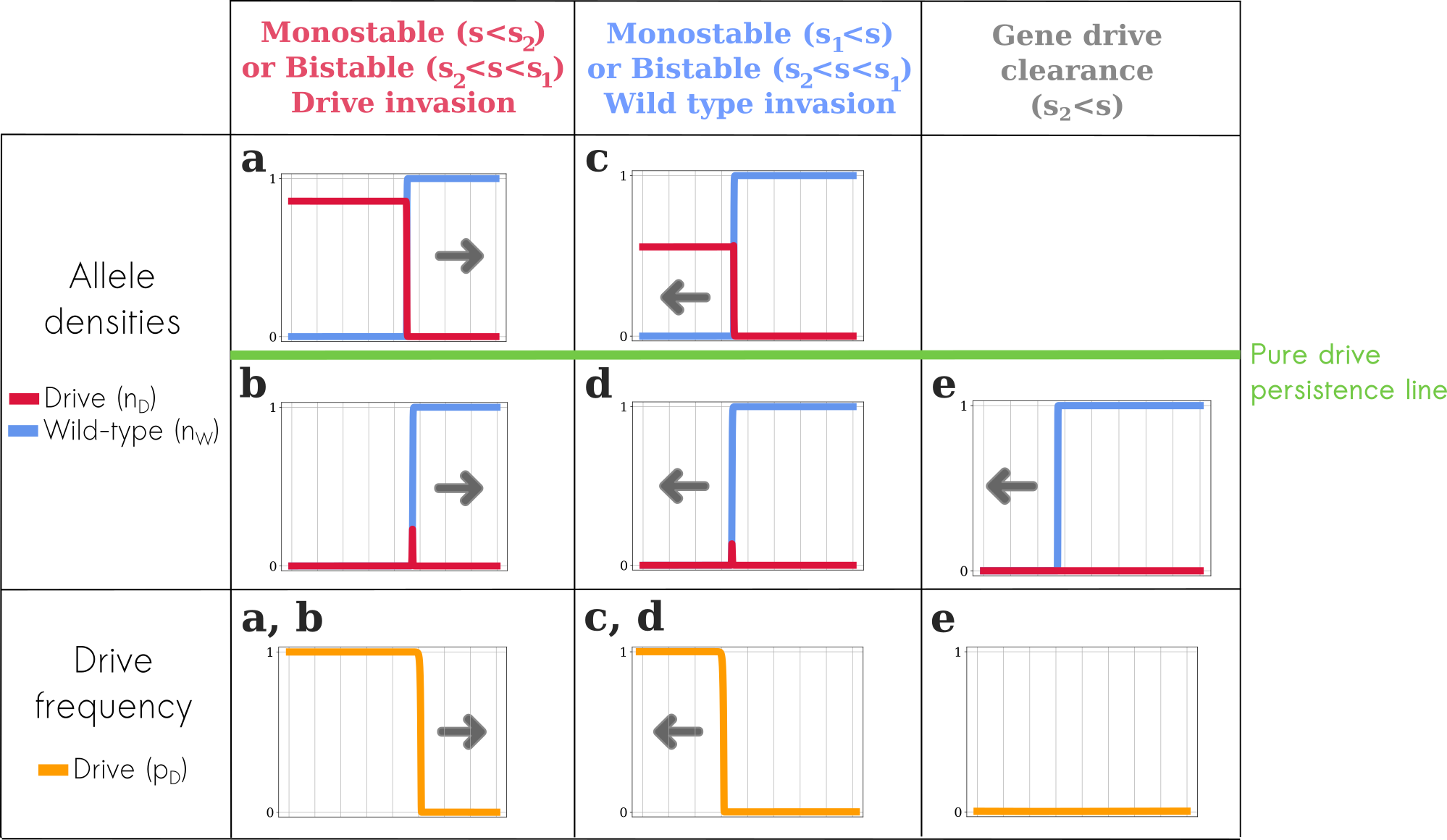}
    \caption{\lk{Illustrations of the shape of the wave (allele densities and drive frequency) for each corresponding case in the heatmap.}}
\end{subfigure}
     \caption{\lk{(A) Heatmap representing the speed of the wave for $c=0.75$, $h=0.1$ when conversion occurs in the zygote.} When the drive invades the population, the speed is positive (in yellow-orange-red). On the contrary, when the wild-type invades the population, the speed is negative (in blue). We have $\mathscr{A}_z<0$, therefore when $r = + \infty$: when $ s < s_{2,z}$ the system is monostable and the drive always invades; when $ s_{2,z} < s < s_1$ the system is bistable and the final state depends on the initial condition; when $ s > s_1$ the system is monostable and the wild-type invades or there is gene drive clearance. The turquoise horizontal lines at the bottom and at the top of the heatmap indicate the theoretical values of $s$ such that there exists a pulled wave with positive speed, respectively for $r = + \infty$ and $r = 0$. Below the pure drive persistence line (light green), a well-mixed population containing only drive homozygous individuals will necessarily go extinct (calculations for this line available in Appendix \ref{ann:persistence}). For $ s \in \mathscr{S}_z$, i.e. $ s \lesssim 0.38$, the level lines are apparently vertical: this is in agreement with the fact that the values of the speed coincide when $r= +\infty$ and $r=0$ in this area. \lk{(B) Shape of the wave for each case indicated by a letter in the heatmap above. The position of the graphs in the table reflects the position in the heatmap with respect to the pure drive persistence line.}}
     \label{fig:heatmap_zyg_bist}
\end{figure}
 
For a better understanding of Figure \ref{fig:heatmap_zyg_bist}, we detail the effect of fitness disadvantage $s$ and dominance coefficient $h$ on drive dynamics for $r = + \infty$ and $c=0.75$, without spatial structure, in Appendix \ref{ann:rode_debarre} (in Figure \ref{ann:rode_zyg_c75}).

\newpage

\subsubsection{Conversion occurring in the germline}\label{subsubsec:partial_ger}

When conversion occurs in the germline, we can deduce the following system from model~\eqref{eq:par_ger}, with $ \n{D} = \n{DD} + (1+c) \ \frac{1}{2} \  \n{DW} $ and $  \n{W} = \n{WW} + (1-c) \  \frac{1}{2}  \ \n{DW}$:

\begin{equation}\label{eq:par_ger_nD_nW}
   \left\{
    \begin{array}{ll}
      \partial_t \n{D}  - \partial_{xx}^2 \n{D} \ = \  \n{D} \ \Big[ \dfrac{ r \ (1-n)+1 }{n} \Big[ (1-s) \n{D} +   (1-sh) \  (1+c) \    \n{W}   \Big]  - 1 \Big] = F^g_D(\n{D}, \n{W}), \\
        \\
      \partial_t  \n{W}  - \partial_{xx}^2  \n{W} = \  \n{W} \ \Big[ \dfrac{ r \ (1-n)+1 }{n}  \Big[ \  \n{W}   +   (1-sh)  \  (1-c) \  \n{D}  \Big]  - 1 \Big] = F^g_W(\n{D}, \n{W}).
    \end{array}
\right. 
\end{equation}

The density $ \n{W}$ (resp. $\n{D}$) corresponds to one half of the wild-type (resp. drive) allele density at the time of zygote formation. When conversion happens in the germline, heterozygous individuals undergo a conversion of their wild-type alleles with probability $c$, and produce a fraction $(1+c)/2$ of drive-carrying gametes.

\subsubsubsection{Preliminary statements on the model}\label{subsubsubsec:preliminary_partial_germline}

As before, we detail the minimal speed of the problem linearized at low densities, for both drive and wild-type alleles.

In case of drive invasion, the minimal speed of the problem linearized at low drive density, i.e. the speed of any pulled monostable wave with positive speed is given by: \begin{equation}\label{eq:lin_speed_ger_drive}
    2 \sqrt{\partial_{\n{D}} F^g_D(0,1)} =  2  \sqrt{ (1-sh)(1+c) - 1 }.
\end{equation}

Note that $ F^g_W(\n{D}, \n{W}) = F^z_W(\n{D}, \n{W})$: in case of a wild-type invasion, the minimal speeds are already given by \eqref{eq:lin_speed_zyg_wt_1} and \eqref{eq:lin_speed_zyg_wt_2} (Section \ref{subsubsec:partial_zyg}).

For our analysis, it will be convenient to rewrite model \eqref{eq:par_ger_nD_nW} so that it follows the frequency of the drive $ \p{D}=\frac{\n{D}}{ \n{W}+\n{D}}$ and the total population density $ n = \n{WW} + \n{DW} + \n{DD} = \n{W} + \n{D}$ (details in Appendix \ref{an:rewrite3}): 

\begin{equation}\label{eq:par_ger_p}
    \left\{
    \begin{array}{ll}
      \partial_t n  - \partial_{xx}^2 n \  &=  \big( r \ (1-n)+1 \big) \ \Big( (1-s) \  \p{D}^2 + 2 \ (1-sh) \   \p{D} \ (1- \p{D}) + (1- \p{D})^2 \Big) n - n, \\
        \\
      \partial_t  \p{D}  - \partial_{xx}^2  \p{D}  \  &=   2 \ \partial_x \log(n) \ \partial_x  \p{D} +  \big( r \ (1-n)+1 \big) \Big( (2h-1) \  s \  \p{D} + (1-sh) (1+c) - 1  \Big) \  \p{D} \ (1- \p{D}) .
    \end{array}
\right.
\end{equation}

\subsubsubsection{$r = + \infty$}\label{subsubsubsec:par_ger_rinf}


Similarly as in Section \ref{subsubsec:perfect_rinf}, we can compute formally the limiting equation on $\p{D}$ when $r = +\infty$:

\begin{equation}\label{eq:par_ger_rinf}
    \partial_t  \p{D}  -  \partial_{xx}^2  \p{D}  \ = \dfrac{ \Big(  - (1-2h) \  s \  \p{D} + [ (1-sh) (1+c) - 1 ]  \Big) \  \p{D} \ (1- \p{D}) }{(1-s) \  \p{D}^2 + 2 \ (1-sh) \   \p{D} \ (1- \p{D}) + (1- \p{D})^2}.
\end{equation}

Note that as in section \ref{subsec:perfect}, this equation does not depend on $n$. We introduce:

\begin{equation}
    \mathscr{A}_g := s \ (1-2h), \quad \quad s_1 := \dfrac{c}{1-h(1-c)}, \quad \quad s_{2,g} :=  \dfrac{c}{2ch + h (1-c)} =  \dfrac{c}{h(1+c)}.
\end{equation}

where $g$ stands for germline. Note that $ \mathscr{A}_g  > 0 \iff s_1 < s_{2,g}$. We define a set $\mathscr{S}_g$ of $s$ values: \begin{equation}
    \mathscr{S}_g := \Big\{ s \in (0,1) | ( 1 - 2 s h) \big(c - s h ( c + 1 ) \big)  + s (1 - 2h)  > 0  \Big\}.
\end{equation}

Results are exactly the same as in Section \ref{subsubsubsec:par_zyg_rinf}, substituting $\mathscr{A}_z $ by $\mathscr{A}_g$, $s_{2,z}$ by $s_{2,g}$, $\mathscr{S}_z$ by  $\mathscr{S}_g$, and the minimal speed of the problem linearized at low drive density \eqref{eq:lin_speed_zyg_drive} by \eqref{eq:lin_speed_ger_drive} (see Appendix \ref{ann:pull_ger}).





\subsubsubsection{$r = 0$}




Using the relation $ n = \n{W} + \n{D}$, system (\ref{eq:par_ger_nD_nW}) can be rewritten as follows when $r = 0$: 
\begin{equation} \label{eq:par_ger_r0} 
\left\{
    \begin{array}{ll}
      \partial_t \n{D}  - \partial_{xx}^2 \n{D} \ =  \ \Big( c \ (1-sh) + s \ (1-h)  \Big) \ \dfrac{\n{D} \   \n{W}}{ \n{D} +  \n{W}}  - s \ \n{D}, \\
        \\
      \partial_t  \n{W} - \partial_{xx}^2  \n{W} =  \ - \Big(  1 - (1-sh) \ (1-c)    \Big) \ \dfrac{\n{D} \   \n{W}}{\n{D} +  \n{W}}.
    \end{array}
\right. 
\end{equation}

We apply the results of Appendix \ref{ann:exist} with $\beta_1 =  1 - (1-sh) \ (1-c) $ and $\beta_2 = c \ (1-sh) + s \ (1-h) $. There exists a monostable and pulled drive invasion wave if:
\begin{equation}
    \beta_2 > \gamma \ \ \iff \ \ s < s_{2,g} = \dfrac{c}{h(1+c)} = \dfrac{c}{2ch + h(1-c)}.
\end{equation}

On the other hand when $ \beta_2 < \gamma$, the reaction term of $n_D$ in \eqref{eq:par_ger_r0} is strictly negative. As before, the density $\n{D}$ converges to zero uniformly in space at rate at least $\beta_2 - \gamma$ (gene drive clearance) and the final density of wild-type is not clearly determined: the problem boils down to diffusion only in the large time asymptotics (details in Appendix \ref{ann:gd_clearance}).

Note that the same intuitions as in section \ref{subsubsubsec:par_zyg_r0} hold: when $\mathscr{A}_g > 0$, a condition for having a pulled wave with positive speed for both $r = 0$ and $r = + \infty$ is $s < s_{2,g}$ ; when $\mathscr{A}_g < 0$, a condition for having a pulled wave with positive speed for both $r = 0$ and $r = + \infty$ is $ s \in \mathscr{S}_g \subseteq [0, s_{2,g}]$. This suggests that, under those conditions, whatever the value of the demographic parameter $r$ is, the drive invasion wave is always pulled and consequently, travels at a speed which does not depend on $r$ either (speed given by \eqref{eq:lin_speed_ger_drive}). As before, we verify this intuition numerically (vertical level lines) in the following section.

\subsubsubsection{Numerical illustrations}

\vspace{-1cm}

\begin{figure}[H]
\centering
\begin{subfigure}{0.48\textwidth}
    \centering
    \includegraphics[scale = 0.5]{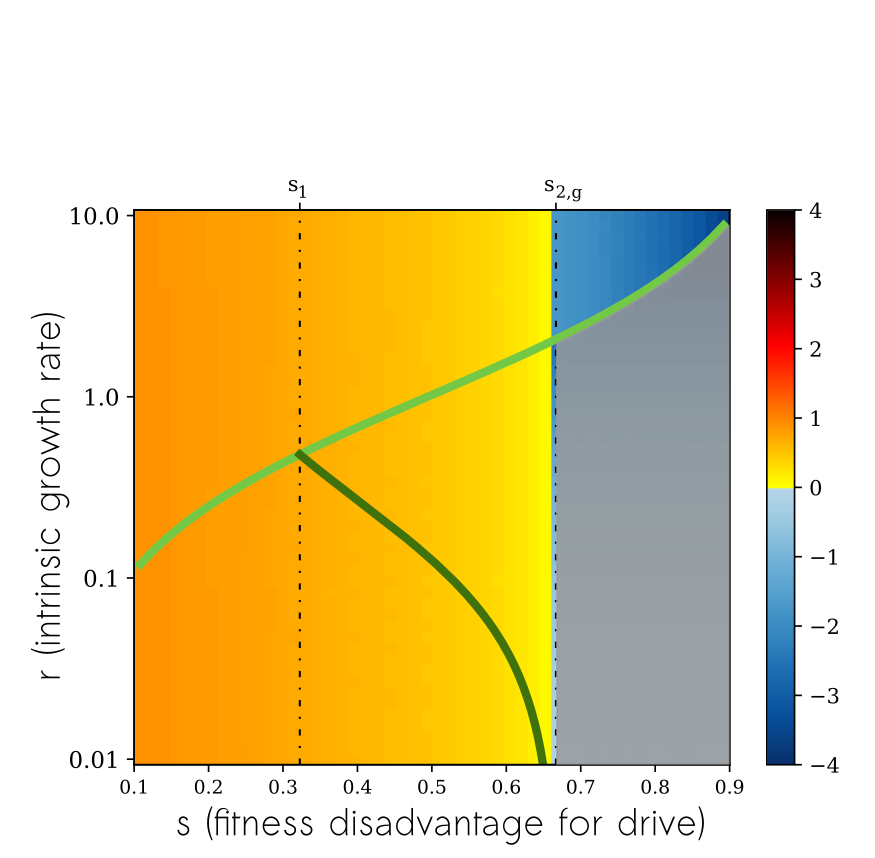}
    \caption{$c = 0.25$, $h = 0.3$}
    \label{fig:heatmap_ger_coex}
\end{subfigure}
\hfill
\begin{subfigure}{0.48\textwidth}
    \centering
    \includegraphics[scale = 0.5]{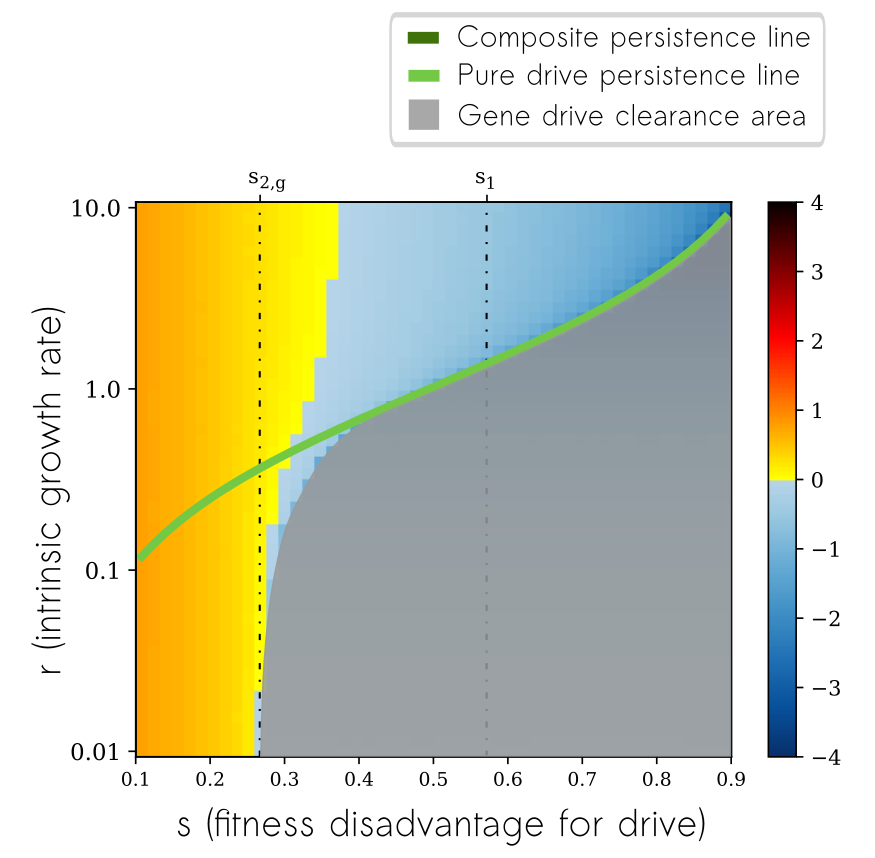}
    \caption{$c = 0.25, h = 0.75$}
    \label{fig:heatmap_ger_bist}
\end{subfigure}
\caption{Heatmap representing the speed of the wave when conversion occurs in the germline. When the drive invades the population, the speed is positive (in yellow-orange-red). On the contrary, when the wild-type invades the population, the speed is negative (in blue). Below the pure drive persistence line (light green), a well-mixed population containing only drive homozygous individuals will necessarily go extinct. Below the composite persistence line (dark green), it is the whole population that goes extinct (calculations for both lines available in Appendix \ref{ann:persistence}). The gray zone corresponds to the gene drive clearance area. In (a) we have $\mathscr{A}_g > 0$, therefore the system is always monostable for $r = + \infty$: when $ s < s_1$ the drive always invades; when $ s_1 < s < s_{2,g}$ the final state is a coexistence state; when $ s > s_{2,g}$ the wild-type invades or there is gene drive clearance. When both drive and wild types invade (coexistence), only the speed of the drive is shown in the heatmap, resulting in an apparent discontinuity at $s=s_{2,z}$. In (b) we have $\mathscr{A}_z<0$, therefore when $r = + \infty$: when $ s < s_{2,g}$ the system is monostable and the drive always invades; when $ s_{2,g} < s < s_1$ the system is bistable and the final state depends on the initial condition; when $ s > s_1$ the system is monostable and the wild-type invades or there is gene drive clearance.}
\end{figure}

For a better understanding of Figures \ref{fig:heatmap_ger_coex} and \ref{fig:heatmap_ger_bist}, we detail the effect of fitness disadvantage $s$ and dominance coefficient $h$ on drive dynamics for $r = + \infty$ and $c=0.25$, without spatial structure, in Appendix \ref{ann:rode_debarre} (in Figure \ref{ann:rode_ger_c25}).

\subsubsection{Conclusion}

When conversion occurs in the zygote (resp. in the germline) for $\mathscr{A}_z < 0$ (resp. $\mathscr{A}_g < 0$), demographics influence the speed of the drive propagation at least for $s \in (s_{2,z}, s_1)$ (resp. $s \in (s_{2,g}, s_1)$). More precisely, the sign of the speed can switch, changing the type of the invasion (drive or wild-type). When $\mathscr{A}_z > 0$ (resp. $\mathscr{A}_g > 0$) however, a model following only frequencies will always predict the correct speed of expansion. However, a model following only frequencies will not provide information on population size, and in particular whether the population is suppressed or eradicated, while this point is of great biological relevance. 

For both zygote and germline conversion timings, the critical values of $\mathscr{A}_z$ and $\mathscr{A}_g$ can be interpreted as the values at which the fitness of adults who were born heterozygous ($f_H'$) is the arithmetic mean of the fitness of adults born homozygote ($(f_D + f_W)/2$). The fitness of adults who were born heterozygous depends on the timing of gene conversion. For germline conversion, $f_H' = f_H$, and $\mathscr{A}_g  = 0$ when $h = 1/2$, i.e. when there is co-dominance between the drive and wild-type alleles, i.e. when $f_H' = f_H = (f_D + f_W)/2$. For zygote conversion, the fitness of adults born heterozygous depends on whether gene conversion has taken place or not ($f_H' = (1-c) (1-hs) + c (1-s)$). The condition $\mathscr{A}_z = 0$ is equivalent to $(1-c)(1-h) = 1/2$, which happens for $f_H' = (f_D + f_W)/2$.

\section{Discussion}\label{sec:discuss}

Following \cite{girardin2021}, we quantified the impact of demography in the case of the propagation of a super-Mendelian drive. We extended the analysis of reference \cite{girardin2021} to the case of partial conversion ($0 < c < 1$), implying the presence of heterozygotes. 

\subsection*{On the final state of the population}
The final size of the population naturally varies. In case where no wild type can survive, the final size is the same regardless of the details of gene conversion(timing nor probability): $\n{D}^* = \min(0, 1- \frac{s}{r (1-s)})$. In case of coexistence between wild-type and drive alleles, the final size depends on all parameters (see Appendix \ref{an:composite_line}). Interestingly, in the case of coexistence, the drive allele can persist in the population even if a pure drive population would not ($\n{D}^* = 0$), see Figure \ref{fig:heatmap_zyg_coex} (note the area between the {\em composite persistence line} and the {\em pure drive persistence line}). In contrast with standard Mendelian genetics (corresponding to $c=0$), coexistence can occur even if the dominance parameter is such that $h\in (0,1)$ \cite{deredec2008,rode2019}. More precisely, when conversion is partial and, either $h<1 - \frac{1}{2(1-c)}$ (zygote conversion), or $h<\frac12$ (germline conversion), there exists a stable coexistence state if $s$ takes intermediate values $s\in(s_1,s_{2,z})$ (zygote conversion), or $s\in(s_1,s_{2,g})$ (germline conversion), where $s_1,s_{2,z},s_{2,g}$ depend on $(c,h)$ but do not depend on the demographic parameter $r$ (see details in Section \ref{subsubsec:partial_zyg} and \ref{subsubsec:partial_ger}). While the final size of the population naturally depends on $r$.

\subsection*{On the transient regime (propagation of waves)}
In order to evaluate the impact of demography on the dynamics of drive expansion, we compared the extreme cases $r \to \infty$ and $r \to 0$ (resp. low demographic variations versus large demographic variations). 

For $r=0$, we found that, when the drive propagates, it does so through a monostable and pulled wave. This happens when the drive is not too costly. In contrary, the drive gets uniformly extinct if it is too costly. The threshold on the fitness cost $s_{2,z}$ (zygote conversion), or $s_{2,g}$ (germline conversion), depends on $(c,h)$. The situation is analogous to the spatial spreading of an epidemic following a SI type model.  

The case $r=0$ gives the possibility to measure the importance of the demographic advection term $2 \partial_x \log (n) \partial_x p$ when the problem is formulated in frequency, see equations \eqref{eq:per_zyg_p}, \eqref{eq:par_zyg_p}, \eqref{eq:par_ger_p}. In fact, we show that ignoring this term can lead to an overestimation of the wave speed. This happens, for instance, in case of perfect conversion in the zygote, when $s \in(\frac25,\frac12)$, then the equation \eqref{eq:per_zyg_p} without  $2 \partial_x \log (n) \partial_x p$ would lead to a pushed front with velocity $\frac{2-3s}{\sqrt{2s}}$ \cite{hadeler1975}. However, we proved that the front is actually pulled with velocity $2\sqrt{1-2s}< \frac{2-3s}{\sqrt{2s}}$. Intuitively, advection due to demographic variations slows down the expansion of the bulk. Noticeably, the effect is so strong that it prevents the front from being pushed. 

In contrast, for $r = \infty$, the analysis boils down to a single equation on the drive allele frequency \cite{strugarek2016,girardin2021}. According to \cite{rode2019}, where the case of germline conversion was investigated, there is a large panel of frequency-dependence relationships, including monostable fixation of one allele, bi-stability, and stable coexistence between the two alleles, even if $h\in (0,1)$. This leads to a variety of propagation phenomena, either pulled or pushed, as described in Section \ref{subsubsec:partial_ger}.
The same panel of relationships arises in the case of zygote conversion, with qualitative similarities but quantitative differences in the thresholds and in the wave speeds, compare Section \ref{subsubsec:partial_ger} with Section \ref{subsubsec:partial_zyg}.

To connect $r=0$ and $r=\infty$, we conjecture that, if the wave of the drive is pulled at $r = \infty$, then it is pulled for any value of $r>0$, and the wave speed is independent of $r$. In particular, this occurs when the frequency-dependence term induces monostable dynamics and $s$ is small enough, or when there is stable co-existence. This conjecture is supported by numerical investigations (Figures \ref{fig:heatmap_zyg_coex}, \ref{fig:heatmap_zyg_bist}, \ref{fig:heatmap_ger_coex} and \ref{fig:heatmap_ger_bist}). Still, the final size of population naturally depends on $r$.

\subsection*{Perspectives}

We have focused on the classical dichotomy between pulled and pushed waves, even if the transition between the two is subject to current research both in theoretical studies \cite{an2021,avery2022,birzu2018}, and in experimental works \cite{dahirel2021}. 

Pulled and pushed waves are associated with different outcomes on the maintenance of neutral diversity (which was not considered in our study). The genetic diversity of a population expanding by a pulled wave is very limited (with possible accumulation of deleterious mutations \cite{peischl2013}), while more diversity is maintained under a pushed wave \cite{roques2012}. It could be interesting to investigate how gene conversion influences the maintenance of diversity along an expanding wave. 
More generally, the bottleneck following spread of a suppression drive will affect diversity, which may have long-lasting consequences even if wild-type individuals later recolonize the area.

It would be highly relevant to explore stochastic dynamics beyond our deterministic approach. When population sizes get to be small, as in the drive eradication case, large fluctuations and even chasing events are expected, as described in \cite{champer2021}. It would also be extremely interesting to extend the scope of the model, including by distinguishing between males and females which may have different fitnesses (especially in transgenic mosquitoes \cite{beaghton2016, north2020, kyrou2018, hammond2015}). Plural life stages or haploid phases might also influence modelling conclusions \cite{liu2022, li2020}.

\section{Acknowledgements}

This work is funded by ANR-19-CE45-0009-01 "TheoGeneDrive".

This project has received funding from the European Research Council
(ERC) under the European Union’s Horizon 2020 research and innovation
programm (grant agreement No 865711).

We are grateful to the INRAE MIGALE bioinformatics facility (MIGALE, INRAE, 2020. Migale bioinformatics Facility, doi: 10.15454/1.5572390655343293E12) for providing computing resources.

\newpage


\addtocontents{toc}{\bigskip \bigskip \hspace{-0,65cm} \large \textbf{Appendix} \normalsize \par}

\addtocontents{lof}{ \bigskip \bigskip \textbf{Appendix} \par }

\appendix

\Large
\textbf{Appendix}

\normalsize

\section{Model with partial conversion: growth term details}\label{an:growthterm}

To obtain the global growth term for each genotypes in models \eqref{eq:par_zyg} and \eqref{eq:par_ger}, we calculate type proportions among the offspring for each possible couple, and then we sum the corresponding terms. The calculations follow standard lines of population genetics, differing only by the timing of gene conversion. When conversion occurs in the zygote, the parameter $c$ appears in between the gametes and the offspring production, whereas when conversion occurs in the germline, it appears before gametes production. These equations in densities are consistent with the one obtained in frequency in the literature (when conversion occurs in the zygote \cite{deredec2008, unckless2015}, or in the germline \cite{deredec2008, rode2019}).

\subsection{Conversion occurring in the zygote}

\begin{table}[H]
\begin{tabular}{|M{.13\textwidth}|M{.37\textwidth}|M{.08\textwidth}|M{.3\textwidth}|}
\hline
  \textbf{Parents} & $ \ $ \hspace{1.25cm} \textbf{Egg} \hspace{1.85cm} \textbf{Adult}  &  \textbf{Fitness} &  \textbf{Growth term} \\
 \hline
    WW + WW & \begin{tikzpicture}
\begin{scope}[every node/.style={circle}]
    \node (A) at (0,-0){};
    \node (B) at (2,0) {WW};
    \node (C) at (5,0) {WW};
\end{scope}
\begin{scope}[>={Stealth[black]},
              every node/.style={fill=white},
              every edge/.style={draw=black, thick}]
    \path [->] (A) edge node {$1$} (B);
    \path [->] (B) edge node {$1$} (C);
\end{scope}
\end{tikzpicture} & 1 & $ \dfrac{\n{WW} \n{WW}}{n} $ \\
\hline
   \vspace{0.9cm} WW + WD & \begin{tikzpicture}
\begin{scope}[every node/.style={circle}]
    \node (A) at (0,-1.5){};
    \node (B) at (2,-0.5) {WD};
    \node (C) at (5,0) {DD};
    \node (D) at (5,-1) {WD};
    \node (E) at (2,-2) {WW};
    \node (F) at (5,-2) {WW};
\end{scope}
\begin{scope}[>={Stealth[black]},
              every node/.style={fill=white},
              every edge/.style={draw=black, thick}]
    \path [->] (A) edge node {$\frac{1}{2}$} (B);
    \path [->] (B) edge node {$c$} (C);
    \path [->] (B) edge node {$1-c$} (D);
    \path [->] (A) edge node {$\frac{1}{2}$} (E);
    \path [->] (E) edge node {$1$} (F);
\end{scope}
\end{tikzpicture} 
& \begin{tikzpicture}
\begin{scope}
    \node (1) at (0,0){$ 1-s $};
    \node (2) at (0,-1) {$ 1-sh $};
    \node (3) at (0,-2) {$ 1 $};
\end{scope}
\end{tikzpicture}
&   \begin{tikzpicture}
\begin{scope}
    \node (1) at (0,0){$ \frac{1}{2} \ c \ \   (1-s)     \ \dfrac{2 \n{WW} \n{DW}}{n} $};
    \node (2) at (0,-1) {$ \frac{1}{2} \ (1-c) \ \   (1-sh)     \ \dfrac{2 \n{WW} \n{DW}}{n} $};
    \node (3) at (0,-2) {$ \frac{1}{2} \        \ \dfrac{2 \n{WW} \n{DW}}{n} $};
\end{scope}
\end{tikzpicture} \\
 \hline
     WW + DD & \begin{tikzpicture}
\begin{scope}[every node/.style={circle}]
    \node (A) at (0,-0.5){};
    \node (B) at (2,-0.5) {WD};
    \node (C) at (5,0) {DD};
    \node (D) at (5,-1) {WD};
\end{scope}
\begin{scope}[>={Stealth[black]},
              every node/.style={fill=white},
              every edge/.style={draw=black, thick}]
    \path [->] (A) edge node {$1$} (B);
    \path [->] (B) edge node {$c$} (C);
    \path [->] (B) edge node {$1-c$} (D);
\end{scope}
\end{tikzpicture} &   
\begin{tikzpicture}
\begin{scope}
    \node (1) at (0,0){$ 1-s $};
    \node (2) at (0,-1) {$ 1-sh $};
\end{scope}
\end{tikzpicture}&   
\begin{tikzpicture}
\begin{scope}
    \node (1) at (0,0){$ c \ \   (1-s)     \ \dfrac{2 \n{WW} \n{DD}}{n} $};
    \node (2) at (0,-1) {$ (1-c) \ \  (1-sh)     \ \dfrac{2 \n{WW} \n{DD}}{n} $};
\end{scope}
\end{tikzpicture}\\
\hline
    WD + WD & \begin{tikzpicture}
\begin{scope}[every node/.style={circle}]
    \node (A) at (0,-1.5){};
    \node (B) at (2,0) {WW};
    \node (C) at (5,0) {WW};
    \node (D) at (2,-1.5) {WD};
    \node (E) at (5,-1) {DD};
    \node (F) at (5,-2) {WD};
    \node (G) at (2,-3) {DD};
    \node (H) at (5,-3) {DD};
\end{scope}

\begin{scope}[>={Stealth[black]},
              every node/.style={fill=white},
              every edge/.style={draw=black, thick}]
    \path [->] (A) edge node {$\frac{1}{4}$} (B);
    \path [->] (B) edge node {$1$} (C);
    \path [->] (A) edge node {$\frac{1}{2}$} (D);
    \path [->] (D) edge node {$c$} (E);
    \path [->] (D) edge node {$1-c$} (F);
    \path [->] (A) edge node {$\frac{1}{4}$} (G);
    \path [->] (G) edge node {$1$} (H);
\end{scope}
\end{tikzpicture}  &  
\begin{tikzpicture}
\begin{scope}
    \node (1) at (0,0){$ 1 $};
    \node (2) at (0,-1) {$ 1-s  $};
    \node (3) at (0,-2) {$ 1-sh $};
    \node (4) at (0,-3) {$ 1-s $};
\end{scope}
\end{tikzpicture} &  
\begin{tikzpicture}
\begin{scope}
    \node (1) at (0,0){$ \frac{1}{4} \         \ \dfrac{\n{DW} \n{DW}}{n} $};
    \node (2) at (0,-1) {$ \frac{1}{2} \ c \ \   (1-s)      \ \dfrac{\n{DW} \n{DW}}{n}  $};
    \node (3) at (0,-2) {$ \frac{1}{2} \ (1-c) \ \  (1-sh)      \ \dfrac{\n{DW} \n{DW}}{n}  $};
    \node (4) at (0,-3) {$ \frac{1}{4} \ \   (1-s)      \ \dfrac{\n{DW} \n{DW}}{n}  $};
\end{scope}
\end{tikzpicture}\\
\hline
  \vspace{0.9cm}  WD + DD & \begin{tikzpicture}
\begin{scope}[every node/.style={circle}]
    \node (A) at (0,-1.5){};
    \node (B) at (2,-0.5) {WD};
    \node (C) at (5,0) {DD};
    \node (D) at (5,-1) {WD};
    \node (E) at (2,-2) {DD};
    \node (F) at (5,-2) {DD};
\end{scope}
\begin{scope}[>={Stealth[black]},
              every node/.style={fill=white},
              every edge/.style={draw=black, thick}]
    \path [->] (A) edge node {$\frac{1}{2}$} (B);
    \path [->] (B) edge node {$c$} (C);
    \path [->] (B) edge node {$1-c$} (D);
    \path [->] (A) edge node {$\frac{1}{2}$} (E);
    \path [->] (E) edge node {$1$} (F);
\end{scope}
\end{tikzpicture}  &  
\begin{tikzpicture}
\begin{scope}
    \node (1) at (0,0){$ 1-s $};
    \node (2) at (0,-1) {$ 1-sh $};
    \node (3) at (0,-2) {$ 1-s $};
\end{scope}
\end{tikzpicture}&  
\begin{tikzpicture}
\begin{scope}
    \node (1) at (0,0){$ \frac{1}{2} \ c \ \   (1-s)     \ \dfrac{2 \n{DW} \n{DD}}{n} $};
    \node (2) at (0,-1) {$ \frac{1}{2} \ (1-c) \ \  (1-sh)     \ \dfrac{2 \n{DW} \n{DD}}{n} $};
    \node (3) at (0,-2) {$ \frac{1}{2} \ \   (1-s)     \ \dfrac{2 \n{DW} \n{DD}}{n} $};
\end{scope}
\end{tikzpicture}\\
 \hline
  DD + DD & \begin{tikzpicture}
\begin{scope}[every node/.style={circle}]
    \node (A) at (0,-0){};
    \node (B) at (2,0) {DD};
    \node (C) at (5,0) {DD};
\end{scope}
\begin{scope}[>={Stealth[black]},
              every node/.style={fill=white},
              every edge/.style={draw=black, thick}]
    \path [->] (A) edge node {$1$} (B);
    \path [->] (B) edge node {$1$} (C);
\end{scope}
\end{tikzpicture} &  $1-s$ &
\begin{tikzpicture}
\begin{scope}
    \node (1) at (0,0){$  (1-s)  \  \dfrac{\n{DD} \n{DD}}{n} $};
\end{scope}
\end{tikzpicture}\\
\hline
\end{tabular}
\caption{Growth term details when conversion occurs in the zygote.}
\end{table}

\newpage

\subsection{Conversion occurring in the germline}


\begin{table}[H]
\begin{tabular}{|M{.13\textwidth}|M{.43\textwidth}|M{.08\textwidth}|M{.3\textwidth}|}
\hline
  \textbf{Parents} & $ \ $ \hspace{2cm} \textbf{Gametes} \hspace{1.5cm} \textbf{Adult}  &  \textbf{Fitness} &  \textbf{Growth term} \\
 \hline
    WW + WW & \begin{tikzpicture}
\begin{scope}[every node/.style={circle}]
    \node (A) at (0,-0){};
    \node (B) at (3,0) {W,W + W,W};
    \node (C) at (6,0) {WW};
\end{scope}
\begin{scope}[>={Stealth[black]},
              every node/.style={fill=white},
              every edge/.style={draw=black, thick}]
    \path [->] (A) edge node {$1$} (B);
    \path [->] (B) edge node {$1$} (C);
\end{scope}
\end{tikzpicture} &  $1$ & $ \dfrac{\n{WW} \n{WW}}{n} $ \\
\hline 
  WW + WD  \vspace{0.2cm}  & \vspace{-0.2cm}  \begin{tikzpicture}
\begin{scope}[every node/.style={circle}]
    \node (A) at (0,-0.75){};
    \node (B) at (3,0) {W,W + D,D};
    \node (C) at (6,0) {WD};%
    \node (D) at (6,-1) {WW};%
    \node (E) at (3,-1.5) {W,W + W,D};
    \node (F) at (6,-2) {WD};%
\end{scope}
\begin{scope}[>={Stealth[black]},
              every node/.style={fill=white},
              every edge/.style={draw=black, thick}]
    \path [->] (A) edge node {$c$} (B);
    \path [->] (B) edge node {$1$} (C);
    \path [->] (E) edge node {$\frac{1}{2}$} (D);
    \path [->] (A) edge node {$1-c$} (E);
    \path [->] (E) edge node {$\frac{1}{2}$} (F);
\end{scope}
\end{tikzpicture} 
 &   \begin{tikzpicture}
\begin{scope}
    \node (1) at (0,0){$  1-sh $};
    \node (2) at (0,-1) {$ 1$};
    \node (3) at (0,-2) {$ 1-sh $};
\end{scope}
\end{tikzpicture} 
&   \begin{tikzpicture}
\begin{scope}
    \node (1) at (0,0){$  c \ \ (1-sh)    \ \dfrac{2 \n{WW} \n{DW}}{n} $};
    \node (2) at (0,-1) {$  (1-c) \ \frac{1}{2} \    \ \dfrac{2 \n{WW} \n{DW}}{n} $};
    \node (3) at (0,-2) {$ (1-c) \ \frac{1}{2} \  \ (1-sh)   \ \dfrac{2 \n{WW} \n{DW}}{n} $};
\end{scope}
\end{tikzpicture} \\
 \hline 
    WW + DD &  \begin{tikzpicture}
\begin{scope}[every node/.style={circle}]
    \node (A) at (0,-0){};
    \node (B) at (3,0) {W,W + D,D};
    \node (C) at (6,0) {WD};
\end{scope}
\begin{scope}[>={Stealth[black]},
              every node/.style={fill=white},
              every edge/.style={draw=black, thick}]
    \path [->] (A) edge node {$1$} (B);
    \path [->] (B) edge node {$1$} (C);
\end{scope}
\end{tikzpicture} &  $1-sh$ & $  (1-sh)   \ \dfrac{2 \n{WW} \n{DD}}{n} $ \\
\hline 
    WD + WD  \vspace{1.4cm} & \vspace{-0.6cm}  \begin{tikzpicture}
\begin{scope}[every node/.style={circle}]
    \node (A) at (-0.5,-1.5){};
    \node (B) at (3,0.2) {D,D + D,D};
    \node (C) at (6,0.2) {DD};
    \node (D) at (3,-1.5) {D,D + W,D};
    \node (E) at (6,-1) {DD};
    \node (F) at (6,-2) {WD};
    \node (G) at (3,-4) {W,D + W,D};
    \node (H) at (6,-3) {DD};
    \node (I) at (6,-4) {WD};
    \node (J) at (6,-5) {WW};
\end{scope}

\begin{scope}[>={Stealth[black]},
              every node/.style={fill=white},
              every edge/.style={draw=black, thick}]
    \path [->] (A) edge node {$c^2$} (B);
    \path [->] (B) edge node {$1$} (C);
    \path [->] (A) edge node {$2c(1-c)$} (D);
    \path [->] (D) edge node {$\frac{1}{2}$} (E);
    \path [->] (D) edge node {$\frac{1}{2}$} (F);
    \path [->] (A) edge node {$(1-c)^2$} (G);
    \path [->] (G) edge node {$\frac{1}{4}$} (H);
    \path [->] (G) edge node {$\frac{1}{2}$} (I);
    \path [->] (G) edge node {$\frac{1}{4}$} (J);
\end{scope}
\end{tikzpicture} &  
\begin{tikzpicture}
\begin{scope}
    \node (1) at (0,-0){$1-s $};
    \node (2) at (0,-1) {$ 1-s  $};
    \node (3) at (0,-2) {$ 1-sh $};
    \node (4) at (0,-3) {$ 1-s $};
    \node (5) at (0,-4) {$ 1-sh$};
    \node (6) at (0,-5) {$ 1 $};
\end{scope}
\end{tikzpicture} &  
\begin{tikzpicture}
\begin{scope}
    \node (1) at (0,0){$ c^2 \ \ (1-s)  \ \dfrac{\n{DW} \n{DW}}{n} $};
    \node (2) at (0,-1) {$ c \ (1-c) \ \ (1-s)      \ \dfrac{\n{DW} \n{DW}}{n}  $};
    \node (3) at (0,-2) {$ c \ (1-c)  \ \ (1-sh)      \ \dfrac{\n{DW} \n{DW}}{n} $};
    \node (4) at (0,-3) {$ (1-c)^2 \  \frac{1}{4} \ \ (1-s)     \  \dfrac{\n{DW} \n{DW}}{n}  $};
    \node (5) at (0,-4) {$ (1-c)^2 \  \frac{1}{2} \ \ (1-sh)    \ \dfrac{\n{DW} \n{DW}}{n}   $};
    \node (6) at (0,-5) {$ (1-c)^2 \  \frac{1}{4} \    \  \dfrac{\n{DW} \n{DW}}{n}  $};
\end{scope}
\end{tikzpicture}\\
\hline 
  WD + DD  \vspace{0.2cm}  & \vspace{-0.2cm}  \begin{tikzpicture}
\begin{scope}[every node/.style={circle}]
    \node (A) at (0,-0.75){};
    \node (B) at (3,0) {D,D + D,D};
    \node (C) at (6,0) {DD};%
    \node (D) at (6,-1) {WD};%
    \node (E) at (3,-1.5) {W,D + D,D};
    \node (F) at (6,-2) {DD};%
\end{scope}
\begin{scope}[>={Stealth[black]},
              every node/.style={fill=white},
              every edge/.style={draw=black, thick}]
    \path [->] (A) edge node {$c$} (B);
    \path [->] (B) edge node {$1$} (C);
    \path [->] (E) edge node {$\frac{1}{2}$} (D);
    \path [->] (A) edge node {$1-c$} (E);
    \path [->] (E) edge node {$\frac{1}{2}$} (F);
\end{scope}
\end{tikzpicture}  &   \begin{tikzpicture}
\begin{scope}
    \node (1) at (0,0){$ 1-s $};
    \node (2) at (0,-1) {$ 1-sh $};
    \node (3) at (0,-2) {$ 1-s $};
\end{scope}
\end{tikzpicture}  &   \begin{tikzpicture}
\begin{scope}
    \node (1) at (0,0){$  c \ \ (1-s)    \ \dfrac{2 \n{DW} \n{DD}}{n} $};
    \node (2) at (0,-1) {$ (1-c) \ \frac{1}{2} \  \  (1-sh)    \ \dfrac{2 \n{DW} \n{DD}}{n} $};
    \node (3) at (0,-2) {$ (1-c) \ \frac{1}{2} \ \ (1-s)   \ \dfrac{2 \n{DW} \n{DD}}{n} $};
\end{scope}
\end{tikzpicture} \\
 \hline 
  DD + DD & \begin{tikzpicture}
\begin{scope}[every node/.style={circle}]
    \node (A) at (0,-0){};
    \node (B) at (3,0) {D,D + D,D};
    \node (C) at (6,0) {DD};
\end{scope}
\begin{scope}[>={Stealth[black]},
              every node/.style={fill=white},
              every edge/.style={draw=black, thick}]
    \path [->] (A) edge node {$1$} (B);
    \path [->] (B) edge node {$1$} (C);
\end{scope}
\end{tikzpicture} &  $1-s$ &
\begin{tikzpicture}
\begin{scope}
    \node (1) at (0,0){$ (1-s)    \ \dfrac{\n{DD} \n{DD}}{n} $};
\end{scope}
\end{tikzpicture}\\
\hline
\end{tabular}
\caption{Growth term details when conversion occurs in the germline.}
\end{table}

\newpage

\section{System rewritten with variables $(n, \p{D})$}

Below, we present the details of the reformulation from models \eqref{eq:per_zyg}, \eqref{eq:par_zyg_nD_nW} and \eqref{eq:par_ger_nD_nW} in terms of total population density $n$ and drive allele frequency $\p{D}$.


\subsection{Model with perfect conversion} \label{an:rewrite1}

We rewrite model \eqref{eq:per_zyg} with variables:

\begin{equation}
    n = \n{WW} + \n{DD}, \quad \quad \quad \quad  \p{D} = \dfrac{\n{D}}{\n{W} + \n{D}} = \dfrac{\n{DD}}{\n{WW} + \n{DD}} .
\end{equation} 

where $n$ is the total population density and $\p{D}$ is the drive allele frequency, or equivalently in this model, the frequency of drive homogeneous individuals.\\

\textbf{Equation on n:} \begin{equation}
\begin{aligned}
   \partial_t n - \partial_{xx}^2 n & =  \Big(r  \ (1-n) + 1 \Big) \ \Big( (1-s) \ \dfrac{\n{DD}^2 + 2 \ \n{WW} \n{DD}}{\n{WW}+\n{DD}} +  \ \dfrac{\n{WW}^2}{\n{WW}+\n{DD}} \Big) - n, \\ 
 & =  \Big(r  \ (1-n) + 1 \Big)  \ \Big( (1-s) \ \p{D}^2 + (1-s) \ 2 \ \p{D} (1-\p{D}) + (1-s ) (1-\p{D})^2 + s (1-\p{D})^2 \Big) n - n, \\
 & =  \Big(r  \ (1-n) + 1 \Big)  \ \Big( (1-s) (\p{D} + 1 - \p{D})^2   + s (1-\p{D})^2 \Big) n - n, \\
 & =  \Big(r  \ (1-n) + 1 \Big)  \ \Big( (1-s) + s (1-\p{D})^2 \Big) n - n, \\ 
 & =  \Big(r  \ (1-n) + 1 \Big)  \ \Big( (1-s) \  \p{D}  \ (2-\p{D}) + (1-\p{D})^2 \Big) n - n.
\end{aligned} 
\end{equation}

\textbf{Equation on $\mathbf{\n{DD}}$:} \begin{equation}
\begin{aligned}
\partial_t \n{DD} & =  \partial_{xx}^2 \n{DD} \  + \ (1-s) \ \Big(r  \ (1-\n{DD}-\n{WW}) + 1 \Big) \ \dfrac{\n{DD}^2 + 2 \ \n{WW} \n{DD}}{\n{WW}+\n{DD}} - \n{DD},\\
& \stackrel{\footnotemark}{=}  \p{D} \ \partial_{xx}^2 n + 2 \ \partial_x n \ \partial_x \p{D} +  n \ \partial_{xx}^2 \p{D}  \  + \ (1-s) \ \Big(r  \ (1-n) + 1 \Big) \ ( 2 -\p{D} ) \ n \ \p{D} - n \ \p{D}.
\end{aligned} 
\end{equation}

\footnotetext{$\partial_{xx}^2 \n{DD} = \partial_{xx}^2 n \p{D} =  \partial_{x} ( \p{D} \ \partial_{x} n + n \ \partial_{x} \p{D} ) = \p{D} \ \partial_{xx}^2  n + 2 \ \partial_{x} \p{D}  \ \partial_{x} n + n  \ \partial_{xx}^2  \p{D} $} 




\textbf{Equation on $\mathbf{\p{D} = \dfrac{\n{DD}}{n}}$:} \begin{equation}
\begin{aligned}
    \partial_t \p{D}  &   = \dfrac{ n \ (\partial_t \n{DD}) - \n{DD} \ (\partial_t n)}{n^2} = \dfrac{ n \ (\partial_t \n{DD}) - n \ \p{D} \ (\partial_t n)}{n^2} = \dfrac{1}{n} \big( \ \partial_t \n{DD} - \p{D} \ (\partial_t n) \big), \\
    & =  \dfrac{1}{n} \Big[ \   \p{D} \ \partial_{xx}^2 n + 2 \ \partial_x n \ \partial_x \p{D} +  n \ \partial_{xx}^2 \p{D}  \  + \  \Big(r  \ (1-n) + 1 \Big) \ (1-s) \ ( 2 -\p{D} ) \ n \ \p{D} - n \ \p{D} \ \Big] \\
  & \quad \quad   -  \dfrac{1}{n} \big[  \  \p{D}  \partial_{xx}^2 n \  +  \ \Big(r  \ (1-n) + 1 \Big)  \ \Big( (1-s) \  \p{D}  \ (2-\p{D}) + (1-\p{D})^2 \Big) n \ \p{D} - n \ \p{D} \ \Big], \\
     &  =  \partial_{xx}^2 \p{D} \  +  \ 2 \ \partial_x \log(n) \ \partial_x \p{D} \  +  \Big(r  \ (1-n) + 1 \Big) \ \p{D} \ \Big( (1-s) ( 2-\p{D} )    -   (1-s) (2-\p{D}) \p{D} - (1-\p{D})^2 \Big), \\ 
 & = \partial_{xx}^2 \p{D} \  +  \ 2 \ \partial_x \log(n) \ \partial_x \p{D} \  +  \Big(r  \ (1-n) + 1 \Big) \ s \ \p{D} \ (1-\p{D}) \ ( \p{D} - \frac{2s-1}{s} ).
 \end{aligned}
\end{equation}

Combining equations on $n$ and $\p{D}$, we obtain model \eqref{eq:per_zyg_p}.

\newpage

\subsection{Model with partial conversion} 

\subsubsection{Conversion in the zygote}\label{an:rewrite2}

We rewrite model \eqref{eq:par_zyg_nD_nW} with variables:

\begin{equation}
    n = \n{W} + \n{D}, \quad \quad \quad \quad  \p{D} = \dfrac{\n{D}}{\n{W} + \n{D}}.
\end{equation} 

where $n$ is the total population density and $\p{D}$ is the drive allele frequency.\\

\textbf{Equation on $\mathbf{n}$:} \begin{equation}
\begin{aligned}
   \partial_t n - \partial_{xx}^2 n & =  \dfrac{ r \ (1-n)+1 }{n} \ \Big( (1-s) \ \n{D}^2 + [ 2 \ c \ (1-s)  + 2 \ (1-sh) \ (1-c) ] \  \n{D} \n{W} + \n{W}^2 \Big) - n, \\
   & =  \dfrac{ r \ (1-n)+1 }{n} \ \Big( (1-s) \ (n\p{D})^2 + [ 2 \ c \ (1-s)  + 2 \ (1-sh) \ (1-c) ] \  \p{D} (1-\p{D}) n^2 + (1-\p{D})^2 n^2 \Big) - n, \\
   &  =  \big( r \ (1-n)+1 \big) \ \Big( (1-s) \ \p{D}^2 + [ 2 \ c \ (1-s)  + 2 \ (1-sh) \ (1-c) ] \  \p{D} \ (1-\p{D}) + (1-\p{D})^2 \Big) n - n,\\
    & =  \big( r \ (1-n)+1 \big) \ \Big( (1-s) \ \p{D}^2 + 2 \ \p{D}  \ (1-\p{D}) \ [ c \ (1-s)  + (1-c) \ (1-sh)  ] + (1-\p{D})^2  \Big) n - n.
\end{aligned}
\end{equation}

\textbf{Equation on $\mathbf{\n{D}}$:} \begin{equation}
\begin{aligned}
 \partial_t \n{D}  & =  \partial_{xx}^2 \n{D} \  + \dfrac{ r \ (1-n)+1 }{n} \Big[ \ (1-s)  (\n{D} +  2  \  c  \ \n{W} ) +   (1-sh)   \  (1-c)  \   \n{W}  \Big] \n{D} - \n{D},  \\
&  = \partial_{xx}^2 \n{D} \   + \dfrac{ r \ (1-n)+1 }{n} \Big[ \ (1-s)  ( n \ \p{D}  +  2  \  c  \  n \ (1-\p{D}) ) +   (1-sh)   \  (1-c)  \    n \ (1-\p{D})  \Big] \  n \ \p{D}  - n \ \p{D}, \\
 & \stackrel{\footnotemark}{=}
 \p{D} \ \partial_{xx}^2 n + 2 \ \partial_x n \ \partial_x \p{D} +  n \ \partial_{xx}^2 \p{D}  \    \\
   & \quad \quad + \big( r \ (1-n)+1 \big) \Big[ \ (1-s)  ( \p{D}  +  2  \  c  \ (1-\p{D}) ) +   (1-sh)   \  (1-c)  \ (1-\p{D})  \Big] \  n \ \p{D} - n \ \p{D}.  \\
\end{aligned}
\end{equation}

\footnotetext{$\partial_{xx}^2 \n{D} = \partial_{xx}^2 n \p{D} =  \partial_{x} ( \p{D} \ \partial_{x} n + n \ \partial_{x} \p{D} ) = \p{D} \ \partial_{xx}^2  n + 2 \ \partial_{x} \p{D}  \ \partial_{x} n + n  \ \partial_{xx}^2  \p{D} $} 


\textbf{Equation on $\mathbf{\p{D}= \dfrac{\n{D}}{n}}$:}\begin{equation}
\begin{aligned}
    \partial_t \p{D} & = \dfrac{ n \ (\partial_t \n{D}) - \n{D} \ (\partial_t n)}{n^2} = \dfrac{ n \ (\partial_t \n{D}) - n \ \p{D} \ (\partial_t n)}{n^2} = \dfrac{1}{n} \big( \ \partial_t \n{D} - \p{D} \ (\partial_t n) \big), \\
    & = \dfrac{1}{n} \Big[ 2 \ \partial_x n \ \partial_x \p{D} +  n \ \partial_{xx}^2 \p{D}  \  + \big( r \ (1-n)+1 \big) \Big( \ (1-s)  ( \p{D}  +  2  \  c  \ (1-\p{D}) ) +   (1-sh)   \  (1-c)  \ (1-\p{D})   \\
    & \quad \quad  -  (1-s) \ \p{D}^2 - 2 \ \p{D}  \ (1-\p{D}) \ [ c \ (1-s)  + (1-c) \ (1-sh)  ] - (1-\p{D})^2  \Big) \ \p{D} \ n  \Big], \\
     & =  2 \ \partial_x \log(n) \ \partial_x \p{D} +  \partial_{xx}^2 \p{D}   + \big( r \ (1-n)+1 \big) \Big( \ \p{D} \ (1-s)  \  (1-\p{D})  +   2  \ (1-s) \   c  \ (1-\p{D})  +   (1-sh)   \  (1-c)  \ (1-\p{D})   \\
   & \quad \quad   - 2 \ \p{D}   \  c \ (1-s) \ (1-\p{D})  - 2 \ \p{D}   \ (1-c) \ (1-sh) \ (1-\p{D})  - (1-\p{D})^2  \Big) \ \p{D}   \Big], \\
  & =  2 \ \partial_x \log(n) \ \partial_x \p{D} +  \partial_{xx}^2 \p{D}  \  + \big( r \ (1-n)+1 \big) \Big( \p{D} \big[ (1-s) \ (1-2c) -  2 \ (1-sh)\   (1-c)   + 1 \big] + 2 \  (1-s) \ c \\
& \quad \quad + (1-sh) \ (1-c) - 1   \Big) \ (1-\p{D}) \ \p{D} ,  \\
   & =  2 \ \partial_x \log(n) \ \partial_x \p{D} +  \partial_{xx}^2 \p{D}  \\
   & \quad \quad + \big( r \ (1-n)+1 \big) \Big( [1- 2(1-c)(1-h)] s \p{D} - s [ 1-(1-c)(1-h)] + c (1-s) \Big) \ (1-\p{D}) \ \p{D} .  \\
\end{aligned}
\end{equation}

Combining equations on $n$ and $\p{D}$, we obtain model \eqref{eq:par_zyg_p}.

  

\newpage

\subsubsection{Conversion in the germline}\label{an:rewrite3}

We rewrite model \eqref{eq:par_ger_nD_nW} with variables:

\begin{equation}
    n = \n{W} + \n{D}, \quad \quad \quad \quad  \p{D} = \dfrac{\n{D}}{\n{W} + \n{D}}.
\end{equation} 

where $n$ is the total population density and $\p{D}$ is the drive allele frequency.\\

\textbf{Equation on n:} 
\begin{equation}
\begin{aligned}
    \partial_t n - \partial_{xx}^2 n & =  \dfrac{ r \ (1-n)+1 }{n} \ \Big( (1-s) \ \n{D}^2 +  2 \ (1-sh) \  \n{D} \n{W} + \n{W}^2 \Big) - n, \\
     & =  \dfrac{ r \ (1-n)+1 }{n} \ \Big( (1-s) \ (n\p{D})^2 + 2 \ (1-sh) \  \p{D}  (1-\p{D}) n^2 + ((1-\p{D})n)^2 \Big) - n, \\
    & =  \big( r \ (1-n)+1 \big) \ \Big( (1-s) \ \p{D}^2 + 2 \ (1-sh) \  \p{D} \ (1-\p{D}) + (1-\p{D})^2 \Big) n - n. \\
\end{aligned}
\end{equation}

\textbf{Equation on $\mathbf{\n{D}}$:} 

\begin{equation}
\begin{aligned}
\partial_t \n{D} & =  \partial_{xx}^2 \n{D} \  + \dfrac{ r \ (1-n)+1 }{n} \Big[ \ (1-s) \n{D}  +   (1-sh)   \  (1+c)  \   \n{W}  \Big] \n{D} - \n{D},  \\
&  =  \p{D} \ \partial_{xx}^2 n + 2 \ \partial_x n \ \partial_x \p{D} +  n \ \partial_{xx}^2 \p{D}  \  + \dfrac{ r \ (1-n)+1 }{n} \Big[ \ (1-s) n \p{D} +   (1-sh)   \  (1+c)  \  n \ (1-\p{D})  \Big] \  n \ \p{D}  - n \ \p{D}, \\
& \stackrel{\footnotemark}{=}  \p{D} \ \partial_{xx}^2 n + 2 \ \partial_x n \ \partial_x \p{D} +  n \ \partial_{xx}^2 \p{D}  \  + \Big( r \ (1-n)+1 \Big) \Big[ \ (1-s) \p{D} +   (1-sh)   \  (1+c)  \   (1-\p{D})  \Big] \  n \ \p{D}  - n \ \p{D}.  \\
\end{aligned}
\end{equation}

\footnotetext{$\partial_{xx}^2 \n{D} = \partial_{xx}^2 n \p{D} =  \partial_{x} ( \p{D} \ \partial_{x} n + n \ \partial_{x} \p{D} ) = \p{D} \ \partial_{xx}^2  n + 2 \ \partial_{x} \p{D}  \ \partial_{x} n + n  \ \partial_{xx}^2  \p{D} $} 

\textbf{Equation on $\mathbf{\p{D}= \dfrac{\n{D}}{n}}$:}



\begin{equation}
\begin{aligned}
    \partial_t \p{D} & = \dfrac{ n \ (\partial_t \n{D}) - \n{D} \ (\partial_t n)}{n^2} = \dfrac{ n \ (\partial_t \n{D}) - n \ \p{D} \ (\partial_t n)}{n^2} = \dfrac{1}{n} \big( \ \partial_t \n{D} - \p{D} \ (\partial_t n) \big), \\
  & = \dfrac{1}{n} \Big[ 2 \ \partial_x n \ \partial_x \p{D} +  n \ \partial_{xx}^2 \p{D}  \  + \big( r \ (1-n)+1 \big) \Big( \   (1-s) \  \p{D} +   (1-sh)   \  (1+c)  \   (1-\p{D})    \\
   &  \quad \quad  - (1-s) \ \p{D}^2 - 2 \ (1-sh) \  \p{D} \ (1-\p{D}) - (1-\p{D})^2 \Big) \  \p{D} \  n  \ \Big], \\
   & =  2 \ \partial_x \log(n) \ \partial_x \p{D} +   \partial_{xx}^2 \p{D}  \  + \big( r \ (1-n)+1 \big) \Big( (1-s) \ \p{D} + (1-sh)(1+c) - 2 \ (1-sh) \ \p{D} - (1-\p{D}) \Big) \ \p{D} \ (1-\p{D}), \\
   & =  2 \ \partial_x \log(n) \ \partial_x \p{D} +   \partial_{xx}^2 \p{D}  \  + \big( r \ (1-n)+1 \big) \Big( (2h-1) \  s \ \p{D} + (1-sh) (1+c) - 1  \Big) \ \p{D} \ (1-\p{D}) .
\end{aligned}
\end{equation}

Combining equations on $n$ and $\p{D}$, we obtain model \eqref{eq:par_ger_p}.

\newpage


\section{Proofs for model \eqref{eq:per_zyg_p} with perfect conversion in the zygote} \label{ann:perfect}

\subsection{Numerical evidence for the continuity when $r \rightarrow 0$} \label{ann:continuity}


In Figure \ref{fig:continuity}, we plot the speed of the traveling wave solutions of  model \eqref{eq:per_zyg_p} for a range of $r$ and $s$ values, and for $r=0$. A positive speed correspond to drive invasion.

\begin{figure}[H]
        \centering
     \includegraphics[scale = 0.7]{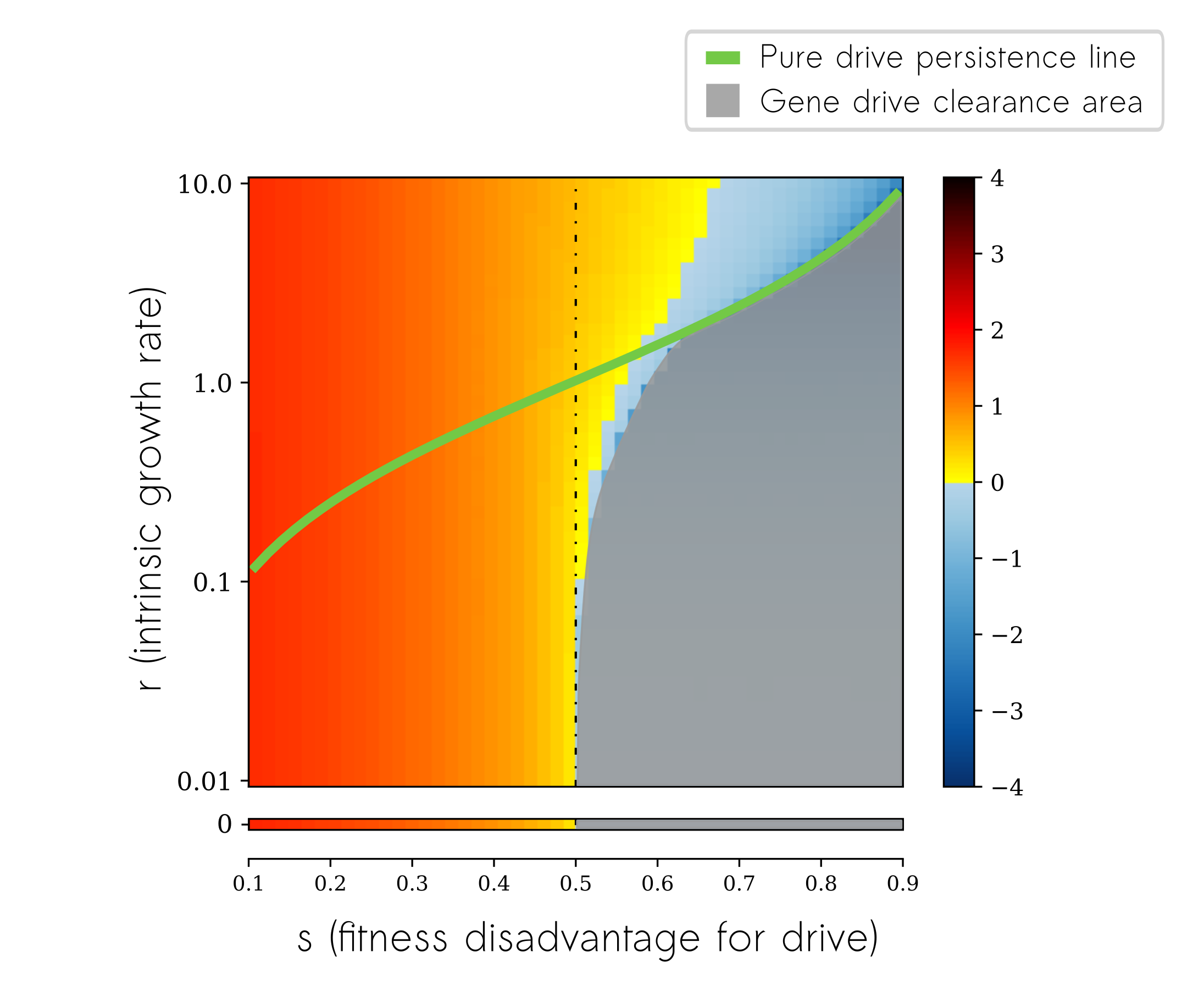}
        \caption{Wave speed values in model with perfect conversion in the zygote (\ref{eq:per_zyg_p}), 
        regarding parameters $r$ the intrinsic growth rate (log scale in between $0.01$ and $10$, plus the exact value $r=0$ in the bottom color line) and $s$ the fitness disadvantage for drive (normal scale). Below the pure drive persistence line (light green), a well-mixed population containing only drive homozygous individuals will necessarily go extinct.}
        \label{fig:continuity}
\end{figure}

We observe continuity in the speed value when $r \rightarrow 0$ away from $s = \frac{1}{2}$, meaning that the case $r=0$ is relevant to approximate very small intrinsic growth rates.




\newpage

\subsection{Proof of the statements in Tables \ref{tab:per_zyg_rinf} and \ref{tab:per_zyg_r0} when perfect conversion occurs in the zygote} \label{ann:tab_perfect}

In this section we prove the statements of Tables \ref{tab:per_zyg_rinf} and \ref{tab:per_zyg_r0} on the two models of interest:


\underline{$\mathbf{r = \infty}$}
        \begin{equation} \label{ann:sys_rinf}
        \partial_t p - \partial_{xx}^2 p =  \dfrac{s \ p \  (1-p) \ \big(p- \dfrac{2s -1}{s} \big)}{1-s+s(1-p)^2} = f^{\infty}(p).
        \end{equation}

\underline{$\mathbf{r = 0}$} 
\begin{equation} \label{ann:sys_r0}
\left\{
    \begin{array}{ll}
    \partial_t \n{DD} - \partial_{xx}^2 \n{DD} & = \ (1-s)  \ \dfrac{\n{WW}  \n{DD}}{\n{WW}+\n{DD}} - s \ \n{DD} = f^0(\n{DD}, \n{WW}) \\
    \\
     \partial_t \n{WW} - \partial_{xx}^2 \n{WW} & = \  \dfrac{- \ \n{WW} \n{DD}}{\n{WW}+\n{DD}}.
    \end{array}
\right.
\end{equation}

\vspace{0.2cm}

\textbf{Monostable / Bistable}

\begin{itemize}
    \item[] \underline{$\mathbf{r = \infty}$}  

    \begin{minipage}[t]{0.15\linewidth}
     $0<s<0.5$ 
    \end{minipage} \hfill 
    \begin{minipage}[t]{0.85\linewidth}
        The equation admits two admissible steady states $0$ and $1$. 
        
        As $(f^{\infty})'(0)>0$ and $(f^{\infty})'(1)<0$, the only stable state is $p=1$.
    \end{minipage}

    \begin{minipage}[t]{0.15\linewidth}
     $0.5<s<1$ 
    \end{minipage} \hfill 
    \begin{minipage}[t]{0.85\linewidth}
       The equation admits three admissible steady states $0$, $\dfrac{2s-1}{s}$ and $1$. 
       
       As $(f^{\infty})'(0)<0$, $(f^{\infty})'(\frac{2s-1}{s})>0$ and $(f^{\infty})'(1)<0$, both $p=0$ and $p=1$ are stable states.
    \end{minipage}

\item[] \underline{$\mathbf{r = 0}$} 

    \begin{minipage}[t]{0.15\linewidth}
     $0<s<0.5$ 
    \end{minipage} \hfill 
    \begin{minipage}[t]{0.85\linewidth}
        The system admits $(\n{DD}=0, \n{WW} \in[0,1])$ as admissible steady states. The Jacobian matrix, when switching to $n$ and $\p{D}$ variables, indicates that the only stable state is $(n=0, \p{D}=1)$, i.e. $(\n{DD}=0, \n{WW}=0)$.
    \end{minipage}

    \begin{minipage}[t]{0.15\linewidth}
     $0.5<s<1$ 
    \end{minipage} \hfill 
    \begin{minipage}[t]{0.85\linewidth}
       The system admits $(\n{DD}=0, \n{WW} \in[0,1])$ as admissible steady states. The Jacobian matrix, when switching to $n$ and $\p{D}$ variables, indicates that the stable states are $(n=0, \p{D}=1)$ and $(n \in[0,1], \p{D}=0)$, i.e. $(\n{DD}=0, \n{WW} \in[0,1])$.
    \end{minipage}

\end{itemize}

\vspace{0.2cm}

\textbf{Existence of critical traveling waves}

\begin{itemize}

\item[] \underline{$\mathbf{r = \infty}$} 

\begin{itemize}
\item[] The existence of traveling waves for the scalar equation \eqref{eq:per_zyg_rinf} in both monostable and bistable cases is a classical result in the theory of reaction-diffusion equations, see for instance the seminal works in \cite{aronson1975, aronson1978}. 
\end{itemize}

\vspace{0.2cm}

\item[] \underline{$\mathbf{r = 0}$} 

\begin{minipage}[t]{0.15\linewidth}
     $0<s<0.5$ 
    \end{minipage} \hfill 
    \begin{minipage}[t]{0.85\linewidth}
       We apply the results of Appendix \ref{ann:exist} with $\beta_1 = 1$ and $\beta_2 = 1 - s$. Therefore system \eqref{ann:sys_r0} admits a traveling wave when $0<s<0.5$. 
\end{minipage}

\begin{minipage}[t]{0.15\linewidth}
     $0.5<s<1$ 
    \end{minipage} \hfill 
    \begin{minipage}[t]{0.85\linewidth} There is no drive propagation due to the gene drive clearance: the drive allele density decreases uniformly in space (details in section \ref{ann:gd_clearance}). Regarding the wild-type alleles, their dynamic is given by the heat equation implying only diffusion and no growth. It cannot admit traveling wave solutions.
    \end{minipage}

 \end{itemize}
 
 \vspace{0.2cm}

\textbf{Pulled/pushed waves and speed values}

For both models, the  speed of the linearized problem around zero density of drive allele is given by $ 2 \sqrt{1-2s} = 2 \sqrt{(f^{\infty})'(0)} = 2 \sqrt{ \partial_p f^0(0,1)}$.

 \begin{itemize}

\item[] \underline{$\mathbf{r = \infty}$} 


    \begin{minipage}[t]{0.15\linewidth}
     $0<s \lesssim 0.35$ 
    \end{minipage} \hfill 
    \begin{minipage}[t]{0.85\linewidth}
      Numerically, we observe that the speed of the wave is equal to the minimal speed of the linearized problem: the wave is pulled (detail in section \ref{ann:num_threshold})
    \end{minipage}

    \begin{minipage}[t]{0.15\linewidth}
     $0.35 \lesssim s<0.5$ 
    \end{minipage} \hfill 
    \begin{minipage}[t]{0.85\linewidth}
      Numerically, we observe that the speed of the wave is strictly above the minimal speed of the linearized problem: the wave is pushed (detail in section \ref{ann:num_threshold}).
    \end{minipage}

    \begin{minipage}[t]{0.15\linewidth}
     $0.5<s<1$ 
    \end{minipage} \hfill 
    \begin{minipage}[t]{0.85\linewidth}
        As the system is bistable, the wave is necessarily pushed. The numerical approximation $s \approx 0.70$ indicating whether the drive of the wild-type population will invade the environment was already determined in the work of Tanaka et al \cite{tanaka2017}. 
    \end{minipage}

\item[] \underline{$\mathbf{r = 0}$} 

    \begin{minipage}[t]{0.15\linewidth}
     $0<s<0.5$ 
    \end{minipage} \hfill 
    \begin{minipage}[t]{0.85\linewidth}
     We apply the results of Appendix \ref{ann:exist} with $\beta_1 = 1$ and $\beta_2 = 1 - s$. Therefore system \eqref{ann:sys_r0} admits a traveling wave with speed $v = 2 \sqrt{1-2s}$ when $0<s<0.5$. This value corresponds to the KPP speed, by definition the wave is pulled.
    \end{minipage}

    \begin{minipage}[t]{0.15\linewidth}
     $0.5 < s < 1 $ 
    \end{minipage} \hfill 
    \begin{minipage}[t]{0.85\linewidth}
       No wave (see above, in Existence of critical traveling waves). 
    \end{minipage}
 \end{itemize}


\vspace{1cm}

\subsubsection{Gene drive clearance for $s \in (0.5,1)$ when $r = 0$}\label{ann:gd_clearance}



Consider the model with perfect conversion in the zygote (\ref{eq:per_zyg}). The densities $\n{WW}$ and $\n{DD}$ dynamics are qualitatively given in Figure \ref{fig:per_zyg_densities_r0} for $r=0$. 

\begin{figure}[H]
\centering
\begin{subfigure}{0.48\textwidth}
    \centering
    \caption{Spreading eradication drive when $0<s<0.5$.}
    \includegraphics[scale = 0.7]{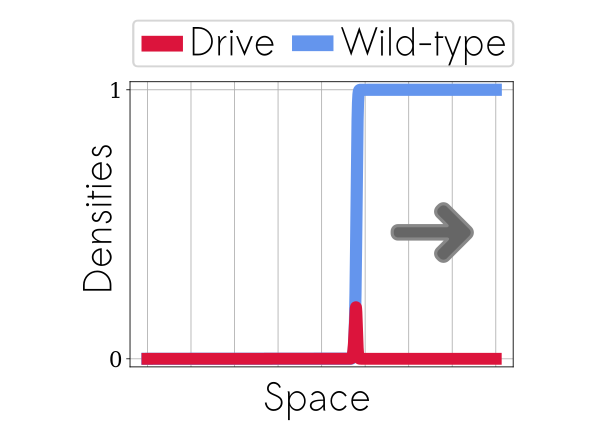}
    \label{subfig:eradication_drive_r0}
\end{subfigure}
\hfill
\begin{subfigure}{0.48\textwidth}
    \centering
    \caption{Gene drive clearance when $0.5<s<1$}
    \includegraphics[scale = 0.7]{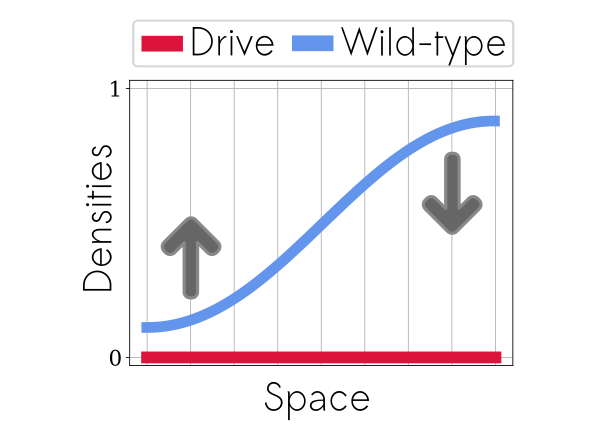}
    \label{subfig:gene_drive_clearance_r0}
\end{subfigure}
\caption{Qualitative dynamics of the drive homozygotes density $\n{DD}$ (red line) and the wild-type homozygotes density $\n{WW}$ (blue line) in space.}
\label{fig:per_zyg_densities_r0}
\end{figure}

When $s>0.5$ and $r=0$, we observe gene drive clearance (in Figure \ref{subfig:gene_drive_clearance_r0}). More precisely, we have the following estimate, deduced from \eqref{eq:per_zyg_r0}:
\begin{equation}
\partial_t \n{DD} - \partial_{xx}^2 \n{DD} \leq \ (1-2s) \ \n{DD},
\end{equation}
Therefore, $\n{DD}$ is exponentially decaying in time, uniformly in space. The dynamics of the wild type then boil down to the standard heat equation, there cannot exist a traveling wave.


\vspace{1cm}

\subsubsection{Numerical approximation of $s$ threshold value for the pulled/pushed wave when $r=+ \infty$}\label{ann:num_threshold}

In order to determine an approximation of the threshold value at which the wave switches from a pulled wave to a pushed wave, we used the recent continuation procedure published in \cite{avery2022}. Figure \ref{fig:continuation-Holzer} presents the value of the wave speed obtained via the latter continuation numerical scheme \cite{holzer2022}, for a wide range of $s$ values. Notice the transition between pulled fronts (plain red) and pushed fronts (plain green). For the sake of clarity, the value of the minimal speed of the linearized problem $v = 2 \sqrt{1-2s}$ is shown in red for $s\in (0,\frac12)$. Notice that the speed of the pushed front changes sign approximately at $s\approx 0.70$, in agreement with the theoretical criterion \eqref{eq:integral}. 

\begin{figure}[H]
        \centering
     \includegraphics[scale = 0.3]{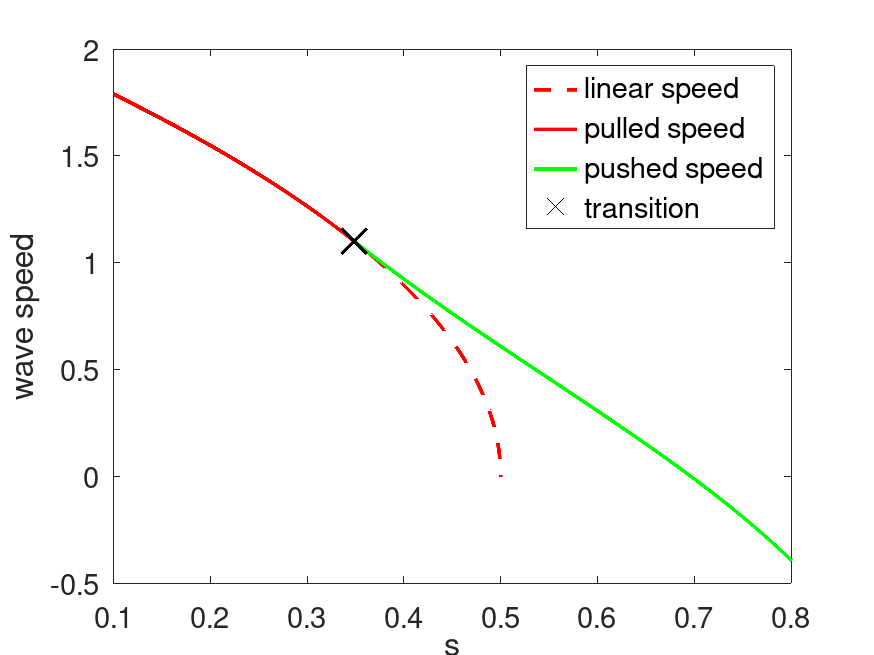}
        \caption{Value of the wave speed when $r=+ \infty$ obtained via the numerical scheme in \cite{holzer2022}, for a wide range of $s$ values. The transition between pulled fronts (plain red) and pushed fronts (plain green) is approximately $0.35$. For the sake of clarity, the value of the minimal speed of the linearized problem $v = 2 \sqrt{1-2s}$ is shown in red for $s\in (0,\frac12)$.}
        \label{fig:continuation-Holzer}
\end{figure}

\newpage

\section{Critical traveling wave for an SI similar model.} \label{ann:exist}

Consider the following epidemiological model, 

\begin{equation}\label{ann:si}
\left\{
    \begin{array}{ll}
        \partial_t S - \partial_{xx}^2 S & =  \ - \beta \ \dfrac{\ S \ I}{S + I},\\
        \\
       \partial_t I - \partial_{xx}^2 I & = \ \beta \ \dfrac{ S \ I}{S + I} - \gamma I.\\
    \end{array}
\right.
\end{equation}
where $S$ is the density of susceptible individuals, $I$ is the density of infected individuals, $\gamma$ is the mortality of infected individuals and $\beta$ is the transmission coefficient. This model has already been studied in the literature, see \cite{zhou2019} and references therein. In particular, the existence of a minimal traveling wave has been established in the latter reference.

Models \eqref{eq:per_zyg_r0}, \eqref{eq:par_zyg_r0} and \eqref{eq:par_ger_r0} are very similar to the above SI model \eqref{ann:si}, except that the coefficient $\beta$ is different in the first and the second equation of the system. We write this new system with two coefficients $\beta_1, \beta_2$:

 \begin{equation}\label{ann:simsi}
    \left\{
    \begin{array}{ll}
        \partial_t S - \partial_{xx}^2 S & =  \ - \beta_1 \ \dfrac{ \ S \ I}{S + I}, \\
        \\
       \partial_t I - \partial_{xx}^2 I & = \ \beta_2 \ \dfrac{ S \ I}{S + I} - \gamma I.
    \end{array}
    \right.
\end{equation}


\subsection{Existence of critical traveling wave solutions}

We are able to establish the following Theorem by adapting the proof in \cite{zhou2019}. 

\begin{theorem}
Suppose that $\beta_1>0$, and $\beta_2 > \gamma$, then system (\ref{ann:simsi}) admits a positive and bounded traveling wave solution with profile $(S^*, I^*)$, and speed $v = 2 \sqrt{\beta_2 - \gamma}$. Furthermore, both $S^*$ and $I^*$ are positive, and bounded by $1$ and $\frac{\beta_2 - \gamma}{\gamma}$ respectively. 
\end{theorem}

By adapting further the elements of \cite{zhou2019}, it would be possible to prove that the profile $S^*$ is increasing, whereas the profil $I^*$ is unimodal, as shown in Figure \ref{fig:cartoon S*I*}.

\begin{figure}[H]
        \centering
     \includegraphics[scale = 0.4]{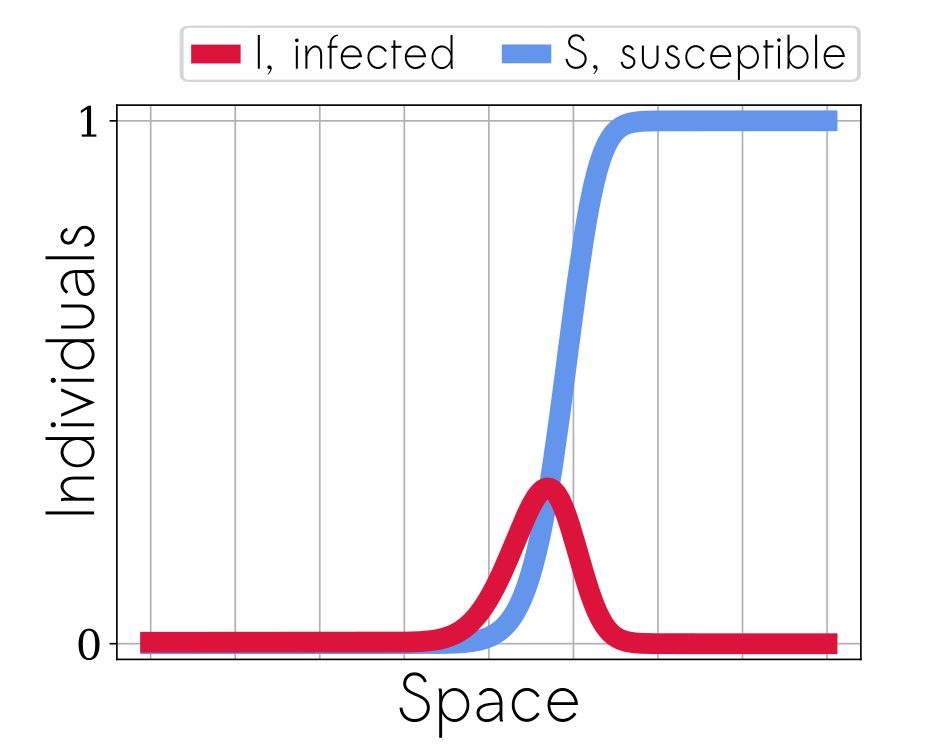}
        \caption{Qualitative shape of the solution ($I^*$ in red, $S^*$ in blue).}
        \label{fig:cartoon S*I*}
\end{figure}

\subsection{Proof of the theorem}

We proceed as follows: 

\begin{enumerate}
    \item Although the system does not satisfy the comparison principle due to a lack of monotonicity, the construction of traveling waves is performed through a construction of sub-solutions ($\textcolor{Cerulean}{\underline{S}}$, $\textcolor{orange}{\underline{I}}$) and super-solutions ($\textcolor{Orchid}{\bar{S}}$,$\textcolor{RedOrange}{\bar{I}}$) for the system.
    \vspace{0.2cm}
    \item Using Schauder’s fixed point theorem, we prove the existence of a critical traveling wave solution $(S^*, I^*)$ with speed $v$ such that $\textcolor{Cerulean}{\underline{S}}(z) \leq S^*(z) \leq \textcolor{Orchid}{\bar{S}}(z)$ and  $\textcolor{orange}{\underline{I}}(z) \leq I(z) \leq \textcolor{RedOrange}{\bar{I}}(z)$ for all  $z$ in $\mathbb{R}$.
    \vspace{0.2cm}
    \item Finally, we conclude with the positivity of the critical traveling wave solution thanks to the strong maximum principle.
\end{enumerate} 

\subsubsection{Construction of sub- and super-solutions}\label{an:sub_super}



We are seeking sub- and super-solutions, respectively ($\textcolor{Cerulean}{\underline{S}}$, $\textcolor{orange}{\underline{I}}$), ($\textcolor{Orchid}{\bar{S}}$,$\textcolor{RedOrange}{\bar{I}}$). Because of the non-monotonic coupling in the system, the following set of cross-relationships must be satisfied:
\begin{enumerate}
    \item $ -v  \ \textcolor{Orchid}{\textcolor{Orchid}{\bar{S}}}'  - \textcolor{Orchid}{\bar{S}}'' \geq - \beta_1  \ \dfrac{\textcolor{Orchid}{\bar{S}} \ I}{\textcolor{Orchid}{\bar{S}}+I}\quad \forall \textcolor{orange}{\underline{I}} \leq I \leq \textcolor{RedOrange}{\bar{I}} $;
    \item  $ -v  \ \textcolor{Cerulean}{\underline{S}}'  - \textcolor{Cerulean}{\underline{S}}'' \leq - \beta_1  \dfrac{\textcolor{Cerulean}{\underline{S}} \ I}{\textcolor{Cerulean}{\underline{S}}+I} \quad \forall \textcolor{orange}{\underline{I}} \leq I \leq \textcolor{RedOrange}{\bar{I}}  $;
    \item $ -v \ \textcolor{RedOrange}{\bar{I}}'  - \textcolor{RedOrange}{\bar{I}}'' \geq \beta_2 \  \dfrac{S \ \textcolor{RedOrange}{\bar{I}}}{S + \textcolor{RedOrange}{\bar{I}}} - \gamma \ \textcolor{RedOrange}{\bar{I}} \quad \forall \textcolor{Cerulean}{\underline{S}}  \leq S \leq \textcolor{Orchid}{\bar{S}}  $;
    \item  $ -v  \ \textcolor{orange}{\underline{I}}'  - \textcolor{orange}{\underline{I}}'' \leq \beta_2 \  \dfrac{S \  \textcolor{orange}{\underline{I}}}{S + \textcolor{orange}{\underline{I}}} - \gamma \ \textcolor{orange}{\underline{I}}\quad \forall \textcolor{Cerulean}{\underline{S}}  \leq S \leq \textcolor{Orchid}{\bar{S}}  $. 
\end{enumerate}

\vspace{0.2cm}

\lk{Inequalities 1,2,3 and 4 are valid in a weak sense. As $\textcolor{Cerulean}{\underline{S}}$ is a piece-wise differentiable function, the quantity $\textcolor{Cerulean}{\underline{S}}$ consists of functions on each sub-interval with a Dirac mass at the point of $\mathcal{C}^1$ discontinuity. However, the Dirac mass has the good sign in this case (the transition has a convex shape), hence the second derivative $\textcolor{Cerulean}{\underline{S}}''$ is non-negative in the sense of a measure. All signs are correct: $\textcolor{orange}{\underline{I}}$ has a convex transition also, and $\textcolor{RedOrange}{\bar{I}}$ has a concave transition. This key point is equivalent to the standard principle in the theory of parabolic equations (widely used for reaction-diffusion equations involving comparison techniques): "the maximum of sub-solutions is a sub-solution" (here, $\textcolor{Cerulean}{\underline{S}}$, $\textcolor{orange}{\underline{I}}$) and "the minimum of super-solutions is a super-solution" (here, $\textcolor{Orchid}{\bar{S}}$, $\textcolor{RedOrange}{\bar{I}}$).} \\

To define our sub and super-solutions, it is useful to introduce the the following family of functions $\mathcal I(z) = e^{- \lambda^* z}$, where $\lambda^*$ is solution of the following dispersion equation:
\begin{equation}\label{eq:charac}
   (\lambda^*)^2   - v \lambda^*   +  ( \beta_2 - \gamma )  = 0.
\end{equation}
They are solutions of the linearized problem 
\begin{equation}
    v \mathcal I'  + \mathcal I''   +  ( \beta_2 - \gamma )\  \mathcal I   = 0.
\end{equation}
For the critical speed $v = 2 \sqrt{(\beta_2 - \gamma)} $, the corresponding double root is $ \lambda^* = \dfrac{v}{2} = \sqrt{(\beta_2 - \gamma)} $.

\vspace{0.2cm}

\begin{lemma}
There exist two large enough constants $L_1 > 0$ and $L_2 > 0$, such that the functions $ \textcolor{Orchid}{\bar{S}}$, $\textcolor{Cerulean}{\underline{S}}$ $\textcolor{RedOrange}{\bar{I}}$, $ \textcolor{orange}{\underline{I}} $ defined below satisfy the conditions 1. 2. 3. and 4.:

\begin{minipage}{0.48\linewidth}
\begin{equation}
    \textcolor{Orchid}{\bar{S}} = 1.
    \vphantom{    \begin{array}{ll}
       0  \quad \quad & \forall z \leq  - L_1 \log (\dfrac{1}{L_1}), \\
       \\
       1 - L_1 e^{- \dfrac{z}{L_1} }  \quad \quad & \forall z > - L_1 \log (\dfrac{1}{L_1}). \\
    \end{array}
}
\end{equation}
\begin{equation}
    \textcolor{Cerulean}{\underline{S}} = \left\{
    \begin{array}{ll}
       0  \quad \quad & \forall z \leq L_1 \log (L_1), \\
       \\
       1 - L_1 e^{- \dfrac{z}{L_1} }  \quad \quad & \forall z > L_1 \log (L_1). \\
    \end{array}
    \right.
\end{equation}
\end{minipage} \hfill 
\begin{minipage}{0.48\linewidth}
    \begin{equation}
    \textcolor{RedOrange}{\bar{I}} = \left\{
    \begin{array}{ll}
       M  \quad \quad & \forall z \leq \dfrac{1}{\lambda^*}, \\
       \\
      e M \lambda^* z e^{- \lambda^* z}  \quad \quad & \forall z > \dfrac{1}{\lambda^*} . \\
    \end{array}
    \right.
\end{equation}
\begin{equation}
    \textcolor{orange}{\underline{I}} = \left\{
    \begin{array}{ll}
      0  \quad \quad & \forall z \leq \big( \dfrac{L_2}{e M \lambda^*} \big) ^2 ,\\
      \\
     \big( e M \lambda^* z - L_2 \sqrt{z} \big) e^{-\lambda^* z} \quad \quad & \forall z > \big( \dfrac{L_2}{e M \lambda^*} \big) ^2 .  
    \end{array}
    \right.
\end{equation}
\end{minipage}

with $M = \dfrac{\beta_2 - \gamma}{\gamma} = \dfrac{(\lambda^*)^2}{\gamma} $.
\end{lemma}

\begin{proof}
Before we proceed with the proof, we introduce the following set of conditions, which are more restrictive than 1,2,3,4, but may appear more useful at some point in the calculations. When satisfied, they clearly imply 1,2,3,4. 

\begin{enumerate}[label=(\roman*)]
    \item $ -v  \ \textcolor{Orchid}{\bar{S}}'  - \textcolor{Orchid}{\bar{S}}'' \geq 0 $; 
    
    $ \quad $ 
     
    \item $ - v \ \textcolor{Cerulean}{\underline{S}}'  - \textcolor{Cerulean}{\underline{S}}'' \leq - \beta_1  \ \textcolor{RedOrange}{\bar{I}}   $ ;
    
   $ \quad $ 
    
    \item $  -v \ \textcolor{RedOrange}{\bar{I}}'  - \textcolor{RedOrange}{\bar{I}}'' \geq   (\beta_2 -  \gamma) \  \textcolor{RedOrange}{\bar{I}} $ ;

   $ \quad $ 
   
    \item  $ -v  \ \textcolor{orange}{\underline{I}}'  - \textcolor{orange}{\underline{I}}'' \leq \beta_2 \  \dfrac{\textcolor{Cerulean}{\underline{S}} \  \textcolor{orange}{\underline{I}}}{\textcolor{Cerulean}{\underline{S}} + \textcolor{orange}{\underline{I}}} - \gamma \ \textcolor{orange}{\underline{I}} $.
\end{enumerate}

We verify each of the four conditions on $ \textcolor{Orchid}{\bar{S}}$, $\textcolor{Cerulean}{\underline{S}}$ $\textcolor{RedOrange}{\bar{I}}$, $ \textcolor{orange}{\underline{I}} $:

\underline{Condition 1}: $ -v  \ \textcolor{Orchid}{\textcolor{Orchid}{\bar{S}}}'  - \textcolor{Orchid}{\bar{S}}'' \geq - \beta_1  \ \dfrac{\textcolor{Orchid}{\bar{S}} \ I}{\textcolor{Orchid}{\bar{S}}+I} $.

The constant function $ \textcolor{Orchid}{\bar{S}} = 1 $ satisfies the more restrictive condition (i) $  -v  \ \textcolor{Orchid}{\bar{S}}'  - \textcolor{Orchid}{\bar{S}}'' \geq 0 $ .

\vspace{0.2cm}

\underline{Condition 2}: $\textcolor{Cerulean}{\underline{S}}'  - \textcolor{Cerulean}{\underline{S}}'' \leq - \beta_1  \dfrac{\textcolor{Cerulean}{\underline{S}} \ I}{\textcolor{Cerulean}{\underline{S}}+I} $.

Let us take $L_1$ sufficiently large such that $ \dfrac{1}{\lambda^*} <  L_1 \log(L_1)$.

\begin{itemize}
    \item[$\bullet$] For $ z > L_1 \log(L_1) $:
\end{itemize}

Since $ \textcolor{RedOrange}{\bar{I}} = e \ M \  \lambda^*  \  z  \  e^{\lambda^*z}$ and  $ \ \  \textcolor{Cerulean}{\underline{S}} = 1 - L_1 e^{- \dfrac{z}{L_1} }$, the condition (ii) $ - v \ \textcolor{Cerulean}{\underline{S}}'  - \textcolor{Cerulean}{\underline{S}}'' \leq - \beta_1  \ \textcolor{RedOrange}{\bar{I}}   $ holds for a sufficiently large $L_1 > 0 $: \begin{align}
- v \ \textcolor{Cerulean}{\underline{S}}'  - \textcolor{Cerulean}{\underline{S}}''
 =  (\dfrac{1}{L_1} - v ) \  e^{- \dfrac{z}{L_1}} & \leq   -  \beta_1  \ e \ M \  \lambda^* \  z  \  e^{- \lambda^*z} = - \beta_1  \ \textcolor{RedOrange}{\bar{I}}, \\
    \iff \beta_1  \ e \ M \  \lambda^* \  z  \  e^{ ( \dfrac{1}{L_1} - \lambda^* ) z} & \leq   ( v - \dfrac{1}{L_1}  ) .
\end{align} 

\begin{itemize}
    \item[$\bullet$] For $ z \leq  L_1 \log(L_1) $:
\end{itemize}

\vspace{-0.2cm}

The condition 2. is verified since  $\textcolor{Cerulean}{\underline{S}} = 0$.

\vspace{0.2cm}

\underline{Condition 3}: $ -v \ \textcolor{RedOrange}{\bar{I}}'  - \textcolor{RedOrange}{\bar{I}}'' \geq \beta_2 \  \dfrac{S \ \textcolor{RedOrange}{\bar{I}}}{S + \textcolor{RedOrange}{\bar{I}}} - \gamma \ \textcolor{RedOrange}{\bar{I}} $.

\begin{itemize}
    \item[$\bullet$] For $ z \leq  \dfrac{1}{\lambda^*} $:
\end{itemize}

With $\textcolor{RedOrange}{\bar{I}} = M = \dfrac{\beta_2 -  \gamma}{\gamma} $: 

\begin{equation}
     - v  \ \textcolor{RedOrange}{\bar{I}}'  - \textcolor{RedOrange}{\bar{I}}'' = 0 =  \beta_2 \  \dfrac{M}{1+M} - \gamma \ M = \beta_2 \  \dfrac{\textcolor{Orchid}{\bar{S}} \ \textcolor{RedOrange}{\bar{I}}}{\textcolor{Orchid}{\bar{S}} + \textcolor{RedOrange}{\bar{I}}} - \gamma \ \textcolor{RedOrange}{\bar{I}} \geq \beta_2 \  \dfrac{S \ \textcolor{RedOrange}{\bar{I}}}{S + \textcolor{RedOrange}{\bar{I}}} - \gamma \ \textcolor{RedOrange}{\bar{I}} \quad \quad \quad \quad \forall S \leq \textcolor{Orchid}{\bar{S}}.
\end{equation}

\begin{itemize}
    \item[$\bullet$] For $ z > \dfrac{1}{\lambda^*} $:
\end{itemize}

Since $\textcolor{RedOrange}{\bar{I}}$ is proportional to $z  e^{- \lambda^* z}$, and $\lambda^*$ is precisely the double root of the characteristic equation \eqref{eq:charac}, we deduce that condition (iii) $  -v \ \textcolor{RedOrange}{\bar{I}}'  - \textcolor{RedOrange}{\bar{I}}'' -  (\beta_2 -  \gamma) \  \textcolor{RedOrange}{\bar{I}} \geq  0 $ is verified.

\vspace{0.4cm}

\underline{Condition 4}: $ -v  \ \textcolor{orange}{\underline{I}}'  - \textcolor{orange}{\underline{I}}'' \leq \beta_2 \  \dfrac{S \  \textcolor{orange}{\underline{I}}}{S + \textcolor{orange}{\underline{I}}} - \gamma \ \textcolor{orange}{\underline{I}} $. 
Let us take $L_2$ sufficiently large such that $ L_2 >  e M \lambda^* \sqrt{L_1 \log\left(L_1\right)}$.

\begin{itemize}
    \item[$\bullet$] For $ z \leq \Big( \dfrac{L_2}{e \ M \ \lambda^*} \Big)^2 $:
\end{itemize}

$\textcolor{orange}{\underline{I}} = 0$ so condition 4. is satisfied.

\begin{itemize}
    \item[$\bullet$] For $ z > \Big( \dfrac{L_2}{e \ M \ \lambda^*} \Big)^2 $:
\end{itemize}

The choice of $L_2$ implies $z  > L_1 \log(L_1)$, which means $ \textcolor{Cerulean}{\underline{S}} = 1 - L_1 e^{- \dfrac{z}{L_1}}$.

We can reformulate condition (iv) as follows:\begin{equation}\label{eq:(iv)_reformulated}
-v  \ \textcolor{orange}{\underline{I}}'  - \textcolor{orange}{\underline{I}}'' \leq \beta_2 \  \dfrac{\textcolor{Cerulean}{\underline{S}} \  \textcolor{orange}{\underline{I}}}{\textcolor{Cerulean}{\underline{S}} + \textcolor{orange}{\underline{I}}} - \gamma \ \textcolor{orange}{\underline{I}}  \iff  -v  \ \textcolor{orange}{\underline{I}}'  - \textcolor{orange}{\underline{I}}'' - (\beta_2 - \gamma) \ \textcolor{orange}{\underline{I}} \leq - \beta_2 \  \dfrac{  \textcolor{orange}{\underline{I}}^2}{\textcolor{Cerulean}{\underline{S}} + \textcolor{orange}{\underline{I}}}.   
\end{equation}

With $L_3 = e M  \lambda^*$, and:
\begin{align}
&\textcolor{orange}{\underline{I}} = [ L_3 \ z - L_2 \ \sqrt{z} ] \ e^{- \lambda^* z},\\
&
\textcolor{orange}{\underline{I}}' = [ L_3 - L_2 \ \dfrac{1}{2 \ \sqrt{z}} ] \ e^{-\lambda^* z} - [ L_3 z - L_2 \ \sqrt{z} ] \  \lambda^* \ e^{-\lambda^* z},\\
&
\textcolor{orange}{\underline{I}}'' = [ L_2 \ \dfrac{1}{4 \ z \ \sqrt{z}} ] \ e^{- \lambda^* z} - 2 \ [ L_3 - L_2 \ \dfrac{1}{2 \ \sqrt{z}} ] \ \lambda^* \ e^{ - \lambda^* z} +  [ L_3 z - L_2 \ \sqrt{z} ] \  (\lambda^*)^2 \ e^{ - \lambda^* z}
\end{align}

On the one hand, we obtain the following identities: 
\begin{align}
 -v  \ \textcolor{orange}{\underline{I}}'  - \textcolor{orange}{\underline{I}}'' - (\beta_2 - \gamma) \ \textcolor{orange}{\underline{I}}
&
 =  e^{ - \lambda^* z} \Big[ L_3 \Big( -v + v \  \lambda^* \ z + 2  \ \lambda^* - (\lambda^*)^2 \ z - (\beta_2 - \gamma) z  \Big) \\
&
 \quad  + \ L_2 \Big( v \ \dfrac{1}{2 \ \sqrt{z}} - v \ \sqrt{z} \ \lambda^* -  \dfrac{1}{4 \ z \ \sqrt{z}} - \lambda^* \  \dfrac{1}{ \sqrt{z}} + (\lambda^*)^2 \ \sqrt{z} +  (\beta_2 - \gamma) \ \sqrt{z}   \Big) \Big]  ,\\
& =  e^{- \lambda^* z} \Big[ L_3 \Big( ( 2 \lambda^* - v )  - z \ \Big(  - v \lambda^* + (\lambda^*)^2   + (\beta_2 - \gamma) \Big) \Big)\\
&  \quad + \ L_2 \Big(   \sqrt{z} \ \Big( - v \  \lambda^* + (\lambda^*)^2  +  (\beta_2 - \gamma) \Big) + \dfrac{1}{2 \ \sqrt{z}} (  v - 2 \lambda^*) -  \dfrac{1}{4 \ z \ \sqrt{z}} \Big) \Big]   ,\\
& = - L_2 \  e^{- \lambda^* z} \  \dfrac{1}{4 \ z \ \sqrt{z}} 
\end{align}

On the other hand, we have: \begin{equation}
 - \beta_2 \  \dfrac{  \textcolor{orange}{\underline{I}}^2}{\textcolor{Cerulean}{\underline{S}} + \textcolor{orange}{\underline{I}}} = - \beta_2 \  \dfrac{[ L_3 \ z - L_2 \ \sqrt{z} ]^2 \ e^{ - 2  \lambda^* z} }{ 1 - L_1 e^{- \dfrac{z}{L_1}} + [ L_3 \ z - L_2 \ \sqrt{z} ] \ e^{ - \lambda^* z}}.    
\end{equation}

We resume with the reformulation \eqref{eq:(iv)_reformulated}, which is now equivalent to the following:
\begin{align}
& - L_2 \  e^{-\lambda^* z} \  ( 1 - L_1 e^{- \dfrac{z}{L_1} } + [ L_3 \ z - L_2 \ \sqrt{z} ] \ e^{- \lambda^* z} )    \leq - 4 \beta_2 \ [ L_3 \ z - L_2 \ \sqrt{z} ]^2 \ e^{- 2 \lambda^* z}   \ z \ \sqrt{z}\\
 \iff &     4 \beta_2 \ [ L_3 \ z - L_2 \ \sqrt{z} ]^2 \ e^{ - \lambda^* z}   \ z \ \sqrt{z} -   L_2 \ [ L_3 \  z - L_2 \ \sqrt{z}  ] \  e^{ - \lambda^* z} \leq \ L_2 \   ( 1 - L_1 e^{- \dfrac{z}{L_1}} )    ,\\
 \iff &   4 \beta_2 \ e^{ - \lambda^* z}   \ z^3  \sqrt{z} \  (L_3)^2  +  e^{ - \lambda^* z} \Big( ( 1 - 8 \beta_2  \ z^2)  \ L_3 \  L_2 \ z   + (L_2)^2 \sqrt{z} (1 - 4 \beta_2   \ z^2  )  \Big) \leq \    L_2 \   ( 1 - L_1 e^{- \dfrac{z}{L_1} } ).
\end{align}



We may increase  $L_2$ such that $ 1 - 4 \beta_2   \ \Big( \dfrac{L_2}{e M \lambda^*} \Big)^4  \leq 0$. Then, since $ z > (\dfrac{L_2}{e M \lambda^*})^2$: 

\begin{equation}
    1 - 8 \beta_2  \ z^2 \leq  1 - 4 \beta_2   \ z^2   < 1 - 4 \beta_2   \ \Big( \dfrac{L_2}{e M \lambda^*} \Big) ^4   \leq 0.
\end{equation}

Since $( 1 - 8 \beta_2  \ z^2)$ and $( 1 - 4 \beta_2  \ z^2)$ are negative terms, we need to show $   L_2 \   ( 1 - L_1 e^{- \dfrac{z}{L_1} } )     \geq  4 \beta_2 \ e^{- \lambda^* z}   \ z^3  \sqrt{z} \  (L_3)^2  $.  

Let $g(z) = \beta_2 \ e^{ - \lambda^* z}   \ z^3  \sqrt{z} \  (L_3)^2$ be a $ \mathscr{C}^1([0;- \infty[)$ function. Since $ \lim\limits_{z \rightarrow 0} (g(z)) = 0 $ and  $ \lim\limits_{z \rightarrow + \infty} (g(z)) = 0 $ there exists a constant C (which is independent from $L_2$) such that $g(z) < C \ \ \forall z \geq 0 $. We finally increase $L_2$ so that condition (iv) is verified.
\end{proof}

\subsubsection{Existence and positivity of a critical traveling wave solution}

Now, exactly as in \cite{zhou2019}, we are in a position to define a set of functions
\[
\Gamma = \{(S,I)\in B_\mu(\mathbb R,\mathbb{R}^2)\ |\ \underline{S}\leq S\leq\bar{S},\ \underline{I}\leq I\leq\bar{I}\},
\]
where $B_\mu(\mathbb R,\mathbb{R}^2)$ is the set of two-component continuous functions with each component growing at infinity slower than $e^{\mu|z|}$, as well as an operator $F:\Gamma\to C(\mathbb{R},\mathbb{R}^2)$ that will satisfy the assumptions of the Schauder fixed point theorem and whose fixed point in $\Gamma$ will precisely be the solution $(S,I)$ we seek. Note that the inequalities 1., 2., 3., 4. (beginning of Section \ref{an:sub_super}) are precisely what we use to prove that $F(\Gamma)\subset\Gamma$. Details can be found in \cite{zhou2019}.

The positivity of both $S$ and $I$ comes from the use of the strong maximum principle, again exactly as in \cite{zhou2019}.

\newpage

\section{Study of the
reaction term when $ r = + \infty$ in section \ref{subsec:partial}} \label{ann:partial}

We are searching for conditions implying a pulled monostable wave, using criterion \eqref{eq:pull_cri}. 

\subsection{Conversion occurring in the zygote} \label{ann:pull_zyg}

We rewrite limit equation (\ref{eq:par_zyg_rinf}): \begin{equation}
    \partial_t \p{D}  -  \partial_{xx}^2 \p{D}  \ = \dfrac{ \Big( - \big[  2 (1-c) (1-h) - 1 \big] \ s \  \p{D} - s [ 1-(1-c)(1-h)] + c (1-s)   \Big)  \ (1-\p{D}) \ \p{D} }{ - s  \big[ 2 (1-c) (1-h) - 1 \big] \p{D}^2 - 2 s \big[ 1 - (1-c) (1-h) \big] \p{D} + 1}.
\end{equation}

\vspace{0.2cm}

With $ \mathscr{A}_z := s \  \big[ 2 (1-c) (1-h) - 1 \big] \in [-s, s]$ :


\begin{equation}\label{eq:reac_zyg_p_1}
    \partial_t \p{D}  -  \partial_{xx}^2 \p{D}  \ = \dfrac{ \big(  - \mathscr{A}_z \  \p{D}  + \frac{1}{2} ( \mathscr{A}_z - s) + c (1-s)   \big)  \ (1-\p{D}) \ \p{D} }{-\mathscr{A}_z \p{D}^2 + (\mathscr{A}_z - s) \ \p{D} + 1}.  
\end{equation}

\vspace{0.2cm}

Note that the mean fitness $ \F{}^z(\p{D}) = -\mathscr{A}_z \p{D}^2 + (\mathscr{A}_z - s) \ \p{D} + 1  \in [1-s, 1]$\footnote{$\F{}^{z'}(\p{D}) = \mathscr{A}_z ( 1 - 2 \p{D}) - s \leq 0 $ therefore $\F{}^z(1) \leq \F{}^z(\p{D}) \leq \F{}^z(0)$.}. When $ \mathscr{A}_z \neq 0 $, equation (\ref{eq:reac_zyg_p_1}) can be rewritten:\begin{equation}\label{eq:reac_zyg_p_2}
    \partial_t \p{D}  -  \partial_{xx}^2 \p{D}  \  =  \dfrac{ - \mathscr{A}_z \   (  \p{D} - {\p{D}^*}_z ) \ (1-\p{D}) \ \p{D} }{ -\mathscr{A}_z \p{D}^2 + (\mathscr{A}_z - s)  \ \p{D} + 1}    \quad   \text{with} \quad {\p{D}^*}_z :=  \frac{1}{2}  +  \dfrac{ 2 c \ (1-s) -  s }{ 2 \mathscr{A}_z}.   
\end{equation}

Let us introduce $s_1 := \dfrac{c}{1-h(1-c)}$ and $s_{2,z} := \dfrac{c}{2c + h(1-c)}$.  Note that $ \mathscr{A}_z  > 0 \iff s_1 < s_{2,z}$. We draw the reaction term regarding the sign of $\mathscr{A}_z $ and the $s$ values in Figure \ref{fig:poly_zyg}. 

When $\mathscr{A}_z < 0$ and $ s \in (s_{2,z}, s_1)$, equation (\ref{eq:par_zyg_rinf}) admits two stable steady states (bistability). The final proportion will then strongly depend on the initial condition. On the other hand, when $\mathscr{A}_z > 0$  and $ s \in (s_1, s_{2,z})$, the only possible equilibrium state is a coexistence state: the final proportion $\p{D}$ will be strictly in between $0$ and $1$.

Independently of the sign of $\mathscr{A}_z$, if $s < \min(s_1,s_{2,z})$ the only stable steady state is $\p{D}=1$ meaning that for an initial condition outside of the steady states, we expect that the drive always invades the population. If $s > \max(s_1,s_{2,z})$, the only stable steady state is $\p{D}=0$ meaning that for an initial condition outside of the steady states, we expect that the wild-type always invades the population. \\


\vspace{0.1cm}




\begin{figure}[H]  
        \centering
     \includegraphics[scale = 0.5]{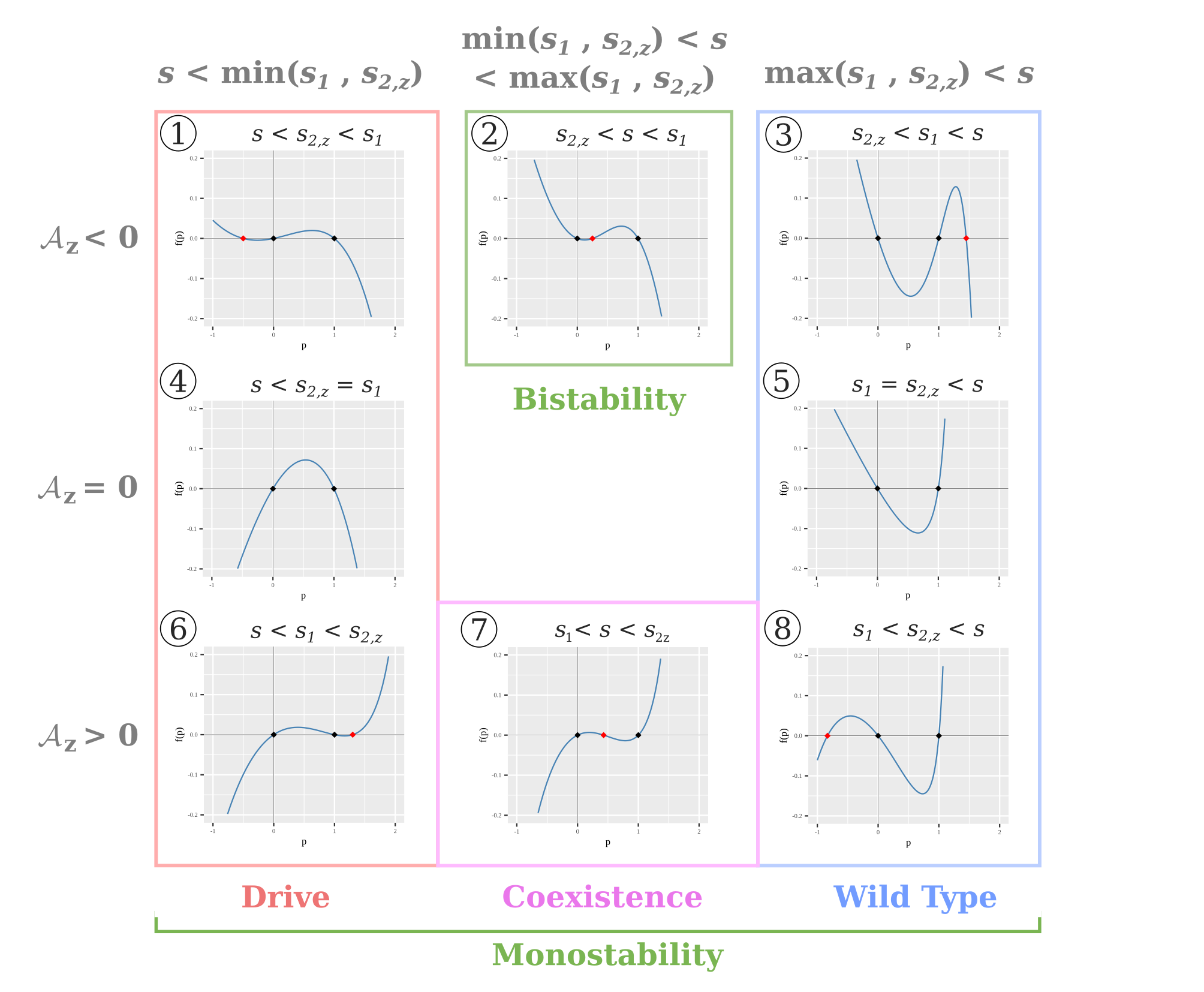}
        \caption{Reaction term of equation \eqref{eq:reac_zyg_p_1}  regarding the sign of $ \mathscr{A}_z = s [2 (1-c)(1-h) - 1]  $ and the $s$ values. The $s$ threshold values are $s_1 = \frac{c}{1-h(1-c)}$ and $s_{2,z} = \frac{c}{2c + h(1-c)}$. The dots on the x axis correspond to the steady states: $\p{D}=0$ and $\p{D}=1$ in black, $ {\p{D}^*}_z = \frac{1}{2}  +  \frac{ 2 c \ (1-s) -  s }{ 2 \mathscr{A}_z} $ in red when it exists.}
        \label{fig:poly_zyg}
\end{figure}

In case of bistability, the wave is always pushed \cite{hadeler1975}; we can dismiss condition \circled{2} in the research of pulled monostable waves. We use criterion \eqref{eq:pull_cri} on monostable cases, i.e. drive invasion \circled{1} \circled{4} \circled{6},  wild type invasion \circled{3} \circled{5} \circled{8}, and coexistence state \circled{7} (the numbers refer to the subgraphs in Figure \ref{fig:poly_zyg}).





\subsubsection{Monostable drive invasion}

\subsubsubsection*{\circled{4}  When $\mathscr{A}_z=0$ and $ s < s_1 = s_{2,z} = \dfrac{2c}{2c + 1} \iff s < 2 c \ (1-s)$}

From equation \eqref{eq:reac_zyg_p_1}, we have for all $\p{D} \in [0,1]$:\begin{equation}
    \sigma(0) - (1-\p{D}) \ \sigma(\p{D}) = \Big( c(1-s) - \dfrac{s}{2}  \Big)  -  \dfrac{ \Big( c(1-s) - \dfrac{s}{2} \Big) \ (1-\p{D})  }{1 - s \p{D}} = \Big( c(1-s) - \dfrac{s}{2} \Big) \  \dfrac{ (1- s) \ \p{D}}{ 1 - s \p{D}} \geq 0,
\end{equation} where $\sigma$ is the selection term defined by equation \eqref{eq:standard_genetics}. Criterion \eqref{eq:pull_cri} is verified.

\vspace{0.6cm}

\subsubsubsection*{ \circled{6} When $\mathscr{A}_z > 0 $ and $ {\p{D}^*}_z > 1 $}

From equation \eqref{eq:reac_zyg_p_2}, we have for all $\p{D} \in [0,1]$: \begin{equation}\label{eq:f0_fp_zyg}
   \sigma(0) - (1-\p{D}) \ \sigma(\p{D})  =  \mathscr{A}_z {\p{D}^*}_z  - \dfrac{ - \mathscr{A}_z \   (  \p{D} - {\p{D}^*}_z ) \ (1-\p{D}) }{ -\mathscr{A}_z \p{D}^2 + (\mathscr{A}_z  - s ) \ \p{D} + 1}  = \mathscr{A}_z \p{D}   \dfrac{ - (\mathscr{A}_z {\p{D}^*}_z + 1) \ \p{D}  + ( \mathscr{A}_z + 1 - s  ) \  {\p{D}^*}_z + 1}{ -\mathscr{A}_z \p{D}^2 + (\mathscr{A}_z  - s ) \  \p{D} + 1} .
\end{equation} 



\vspace{0.2cm}

Note that $ -\mathscr{A}_z \p{D}^2 + (\mathscr{A}_z  - s ) \  \p{D} + 1 > (1-s) > 0$ and $ \mathscr{A}_z  \p{D} > 0 $. The affine term $ - (\mathscr{A}_z {\p{D}^*}_z + 1) \ \p{D}  + ( \mathscr{A}_z + 1 - s  ) \  {\p{D}^*}_z + 1 $ decreases with $\p{D}$. In order to show that it is positive for all $\p{D} \in [0,1]$, we just need to verify that this it is true for $\p{D}=1$:

\begin{equation}
  - (\mathscr{A}_z {\p{D}^*}_z + 1)  + ( \mathscr{A}_z + 1 - s  ) \ {\p{D}^*}_z + 1 =  ( 1-s ) \ {\p{D}^*}_z \geq 0 \quad \Rightarrow     \quad \sigma(0) - (1-\p{D}) \ \sigma(\p{D})  \geq 0 \quad \forall \p{D} \in [0,1].
\end{equation}

Criterion \eqref{eq:pull_cri} is verified.

\subsubsubsection*{ \circled{1}  When $\mathscr{A}_z <0$ and ${\p{D}^*}_z<0$ and $s < s_{2,z} \iff 0 < c - 2 sc - sh + sch   $  }

We consider equation \eqref{eq:f0_fp_zyg} with $-\mathscr{A}_z \p{D}^2 + (\mathscr{A}_z - s ) \  \p{D} + 1 > 1-s > 0 $ and $\mathscr{A}_z \p{D} < 0$.  The affine term $ - (\mathscr{A}_z {\p{D}^*}_z + 1) \ \p{D}  + ( \mathscr{A}_z + 1 - s  ) \  {\p{D}^*}_z + 1$ decreases with $\p{D}$. In order to show that it is negative for all $\p{D} \in [0,1]$, we introduce a condition implying the negativity for $\p{D}=0$:

\begin{equation}
     ( \mathscr{A}_z + 1 - s  ) \  {\p{D}^*}_z + 1  = \dfrac{ ( \mathscr{A}_z + 1 - s  ) \ (\mathscr{A}_z  + 2 c (1- s) -s + 2 \mathscr{A}_z )}{ 2 \mathscr{A}_z}  < 0 
\end{equation} \begin{equation} \label{eq:circled_c_zyg}
     \iff \Big( 1 - 2s [1 - (1-h)(1-c)] \Big) \Big( c - 2 sc - sh + sch \Big) +  s \  \big[ 2 (1-c) (1-h) - 1 \big] > 0
\end{equation}

Criterion \eqref{eq:pull_cri} is verified when condition \eqref{eq:circled_c_zyg} is true.

\subsubsection{Monostable wild-type invasion}

In case of a monostable wild-type invasion, we need to consider the wild-type proportion $\p{W} = 1-\p{D} \in [0,1]$ and rewrite the equation \eqref{eq:reac_zyg_p_1}: \begin{align}\label{eq:reac_zyg_q_1} 
    - \partial_t \p{W}  +  \partial_{xx}^2 \p{W}  \ & = \dfrac{ \big(  - \mathscr{A}_z \  (1-\p{W})  + \frac{1}{2} \ (\mathscr{A}_z - s)  + c (1-s)   \big)  \ (1-\p{W}) \ \p{W} }{-\mathscr{A}_z (1-\p{W})^2 + ( \mathscr{A}_z - s) (1-\p{W}) + 1} \nonumber \\ 
    \iff  \partial_t \p{W}  -  \partial_{xx}^2 \p{W}  \ & = \dfrac{ \big(  - \mathscr{A}_z \  \p{W} + \frac{1}{2} \ (\mathscr{A}_z + s)  - c (1-s)   \big)  \ (1-\p{W}) \ \p{W} }{-\mathscr{A}_z \p{W}^2 + ( \mathscr{A}_z + s) \  \p{W} + (1-s)}
\end{align}

When $ \mathscr{A}_z \neq 0$, equation \eqref{eq:reac_zyg_q_1} can be rewritten:

\begin{equation}\label{eq:reac_zyg_q_2} 
     \partial_t \p{W}  -  \partial_{xx}^2 \p{W}  \ = \dfrac{  - \mathscr{A}_z \  (\p{W} - {\p{W}^*}_z) \ (1-\p{W}) \ \p{W} }{-\mathscr{A}_z \p{W}^2 + ( \mathscr{A}_z + s) \  \p{W} + (1-s)} \quad \text{with} \quad {\p{W}^*}_z  =  \frac{1}{2}  - \dfrac{ 2 c \ (1-s) -  s }{ 2 \mathscr{A}_z} = 1 - {\p{D}^*}_z 
\end{equation}

\subsubsubsection*{ \circled{5}  When $ \mathscr{A}_z = 0$ and  $ s < s_1 = s_{2,z} = \dfrac{2c}{2c + 1} \iff  2 c \ (1-s) < s $} 

From equation \eqref{eq:reac_zyg_q_1} we have for all $\p{W} \in [0,1]$ : \begin{equation}
 \sigma(0) - (1-\p{W}) \ \sigma(\p{W})  = \big(  \dfrac{s}{2}  - c(1-s) \big) \Big(   \dfrac{1 }{  1 - s } - \dfrac{1-\p{W}}{ s \p{W} + 1 - s }  \Big)  \geq   \dfrac{ \p{W} \ \big( \dfrac{s}{2}  - c(1-s) \big) }{ 1 - s } \geq 0
\end{equation}

Criterion \eqref{eq:pull_cri} is verified.


\subsubsubsection*{ \circled{8}  When $ \mathscr{A}_z > 0$ and $ {\p{W}^*}_z = 1 - {\p{D}^*}_z > 1 $}


From equation \eqref{eq:reac_zyg_q_2}, we have for all $\p{W} \in [0,1]$: \begin{equation}\label{eq:g0_gq_zyg}
\begin{aligned}
     \sigma(0) - (1-\p{W}) \ \sigma(\p{W}) & =   \dfrac{ \ \mathscr{A}_z \  {\p{W}^*}_z}{1-s} - \dfrac{ - \mathscr{A}_z \   (  \p{W} - {\p{W}^*}_z) \ (1-\p{W}) }{ -\mathscr{A}_z \p{W}^2 + ( \mathscr{A}_z + s) \  \p{W} + (1-s)} \\ 
     & = \   \mathscr{A}_z \p{W} \ \dfrac{ \  ( - {\p{W}^*}_z+ 1 - s ) \ \p{W}  + {\p{W}^*}_z(\mathscr{A}_z + 1) + ( 1 - s ) }{( 1 - s )(-\mathscr{A}_z \p{W}^2 + ( \mathscr{A}_z + s) \  \p{W} + (1-s) )}.
\end{aligned}
\end{equation} 



\vspace{0.2cm}

Note that $ ( 1 - s )(-\mathscr{A}_z \p{W}^2 + ( \mathscr{A}_z + s) \  \p{W} + (1-s) ) > (1-s)^2 > 0 $ and $ \mathscr{A}_z \p{W}> 0 $.  As $ {\p{W}^*}_z> 1 $, the affine term $( - {\p{W}^*}_z+ 1 - s ) \  \p{W} + {\p{W}^*}_z(\mathscr{A}_z + 1) + ( 1 - s ) $ decreases with $\p{W}$. In order to show that it is positive
for all $\p{W} \in [0, 1]$, we just need to verify that this it is true for $\p{W} = 1$:

\begin{equation}
    ( - {\p{W}^*}_z+ 1 - s )  + {\p{W}^*}_z(\mathscr{A}_z + 1) + ( 1 - s ) =  2 \ (  1 - s )  + \mathscr{A}_z  {\p{W}^*}_z  \geq 0 \quad \Rightarrow     \quad \sigma(0) - (1-\p{W}) \ \sigma(\p{W})   \geq 0 \quad \forall \p{W} \in [0,1].
\end{equation}

Criterion \eqref{eq:pull_cri} is verified.

\subsubsubsection*{ \circled{3}  When $\mathscr{A}_z <0$ and ${\p{W}^*}_z = 1 - {\p{D}^*}_z < 0 $}

We consider equation \eqref{eq:g0_gq_zyg} with $ ( 1 - s )(-\mathscr{A}_z \p{W}^2 + ( \mathscr{A}_z + s) \  \p{W} + (1-s) ) > (1-s)^2 > 0 $ and $ \mathscr{A}_z \p{W} < 0 $. The affine term $( - {\p{W}^*}_z + 1 - s ) \  \p{W} + {\p{W}^*}_z (\mathscr{A}_z + 1) + ( 1 - s ) $ is strictly positive for $\p{W}=1$, therefore criterion \eqref{eq:pull_cri} is not verified.

\subsubsection{Monostable coexistence state}

\subsubsubsection*{ \circled{7}  When $\mathscr{A}_z > 0 $ and $ 0 < {\p{D}^*}_z = 1 - {\p{W}^*}_z < 1 $  }

In the coexistence case, we have to verify that both waves, the drive invasion wave going to the right and the wild-type invasion wave going to the left, are pulled waves (see Figure \ref{fig:coex_illu}).

For the drive invasion wave we consider equation \eqref{eq:f0_fp_zyg} with $ 0 < {\p{D}^*}_z < 1 $ and $\p{D}  \in [0,{\p{D}^*}_z]$ (the term drive wave implies that the proportion of wild type increases after the wave passes; therefore the global stable steady state ${\p{D}^*}_z$ is also the maximum proportion). Once again, we need to prove that the affine term  $ - (\mathscr{A}_z {\p{D}^*}_z + 1) \ \p{D}  + ( \mathscr{A}_z + 1 - s  ) \  {\p{D}^*}_z + 1 $ is positive. As it decreases with $\p{D}  \in [0,{\p{D}^*}_z]$, we determine its sign for $\p{D} = {\p{D}^*}_z$:
\begin{equation}
\begin{aligned}
   - (\mathscr{A}_z {\p{D}^*}_z + 1) \ {\p{D}^*}_z  + ( \mathscr{A}_z + 1 - s  ) \  {\p{D}^*}_z + 1  & =   - \mathscr{A}_z \ ({\p{D}^*}_z)^2   + ( \mathscr{A}_z - s  ) \  {\p{D}^*}_z + 1  \geq 1-s  \geq 0  \\
   & \Rightarrow \ \sigma(0) - (1-\p{D}) \sigma(\p{D})   \geq 0  \quad \forall \p{D}  \in [0,{\p{D}^*}_z].
\end{aligned}
\end{equation} Criterion \eqref{eq:pull_cri} is verified for the drive wave.

For the wild-type invasion wave, we consider equation \eqref{eq:g0_gq_zyg} with $ 0 < {\p{W}^*}_z = 1 - {\p{D}^*}_z < 1 $ and $\p{W}  \in [0,{\p{W}^*}_z]$ (the term wild-type wave implies that the proportion of wild type increases after the wave passes; therefore the global stable steady state ${\p{W}^*}_z$ is also the maximum proportion). Once again, we need to prove that the affine term $( - {\p{W}^*}_z + 1 - s ) \  \p{W} + {\p{W}^*}_z (\mathscr{A}_z + 1) + ( 1 - s )$ is positive. As it decreases with $\p{W}  \in [0,{\p{W}^*}_z]$, we determine its sign for $\p{W} = {\p{W}^*}_z$: \begin{equation}
\begin{aligned}
   & ( - {\p{W}^*}_z + 1 - s ) \  {\p{W}^*}_z + (\mathscr{A}_z + 1) \ {\p{W}^*}_z  + ( 1 - s )  = - ( {\p{W}^*}_z )^2 + (\mathscr{A}_z + 2 - s) \ {\p{W}^*}_z + 1 - s \\ 
   & \geq \min(1, \mathscr{A}_z + 2 ( 1- s))  \geq 0  \quad \Rightarrow  \quad \sigma(0) - (1-\p{W}) \sigma(\p{W})  \geq 0  \quad \forall \p{W}  \in [0,{\p{W}^*}_z].    
\end{aligned}
\end{equation} Criterion \eqref{eq:pull_cri} is verified for the wild-type wave.


\subsection{Conversion occurring in the germline}\label{ann:pull_ger}

We rewrite limit equation (\ref{eq:par_ger_rinf}): \begin{equation}
    \partial_t \p{D}  -  \partial_{xx}^2 \p{D}  \ = \dfrac{ \Big(  - (1-2h) \  s \ \p{D} + [ (1-sh) (1+c) - 1 ]  \Big) \ \p{D} \ (1-\p{D}) }{- s (1-2h) \p{D}^2 - 2 s h \p{D} + 1}
\end{equation}

With $ \mathscr{A}_g := s \ (1-2h) \in [ -s, s ]$  :

\begin{equation}\label{eq:reac_ger_p_1}
    \partial_t \p{D}  -  \partial_{xx}^2 \p{D}   \  = \dfrac{ \big(  - \mathscr{A}_g  \ \p{D} + \frac{1}{2} ( \mathscr{A}_g  - s ) + c (1-sh)  \big) \ \p{D} \ (1-\p{D}) }{- \mathscr{A}_g \p{D}^2 + (\mathscr{A}_g - s) \  \p{D} + 1}.
\end{equation}

Note that the mean fitness $ \F{}^g(\p{D}) = -\mathscr{A}_g \p{D}^2 + (\mathscr{A}_g - s) \ \p{D} + 1  \in [1-s, 1]$\footnote{$\F{}^{g'}(\p{D}) = \mathscr{A}_g ( 1 - 2 \p{D}) - s \leq 0 $ therefore $\F{}^g(1) \leq \F{}^g(\p{D}) \leq \F{}^g(0)$.}.When $\mathscr{A}_g \neq 0$, equation (\ref{eq:reac_ger_p_1}) can be rewritten: 

\begin{equation}\label{eq:reac_ger_p_2}
      \partial_t \p{D}  -  \partial_{xx}^2 \p{D}  \ = \dfrac{- \mathscr{A}_g   \ ( \p{D} -  {\p{D}^*}_g  \big)  \ (1-\p{D}) \ \p{D}}{- \mathscr{A}_g \p{D}^2 + (\mathscr{A}_g - s) \  \p{D} + 1} = f_g(\p{D})  \quad \text{with} \quad {\p{D}^*}_g := \dfrac{1}{2} + \dfrac{2 c \ (1- sh) - s }{ 2 \ \mathscr{A}_g}.
\end{equation}

Let us introduce $s_1 := \dfrac{c}{1-h(1-c)}$ and $s_{2,g} := \dfrac{c}{2ch + h(1-c)} = \dfrac{ c }{h  (1 + c) }$. We draw the reaction term regarding the sign of $  \mathscr{A}_g $ and the $s$ values (in Figure \ref{fig:poly_ger}).  

\begin{figure}[H]  
        \centering
     \includegraphics[scale = 0.5]{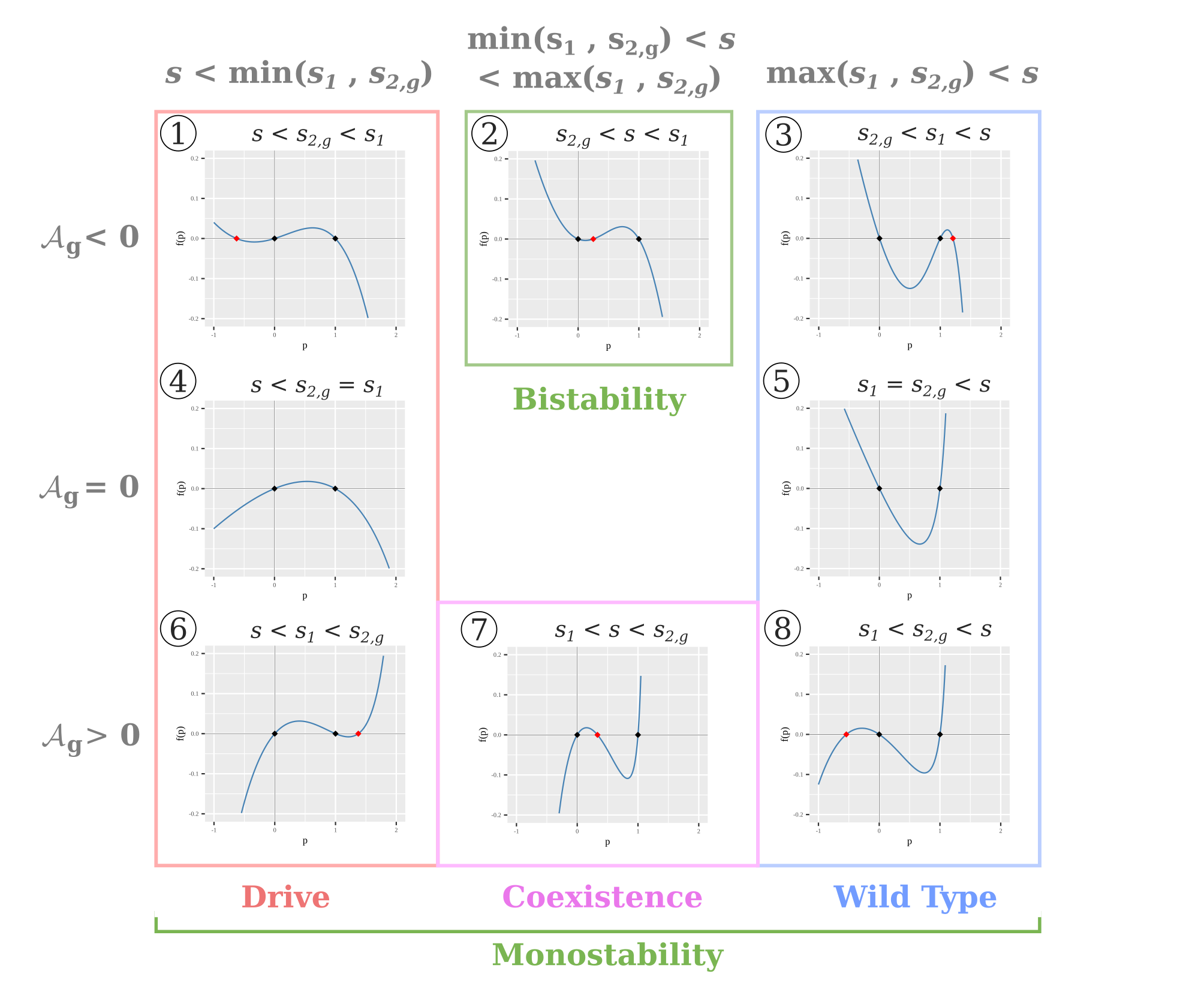}
        \caption{Reaction term of equation \eqref{eq:reac_ger_p_1} regarding the sign of $ \mathscr{A}_g  = s \ (1-2h) $ and the $s$ values. The $s$ threshold values are $s_1 = \frac{c}{1-h(1-c)}$ and $s_{2,g} = \frac{c}{2ch + h(1-c)} = \frac{  c }{ h (1 + c)}$. The dots on the x axis correspond to the steady states: $\p{D}=0$ and $\p{D}=1$ in black, $ {\p{D}^*}_g = \frac{1}{2} + \frac{2 c \ (1- sh) - s }{ 2 \ \mathscr{A}_g} $ in red when it exists.}
        \label{fig:poly_ger}
\end{figure}

In case of bistability, the wave is always pushed \cite{hadeler1975}; we can dismiss condition \circled{2} in the research of pulled monostable waves. We use criterion \eqref{eq:pull_cri} on monostable cases, i.e. drive invasion \circled{1} \circled{4} \circled{6},  wild type invasion \circled{3} \circled{5} \circled{8}, and coexistence state \circled{7} (the numbers refer to the subgraphs in Figure \ref{fig:poly_ger}).


\subsubsection{Monostable drive invasion}

\subsubsubsection*{ \circled{4}  When $\mathscr{A}_g=0 \iff h = \dfrac{1}{2}$ and $ s < s_1 = s_{2,g} = \dfrac{2 c }{c + 1} \iff   \dfrac{s}{2} \ (c + 1) <  c $}

\vspace{0.2cm}

From equation \eqref{eq:reac_ger_p_1}, we have for all $ \p{D} \in [0,1]$: \begin{equation}
   \sigma(0) - (1-\p{D}) \sigma(p) = \Big( c - \dfrac{s}{2} \ (c+1) \Big) - \dfrac{ (c - \dfrac{s}{2} \ (c+1) ) \ (1-\p{D}) }{1  - s \p{D}}   = \Big( c - \dfrac{s}{2} \ (c+1) \Big) \  \dfrac{ (1- s) \ \p{D}  }{ 1 - s \p{D}} \geq 0,
\end{equation}

where $\sigma$ is the selection term defined by equation \eqref{eq:standard_genetics}. Criterion \eqref{eq:pull_cri} is verified.

\subsubsubsection*{ \circled{6}  When $\mathscr{A}_g > 0 $ and $ {\p{D}^*}_g > 1 $}

From equation \eqref{eq:reac_ger_p_2}, we have for all $ \p{D} \in [0,1]$: \begin{equation}\label{eq:f0_fp_ger}
     \sigma(0) - (1-\p{D}) \sigma(p) =  \mathscr{A}_g {\p{D}^*}_g  - \dfrac{ - \mathscr{A}_g \   (  \p{D} - {\p{D}^*}_g ) \ (1-\p{D}) }{ -\mathscr{A}_g \p{D}^2 + (\mathscr{A}_g  - s ) \ \p{D} + 1}  = \mathscr{A}_g \p{D}   \dfrac{ - (\mathscr{A}_g {\p{D}^*}_g + 1) \ \p{D}  + ( \mathscr{A}_g + 1 - s  ) \  {\p{D}^*}_g + 1}{ -\mathscr{A}_g \p{D}^2 + (\mathscr{A}_g - s ) \  \p{D} + 1} .
\end{equation}



Note that $-\mathscr{A}_g \p{D}^2 + (\mathscr{A}_g  - s ) \p{D} + 1 > (1-s) > 0$ and $ \mathscr{A}_g  \p{D} > 0 $. The affine term $- (\mathscr{A}_g {\p{D}^*}_g + 1) \ \p{D}  + ( \mathscr{A}_g + 1 - s ) \  {\p{D}^*}_g + 1$ decreases with $\p{D}$. In order to show that it is positive for all $\p{D} \in [0,1]$, we just need to verify that this it is true for $\p{D}=1$: \begin{equation}
  - (\mathscr{A}_g {\p{D}^*}_g + 1)  + ( \mathscr{A}_g + 1 - s  ) \ {\p{D}^*}_g + 1 =  ( 1-s ) \ {\p{D}^*}_g \geq 0 \quad \Rightarrow     \quad  \sigma(0) - (1-\p{D}) \sigma(p) \geq 0 \quad \forall \p{D} \in [0,1].
\end{equation}

Criterion \eqref{eq:pull_cri} is verified.



\subsubsubsection*{ \circled{1}  When $\mathscr{A}_g <0$ and ${\p{D}^*}_g <0$ and $s < s_{2,g} \iff 0 < c - s h (1 + c )$ }

We consider equation \eqref{eq:f0_fp_ger} with $-\mathscr{A}_g \p{D}^2 + (\mathscr{A}_g - s ) \  \p{D} + 1 > 1-s > 0 $ and $\mathscr{A}_g \p{D} < 0$.  The affine term $ - (\mathscr{A}_g {\p{D}^*}_g + 1) \ \p{D}  + ( \mathscr{A}_g + 1 - s  ) \  {\p{D}^*}_g + 1$ decreases with $\p{D}$. In order to show that it is negative for all $\p{D} \in [0,1]$, we introduce a condition implying the negativity for $p=0$:
\begin{equation}
     ( \mathscr{A}_g + 1 - s  ) \  {\p{D}^*}_g + 1  = \dfrac{(1-2sh) (\mathscr{A}_g  + 2 c (1- sh)-s) + 2 \mathscr{A}_g }{ 2 \mathscr{A}_g}  < 0 
\end{equation} \begin{equation}\label{eq:circled_c_ger}
     \iff ( 1 - 2 sh) (c - sh ( c + 1 ) )  + s (1 - 2h)  > 0 
\end{equation}

Criterion \eqref{eq:pull_cri} is verified when condition \eqref{eq:circled_c_ger} is true.

\subsubsection{Monostable wild-type invasion}

In case of a monostable wild-type invasion, we need to consider the wild-type proportion $\p{W} = 1-\p{D} \in [0,1]$ and rewrite the equation \eqref{eq:reac_ger_p_1}: \begin{align}\label{eq:reac_ger_q_1} 
    - \partial_t \p{W}  +  \partial_{xx}^2 \p{W}  \ & = \dfrac{ \big(  - \mathscr{A}_g \  (1-\p{W})  + \frac{1}{2} \ (\mathscr{A}_g - s)  + c (1-sh)   \big)  \ (1-\p{W}) \ \p{W} }{-\mathscr{A}_g (1-\p{W})^2 + ( \mathscr{A}_g - s) (1-\p{W}) + 1} \nonumber \\ 
    \iff  \partial_t \p{W}  -  \partial_{xx}^2 \p{W}  \ & = \dfrac{ \big(  - \mathscr{A}_g \  \p{W} + \frac{1}{2} \ (\mathscr{A}_g + s)  - c (1-sh)   \big)  \ (1-\p{W}) \ \p{W} }{-\mathscr{A}_g \p{W}^2 + ( \mathscr{A}_g + s) \  \p{W} + (1-s)} 
\end{align}

When $ \mathscr{A}_g \neq 0$, equation \eqref{eq:reac_ger_q_1} can be rewritten:

\begin{equation}\label{eq:reac_ger_q_2} 
     \partial_t \p{W}  -  \partial_{xx}^2 \p{W}  \ = \dfrac{  - \mathscr{A}_g \  (\p{W} - {\p{W}^*}_g) \ (1-\p{W}) \ \p{W} }{-\mathscr{A}_g \p{W}^2 + ( \mathscr{A}_g + s) \  \p{W} + (1-s)}  \quad \text{with} \quad {\p{W}^*}_g  =  \frac{1}{2}  - \dfrac{ 2 c \ (1-sh) -  s }{ 2 \mathscr{A}_g} = 1 - {\p{D}^*}_g 
\end{equation}

\vspace{0.6cm}

\subsubsubsection*{ \circled{5}  When $ \mathscr{A}_g = 0 \iff h=\dfrac{1}{2}$ and  $ s_1 = s_{2,g} = \dfrac{2c}{c + 1} < s \iff  c < \dfrac{s}{2} \ (c+1) $} 

From equation \eqref{eq:reac_ger_q_1} we have for all $\p{W} \in [0,1]$ :\begin{equation}
     \sigma(0) - (1-\p{W}) \sigma(\p{W}) = \big(  \dfrac{s}{2} \ (c+1) - c \big) \Big( \dfrac{1}{  1 - s } - \dfrac{1-\p{W}}{ s \p{W} + 1 - s } \Big) \geq \dfrac{ \p{W} \ \big( \dfrac{s}{2} \ (c+1) - c  \big) }{ 1 - s } \geq 0
\end{equation}

Criterion \eqref{eq:pull_cri} is verified.


\subsubsubsection*{ \circled{8}  When $ \mathscr{A}_g > 0$ and $ {\p{W}^*}_g = 1 - {\p{D}^*}_g > 1 $}


From equation \eqref{eq:reac_ger_q_2}, we have  for all $\p{W} \in [0,1]$:\begin{equation}\label{eq:g0_gq_ger}
\begin{aligned}
    \sigma(0) - (1-\p{W}) \sigma(\p{W}) & =  \dfrac{ \ \mathscr{A}_g \  {\p{W}^*}_g}{1-s} - \dfrac{ - \mathscr{A}_g \   (  \p{W} - {\p{W}^*}_g ) \ (1-\p{W}) }{ -\mathscr{A}_g \p{W}^2 + ( \mathscr{A}_g + s) \  \p{W} + (1-s)} \\
    & = \   \mathscr{A}_g \p{W} \ \dfrac{ \  ( - {\p{W}^*}_g + 1 - s ) \ \p{W}  + {\p{W}^*}_g (\mathscr{A}_g + 1) + ( 1 - s ) }{( 1 - s )(-\mathscr{A}_g \p{W}^2 + ( \mathscr{A}_g + s) \  \p{W} + (1-s) )}.
\end{aligned}
\end{equation} 



\vspace{0.2cm}

Note that $ ( 1 - s )(-\mathscr{A}_g \p{W}^2 + ( \mathscr{A}_g + s) \  \p{W} + (1-s) ) > (1-s)^2 > 0 $ and $ \mathscr{A}_g \p{W}> 0 $. As $ {\p{W}^*}_g > 1 $, the affine term $( - {\p{W}^*}_g + 1 - s ) \  \p{W} + {\p{W}^*}_g (\mathscr{A}_g + 1) + ( 1 - s ) $ decreases with $\p{W}$. In order to show that it is positive
for all $\p{W} \in [0, 1]$, we just need to verify that this it is true for $\p{W} = 1$: \begin{equation}
    ( - {\p{W}^*}_g + 1 - s )  + {\p{W}^*}_g (\mathscr{A}_g + 1) + ( 1 - s ) =  2 \ (  1 - s )  + \mathscr{A}_g  {\p{W}^*}_g  \geq  0 \quad \Rightarrow        \sigma(0) - (1-\p{W}) \sigma(\p{W}) \geq  0 \quad \forall \p{W} \in [0, 1].
\end{equation}

Criterion \eqref{eq:pull_cri} is verified.

\subsubsubsection*{ \circled{3}  When $\mathscr{A}_g <0$ and  ${\p{W}^*}_g = 1 - {\p{D}^*}_z < 0$ }

We consider equation \eqref{eq:g0_gq_ger} with $ ( 1 - s )(-\mathscr{A}_g \p{W}^2 + ( \mathscr{A}_g + s) \  \p{W} + (1-s) ) > (1-s)^2 > 0 $ and $ \mathscr{A}_g \p{W}< 0 $. The affine term  $( - {\p{W}^*}_g + 1 - s ) \  \p{W} + {\p{W}^*}_g (\mathscr{A}_g + 1) + ( 1 - s ) $ is strictly positive for $\p{W}=1$, therefore criterion \eqref{eq:pull_cri} is not verified.

\subsubsection{Monostable coexistence state}

\subsubsubsection*{ \circled{7}  When $\mathscr{A}_g > 0 $ and $ 0 < {\p{D}^*}_g = 1 - {\p{W}^*}_g < 1 $  }

In the coexistence case, we have to verify that both sub-traveling waves, the drive invasion wave going to the right and the wild-type invasion wave going to the left, are pulled waves (see Figure \ref{fig:coex_illu}).

For the drive invasion wave we consider equation \eqref{eq:f0_fp_ger} with $ 0 < {\p{D}^*}_g < 1 $ and $\p{D}  \in [0,{\p{D}^*}_g]$ (the term drive wave implies that the proportion of wild type increases after the wave passes; therefore the global stable steady state ${\p{D}^*}_g$ is also the maximum proportion). Once again, we need to prove that the affine term  $ - (\mathscr{A}_g {\p{D}^*}_g + 1) \ \p{D}  + ( \mathscr{A}_g + 1 - s  ) \  {\p{D}^*}_g + 1 $ is positive. As it decreases with $\p{D}  \in [0,{\p{D}^*}_g]$, we determine its sign for $\p{D} = {\p{D}^*}_g$:
\begin{equation}
\begin{aligned}
   - (\mathscr{A}_g {\p{D}^*}_g + 1) \ {\p{D}^*}_g  + ( \mathscr{A}_g + 1 - s  ) \  {\p{D}^*}_g + 1  & =   - \mathscr{A}_g \ ({\p{D}^*}_g)^2   + ( \mathscr{A}_g - s  ) \  {\p{D}^*}_g + 1 \geq  1-s \geq  0  \\
   & \Rightarrow   \quad \sigma(0) - (1-\p{D}) \sigma(\p{D}) \geq  0 \quad \forall \p{D}  \in [0,{\p{D}^*}_g].
   \end{aligned}
\end{equation} Criterion \eqref{eq:pull_cri} is verified for the drive wave.

For the wild-type invasion wave, we consider equation \eqref{eq:g0_gq_ger} with $ 0 < {\p{W}^*}_g = 1 - {\p{D}^*}_g < 1 $ and $\p{W}  \in [0,{\p{W}^*}_g]$ (the term wild-type wave implies that the proportion of wild type increases after the wave passes; therefore the global stable steady state ${\p{W}^*}_g$ is also the maximum proportion). Once again, we need to prove that the affine term $( - {\p{W}^*}_g + 1 - s ) \  \p{W} + {\p{W}^*}_g (\mathscr{A}_g + 1) + ( 1 - s )$ is positive. As it decreases with $\p{W}  \in [0,{\p{W}^*}_g]$, we determine its sign for $\p{W} = {\p{W}^*}_g$: \begin{equation}
\begin{aligned}
  &   ( - {\p{W}^*}_g + 1 - s ) \  {\p{W}^*}_g + (\mathscr{A}_g + 1) \ {\p{W}^*}_g  + ( 1 - s ) = - ( {\p{W}^*}_g )^2 + (\mathscr{A}_g + 2 - s) \ {\p{W}^*}_g + 1 - s \\
  &  \geq \min(1, \mathscr{A}_g + 2 ( 1- s)) \geq  0   \quad  \Rightarrow   \quad \sigma(0) - (1-\p{W}) \sigma(\p{W}) \geq  0 \quad \forall \p{W}  \in [0,{\p{W}^*}_g].    
\end{aligned}
\end{equation}Criterion \eqref{eq:pull_cri} is verified for the wild-type wave.

\newpage

\section{Heatmap supplementary materials}

\subsection{Effect of fitness disadvantage ($s$) and dominance coefficient ($h$) on drive dynamics, for $ r = + \infty$.} \label{ann:rode_debarre}

In Figure \ref{ann:rode_zyg} and \ref{ann:rode_ger}, we compute heatmaps indicating the stability regime of systems \eqref{eq:par_zyg} and \eqref{eq:par_ger} when $r=+\infty$, depending on the values of ($h$, $s$) and for a fixed value of $c$.

\begin{figure}[H]
\centering
\begin{subfigure}{0.48\textwidth}
    \centering
     \caption{$c=0.25$}
    \includegraphics[scale = 0.5]{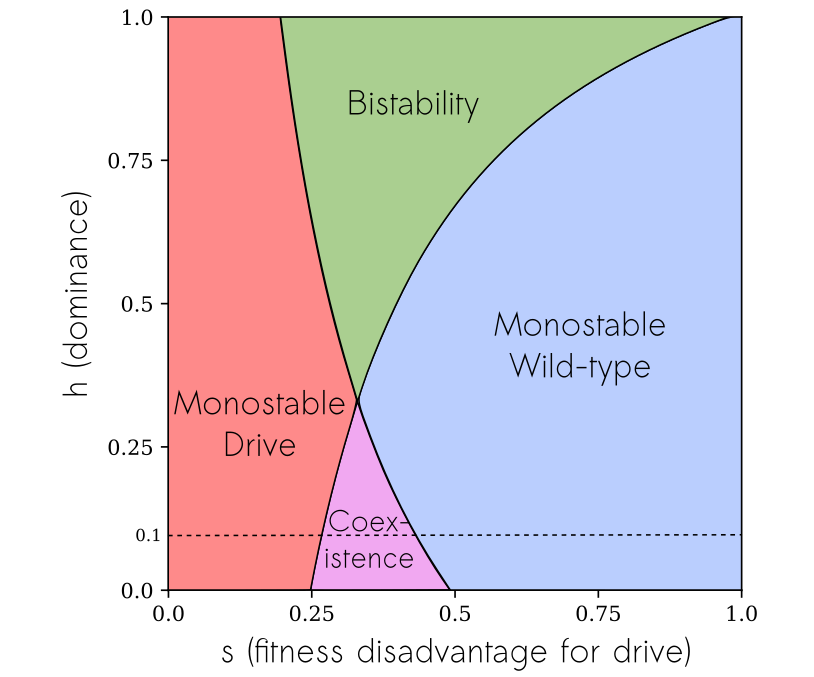}
     \label{ann:rode_zyg_c25}
\end{subfigure}
\hfill
\begin{subfigure}{0.48\textwidth}
    \centering
     \caption{$c=0.75$}
    \includegraphics[scale = 0.5]{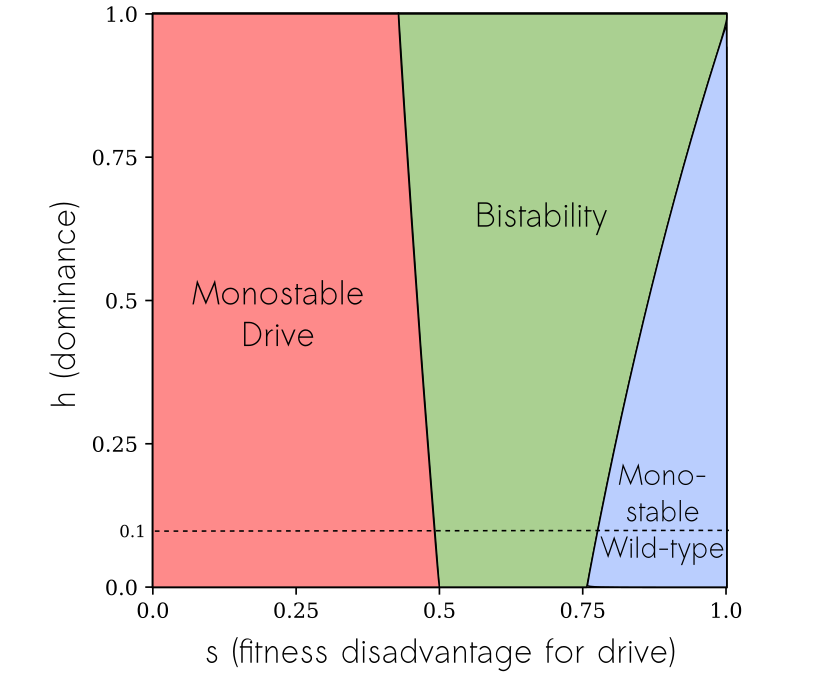}
     \label{ann:rode_zyg_c75}
\end{subfigure}
\caption{Effect of fitness disadvantage ($s$) and dominance coefficient ($h$) on drive dynamics for system \eqref{eq:par_zyg} (when conversion occurs in the zygote) when $ r = + \infty$. Parameters for Figure \ref{fig:heatmap_zyg_coex} ($c=0.25$ and $h=0.1$) and Figure \ref{fig:heatmap_zyg_bist} ($c=0.75$ and $h=0.1$) are materialized by dotted lines.}
\label{ann:rode_zyg}
\end{figure}

\begin{figure}[H]
    \begin{subfigure}{\textwidth}
    \centering
     \caption{$c=0.25$}
    \includegraphics[scale = 0.5]{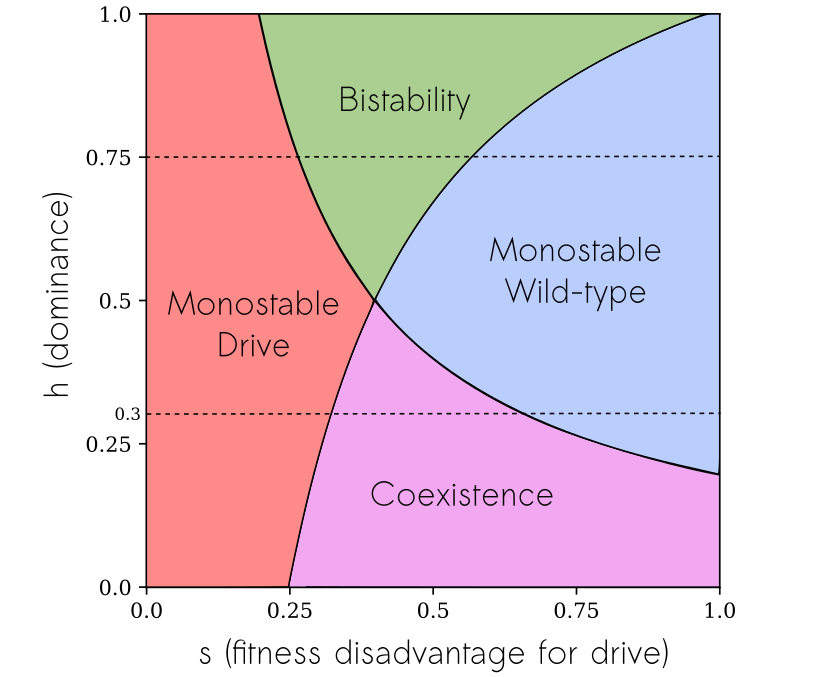}
     \label{ann:rode_ger_c25}
    \end{subfigure}
    \caption{Effect of fitness disadvantage ($s$) and dominance coefficient ($h$) on drive dynamics for system \eqref{eq:par_ger} (when conversion occurs in the germline) when $ r = + \infty$. Parameters for Figure \ref{fig:heatmap_ger_coex} ($c=0.25$ and $h=0.3$) and Figure \ref{fig:heatmap_ger_bist} ($c=0.25$ and $h=0.75$) are materialized by dotted lines.}
    \label{ann:rode_ger}
\end{figure}

Such a figure has already been computed in \cite{rode2019} for $c=0.85$, and conversion occurring in the germline.

\subsection{Heatmap lines}\label{ann:persistence}

\subsubsection{Pure drive line}\label{an:pure_line}

Consider model \eqref{eq:per_zyg} with $\n{WW} = 0$. A well-mixed population containing only drive homozygous individuals will persist in the environment if its equilibrium state $\n{DD}^*$ is strictly positive, i.e. if: 
\begin{equation}
    \n{DD}^* = \min \Big( 0,  1 - \frac{s}{r \ (1-s)} \Big) > 0 \quad \iff \quad r > \dfrac{s}{1-s}
\end{equation}

In case of partial conversion, calculations give the same threshold (consider models \eqref{eq:par_zyg} and \eqref{eq:par_ger} with $\n{DW} = 0$ and $\n{WW} = 0$).

\subsubsection{Composite persistence line} \label{an:composite_line}

Similarly, in case of coexistence, a well-mixed population will persist in the environment only if its equilibrium state $n^*$ is strictly positive. Using Mathematica, we compute this population density equilibrium when conversion occurs in the zygote ($n_z^*$) or in the germline ($n_g^*$) based on systems \eqref{eq:par_zyg_p} and \eqref{eq:par_ger_p}. We obtain the following:

\begin{equation}
n_z^* = \min \Big( 0, 1 - \dfrac{1 -  \F{}^z({\p{D}^*}_z)}{ r \F{}^z({\p{D}^*}_z)} \Big) \quad \text{and} \quad n_g^* = \min \Big(0, 1 - \dfrac{1 -  \F{}^g({\p{D}^*}_g)}{ r \F{}^g({\p{D}^*}_g)} \Big),
\end{equation} 

where the mean fitness $\F{}^z$ and $\F{}^g$ (already defined in Appendix \ref{ann:partial}) are given by:
\begin{equation}
\F{}^z(\p{D}) = -\mathscr{A}_z \p{D}^2 + (\mathscr{A}_z - s) \ \p{D} + 1  \quad \text{and} \quad   \F{}^g(\p{D}) = -\mathscr{A}_g \p{D}^2 + (\mathscr{A}_g - s) \ \p{D} + 1,
\end{equation} 

and the proportions ${\p{D}^*}_z$ and ${\p{D}^*}_g$ (already defined in Appendix \ref{ann:partial}) are given by:   

\begin{equation}
    {\p{D}^*}_z =  \frac{1}{2}  +  \dfrac{ 2 c \ (1-s) -  s }{ 2 \mathscr{A}_z} \quad \text{and} \quad {\p{D}^*}_g := \dfrac{1}{2} + \dfrac{2 c \ (1- sh) - s }{ 2 \ \mathscr{A}_g} 
\end{equation}

with,  
\begin{equation}
\mathscr{A}_z = s \ \big[ 2 (1-c) (1-h) - 1 \big]  \quad \text{and} \quad \mathscr{A}_g = s \ (1-2h).
\end{equation}


Finally, the threshold values for $r$ are given by: \begin{equation}
    n_z^* >0 \quad \iff \quad r > \dfrac{1-\F{}^z({\p{D}^*}_z)}{\F{}^z({\p{D}^*}_z)}
\end{equation}

when conversion occurs in the zygote and,

\begin{equation}
    n_g^* >0 \quad \iff \quad r > \dfrac{1-\F{}^g({\p{D}^*}_g)}{\F{}^g({\p{D}^*}_g)} 
\end{equation}

when conversion occurs in the germline.

\newpage

\subsection*{Conflict of interest} 

The authors declare that they have no conflict of
interest.

\printbibliography

@article{alphey2020,
  title = {Standardizing the Definition of Gene Drive},
  author = {Alphey, Luke S. and Crisanti, Andrea and Randazzo, Filippo (Fil) and Akbari, Omar S.},
  date = {2020-12-08},
  journaltitle = {Proceedings of the National Academy of Sciences},
  volume = {117},
  number = {49},
  pages = {30864--30867},
  publisher = {{Proceedings of the National Academy of Sciences}},
  doi = {10.1073/pnas.2020417117},
  url = {https://www.pnas.org/doi/full/10.1073/pnas.2020417117},
  urldate = {2022-04-04}
}

@misc{an2021,
  title = {Pushed, Pulled and Pushmi-Pullyu Fronts of the {{Burgers-FKPP}} Equation},
  author = {An, Jing and Henderson, Christopher and Ryzhik, Lenya},
  date = {2021-08-17},
  number = {arXiv:2108.07861},
  eprint = {2108.07861},
  eprinttype = {arxiv},
  primaryclass = {math},
  publisher = {{arXiv}},
  doi = {10.48550/arXiv.2108.07861},
  url = {http://arxiv.org/abs/2108.07861},
  urldate = {2022-09-30},
  archiveprefix = {arXiv}
}

@misc{an2022,
  title = {Quantitative Steepness, Semi-{{FKPP}} Reactions, and Pushmi-Pullyu Fronts},
  author = {An, Jing and Henderson, Christopher and Ryzhik, Lenya},
  date = {2022-08-04},
  number = {arXiv:2208.02880},
  eprint = {2208.02880},
  eprinttype = {arxiv},
  primaryclass = {math},
  publisher = {{arXiv}},
  doi = {10.48550/arXiv.2208.02880},
  url = {http://arxiv.org/abs/2208.02880},
  urldate = {2022-09-12},
  archiveprefix = {arXiv}
}

@inproceedings{aronson1975,
  title = {Nonlinear Diffusion in Population Genetics, Combustion, and Nerve Pulse Propagation},
  booktitle = {Partial {{Differential Equations}} and {{Related Topics}}},
  author = {Aronson, D. G. and Weinberger, H. F.},
  editor = {Goldstein, Jerome A.},
  date = {1975},
  series = {Lecture {{Notes}} in {{Mathematics}}},
  pages = {5--49},
  publisher = {{Springer}},
  location = {{Berlin, Heidelberg}},
  doi = {10.1007/BFb0070595},
  isbn = {978-3-540-37440-4},
  langid = {english}
}

@article{aronson1978,
  title = {Multidimensional Nonlinear Diffusion Arising in Population Genetics},
  author = {Aronson, D. G. and Weinberger, Hans F},
  date = {1978-10},
  journaltitle = {Advances in Mathematics},
  volume = {30},
  number = {1},
  pages = {33--76},
  issn = {0001-8708},
  doi = {10.1016/0001-8708(78)90130-5},
  url = {http://www.scopus.com/inward/record.url?scp=49349119963&partnerID=8YFLogxK},
  urldate = {2021-12-15}
}

@misc{avery2022,
  title = {Pushed-to-Pulled Front Transitions: Continuation, Speed Scalings, and Hidden Monotonicity},
  shorttitle = {Pushed-to-Pulled Front Transitions},
  author = {Avery, Montie and Holzer, Matt and Scheel, Arnd},
  date = {2022-06-20},
  number = {arXiv:2206.09989},
  eprint = {2206.09989},
  eprinttype = {arxiv},
  primaryclass = {nlin},
  publisher = {{arXiv}},
  doi = {10.48550/arXiv.2206.09989},
  url = {http://arxiv.org/abs/2206.09989},
  urldate = {2022-09-30},
  archiveprefix = {arXiv}
}

@article{barton1979,
  title = {The Dynamics of Hybrid Zones},
  author = {Barton, N. H.},
  date = {1979-12},
  journaltitle = {Heredity},
  volume = {43},
  number = {3},
  pages = {341--359},
  publisher = {{Nature Publishing Group}},
  issn = {1365-2540},
  doi = {10.1038/hdy.1979.87},
  url = {https://www.nature.com/articles/hdy197987},
  urldate = {2022-04-18},
  issue = {3},
  langid = {english}
}

@article{beaghton2016,
  title = {Gene Drive through a Landscape: {{Reaction}}–Diffusion Models of Population Suppression and Elimination by a Sex Ratio Distorter},
  shorttitle = {Gene Drive through a Landscape},
  author = {Beaghton, Andrea and Beaghton, Pantelis John and Burt, Austin},
  date = {2016-04-01},
  journaltitle = {Theoretical Population Biology},
  shortjournal = {Theoretical Population Biology},
  volume = {108},
  pages = {51--69},
  issn = {0040-5809},
  doi = {10.1016/j.tpb.2015.11.005},
  url = {https://www.sciencedirect.com/science/article/pii/S0040580915001227},
  urldate = {2022-10-10},
  langid = {english}
}

@article{birzu2018,
  title = {Fluctuations Uncover a Distinct Class of Traveling Waves},
  author = {Birzu, Gabriel and Hallatschek, Oskar and Korolev, Kirill S.},
  date = {2018-04-17},
  journaltitle = {Proceedings of the National Academy of Sciences},
  volume = {115},
  number = {16},
  pages = {E3645-E3654},
  publisher = {{Proceedings of the National Academy of Sciences}},
  doi = {10.1073/pnas.1715737115},
  url = {https://www.pnas.org/doi/abs/10.1073/pnas.1715737115},
  urldate = {2022-10-07}
}

@article{buchman2020,
  title = {Broad Dengue Neutralization in Mosquitoes Expressing an Engineered Antibody},
  author = {Buchman, Anna and Gamez, Stephanie and Li, Ming and Antoshechkin, Igor and Li, Hsing-Han and Wang, Hsin-Wei and Chen, Chun-Hong and Klein, Melissa J. and Duchemin, Jean-Bernard and Jr, James E. Crowe and Paradkar, Prasad N. and Akbari, Omar S.},
  date = {2020-01-16},
  journaltitle = {PLOS Pathogens},
  shortjournal = {PLOS Pathogens},
  volume = {16},
  number = {1},
  pages = {e1008103},
  publisher = {{Public Library of Science}},
  issn = {1553-7374},
  doi = {10.1371/journal.ppat.1008103},
  url = {https://journals.plos.org/plospathogens/article?id=10.1371/journal.ppat.1008103},
  urldate = {2022-02-28},
  langid = {english}
}

@article{burt2003,
  title = {Site-Specific Selfish Genes as Tools for the Control and Genetic Engineering of Natural Populations.},
  author = {Burt, Austin},
  date = {2003-05-07},
  journaltitle = {Proceedings of the Royal Society B: Biological Sciences},
  shortjournal = {Proc Biol Sci},
  volume = {270},
  number = {1518},
  eprint = {12803906},
  eprinttype = {pmid},
  pages = {921--928},
  issn = {0962-8452},
  doi = {10.1098/rspb.2002.2319},
  url = {https://royalsocietypublishing.org/doi/pdf/10.1098/rspb.2002.2319},
  urldate = {2021-04-23},
  pmcid = {PMC1691325}
}

@article{champer2021,
  title = {Suppression Gene Drive in Continuous Space Can Result in Unstable Persistence of Both Drive and Wild-Type Alleles},
  author = {Champer, Jackson and Kim, Isabel K. and Champer, Samuel E. and Clark, Andrew G. and Messer, Philipp W.},
  date = {2021-02},
  journaltitle = {Molecular Ecology},
  shortjournal = {Mol Ecol},
  volume = {30},
  number = {4},
  eprint = {33404162},
  eprinttype = {pmid},
  pages = {1086--1101},
  issn = {1365-294X},
  doi = {10.1111/mec.15788},
  langid = {english},
  pmcid = {PMC7887089}
}

@book{committeeongenedriveresearch2016,
  title = {Gene {{Drives}} on the {{Horizon}}: {{Advancing Science}}, {{Navigating Uncertainty}}, and {{Aligning Research}} with {{Public Values}}},
  shorttitle = {Gene {{Drives}} on the {{Horizon}}},
  author = {{Committee on Gene Drive Research in Non-Human Organisms: Recommendations for Responsible Conduct} and {Board on Life Sciences} and {Division on Earth and Life Studies} and {National Academies of Sciences, Engineering, and Medicine}},
  date = {2016},
  eprint = {27536751},
  eprinttype = {pmid},
  publisher = {{National Academies Press (US)}},
  location = {{Washington (DC)}},
  url = {http://www.ncbi.nlm.nih.gov/books/NBK379277/},
  urldate = {2021-04-23},
  isbn = {978-0-309-43787-5},
  langid = {english}
}

@article{dahirel2021,
  title = {Shifts from Pulled to Pushed Range Expansions Caused by Reduction of Landscape Connectivity},
  author = {Dahirel, Maxime and Bertin, Aline and Haond, Marjorie and Blin, Aurélie and Lombaert, Eric and Calcagno, Vincent and Fellous, Simon and Mailleret, Ludovic and Malausa, Thibaut and Vercken, Elodie},
  date = {2021},
  journaltitle = {Oikos},
  volume = {130},
  number = {5},
  pages = {708--724},
  issn = {1600-0706},
  doi = {10.1111/oik.08278},
  url = {https://onlinelibrary.wiley.com/doi/abs/10.1111/oik.08278},
  urldate = {2022-09-30},
  langid = {english}
}

@article{deredec2008,
  title = {The {{Population Genetics}} of {{Using Homing Endonuclease Genes}} in {{Vector}} and {{Pest Management}}},
  author = {Deredec, Anne and Burt, Austin and Godfray, H. C. J.},
  date = {2008-08},
  journaltitle = {Genetics},
  shortjournal = {Genetics},
  volume = {179},
  number = {4},
  eprint = {18660532},
  eprinttype = {pmid},
  pages = {2013--2026},
  issn = {0016-6731},
  doi = {10.1534/genetics.108.089037},
  url = {https://www.ncbi.nlm.nih.gov/pmc/articles/PMC2516076/pdf/GEN17942013.pdf},
  urldate = {2021-04-23},
  pmcid = {PMC2516076}
}

@article{dhole2020,
  title = {Gene {{Drive Dynamics}} in {{Natural Populations}}: {{The Importance}} of {{Density Dependence}}, {{Space}}, and {{Sex}}},
  shorttitle = {Gene {{Drive Dynamics}} in {{Natural Populations}}},
  author = {Dhole, Sumit and Lloyd, Alun L. and Gould, Fred},
  date = {2020-11},
  journaltitle = {Annual Review of Ecology, Evolution, and Systematics},
  shortjournal = {Annu Rev Ecol Evol Syst},
  volume = {51},
  number = {1},
  eprint = {34366722},
  eprinttype = {pmid},
  pages = {505--531},
  issn = {1543-592X},
  doi = {10.1146/annurev-ecolsys-031120-101013},
  langid = {english},
  pmcid = {PMC8340601}
}

@article{esvelt2014,
  title = {Concerning {{RNA-guided}} Gene Drives for the Alteration of Wild Populations},
  author = {Esvelt, Kevin M and Smidler, Andrea L and Catteruccia, Flaminia and Church, George M},
  editor = {Tautz, Diethard},
  date = {2014-07-17},
  journaltitle = {eLife},
  volume = {3},
  pages = {e03401},
  publisher = {{eLife Sciences Publications, Ltd}},
  issn = {2050-084X},
  doi = {10.7554/eLife.03401},
  url = {https://doi.org/10.7554/eLife.03401},
  urldate = {2021-04-27}
}

@article{gantz2015,
  title = {The Mutagenic Chain Reaction: {{A}} Method for Converting Heterozygous to Homozygous Mutations},
  shorttitle = {The Mutagenic Chain Reaction},
  author = {Gantz, Valentino M. and Bier, Ethan},
  date = {2015-04-24},
  journaltitle = {Science},
  volume = {348},
  number = {6233},
  eprint = {25908821},
  eprinttype = {pmid},
  pages = {442--444},
  publisher = {{American Association for the Advancement of Science}},
  issn = {0036-8075, 1095-9203},
  doi = {10.1126/science.aaa5945},
  url = {http://reviverestore.org/wp-content/uploads/2015/03/The-mutagenic-chain-reaction-A-method-for-converting-heterozygous-to-homozygous-mutations-Valentino-M.-Gantz_-and-Ethan-Bier.pdf},
  urldate = {2021-04-23},
  langid = {english}
}

@article{gantz2015a,
  title = {Highly Efficient {{Cas9-mediated}} Gene Drive for Population Modification of the Malaria Vector Mosquito {{Anopheles}} Stephensi},
  author = {Gantz, Valentino and Jasinskiene, Nijole and Tatarenkova, Olga and Fazekas, Aniko and Macias, Vanessa and Bier, Ethan and James, Anthony},
  date = {2015-11-23},
  journaltitle = {Proceedings of the National Academy of Sciences},
  shortjournal = {Proceedings of the National Academy of Sciences},
  volume = {112},
  doi = {10.1073/pnas.1521077112},
  url = {https://www.researchgate.net/publication/284458599_Highly_efficient_Cas9-mediated_gene_drive_for_population_modification_of_the_malaria_vector_mosquito_Anopheles_stephensi}
}

@article{girardin2021,
  title = {Demographic Feedbacks Can Hamper the Spatial Spread of a Gene Drive},
  author = {Girardin, Léo and Débarre, Florence},
  date = {2021-12-04},
  journaltitle = {Journal of Mathematical Biology},
  shortjournal = {J. Math. Biol.},
  volume = {83},
  number = {6},
  pages = {67},
  issn = {1432-1416},
  doi = {10.1007/s00285-021-01702-2},
  url = {https://doi.org/10.1007/s00285-021-01702-2},
  urldate = {2021-12-16},
  langid = {english}
}

@article{grunwald2019,
  title = {Super-{{Mendelian}} Inheritance Mediated by {{CRISPR-Cas9}} in the Female Mouse Germline},
  author = {Grunwald, Hannah A. and Gantz, Valentino M. and Poplawski, Gunnar and Xu, Xiang-Ru S. and Bier, Ethan and Cooper, Kimberly L.},
  date = {2019-02},
  journaltitle = {Nature},
  shortjournal = {Nature},
  volume = {566},
  number = {7742},
  eprint = {30675057},
  eprinttype = {pmid},
  pages = {105--109},
  issn = {1476-4687},
  doi = {10.1038/s41586-019-0875-2},
  langid = {english},
  pmcid = {PMC6367021}
}

@article{hadeler1975,
  title = {Travelling Fronts in Nonlinear Diffusion Equations},
  author = {Hadeler, K. P. and Rothe, F.},
  date = {1975-09-01},
  journaltitle = {Journal of Mathematical Biology},
  shortjournal = {J. Math. Biology},
  volume = {2},
  number = {3},
  pages = {251--263},
  issn = {1432-1416},
  doi = {10.1007/BF00277154},
  url = {https://doi.org/10.1007/BF00277154},
  urldate = {2022-01-31},
  langid = {english}
}

@article{hammond2015,
  title = {A {{CRISPR-Cas9 Gene Drive System Targeting Female Reproduction}} in the {{Malaria Mosquito}} Vector {{Anopheles}} Gambiae},
  author = {Hammond, Andrew and Galizi, Roberto and Kyrou, Kyros and Simoni, Alekos and Siniscalchi, Carla and Katsanos, Dimitris and Gribble, Matthew and Baker, Dean and Marois, Eric and Russell, Steven and Burt, Austin and Windbichler, Nikolai and Crisanti, Andrea and Nolan, Tony},
  date = {2015-12-07},
  journaltitle = {Nature biotechnology},
  shortjournal = {Nature biotechnology},
  volume = {34},
  doi = {10.1038/nbt.3439},
  url = {https://www.researchgate.net/publication/286219237_A_CRISPR-Cas9_Gene_Drive_System_Targeting_Female_Reproduction_in_the_Malaria_Mosquito_vector_Anopheles_gambiae}
}

@online{holzer2022,
  title = {Personal Communication, {{Conference}} "{{Parabolic}} and  Kinetic Models in Population Dynamics", {{Labex CIMI}}, {{Toulouse}}.},
  author = {Holzer, Matt},
  date = {2022},
  url = {https://indico.math.cnrs.fr/event/7589/}
}

@article{jinek2012,
  title = {A Programmable Dual {{RNA-guided DNA}} Endonuclease in Adaptive Bacterial Immunity},
  author = {Jinek, Martin and Chylinski, Krzysztof and Fonfara, Ines and Hauer, Michael and Doudna, Jennifer A. and Charpentier, Emmanuelle},
  date = {2012-08-17},
  journaltitle = {Science (New York, N.Y.)},
  shortjournal = {Science},
  volume = {337},
  number = {6096},
  eprint = {22745249},
  eprinttype = {pmid},
  pages = {816--821},
  issn = {0036-8075},
  doi = {10.1126/science.1225829},
  url = {https://www.ncbi.nlm.nih.gov/pmc/articles/PMC6286148/pdf/nihms-995853.pdf},
  urldate = {2021-04-23},
  pmcid = {PMC6286148}
}

@article{kyrou2018,
  title = {A {{CRISPR}}–{{Cas9}} Gene Drive Targeting Doublesex Causes Complete Population Suppression in Caged {{Anopheles}} Gambiae Mosquitoes},
  author = {Kyrou, Kyros and Hammond, Andrew M. and Galizi, Roberto and Kranjc, Nace and Burt, Austin and Beaghton, Andrea K. and Nolan, Tony and Crisanti, Andrea},
  date = {2018-11},
  journaltitle = {Nature Biotechnology},
  shortjournal = {Nat Biotechnol},
  volume = {36},
  number = {11},
  pages = {1062--1066},
  publisher = {{Nature Publishing Group}},
  issn = {1546-1696},
  doi = {10.1038/nbt.4245},
  url = {https://www.nature.com/articles/nbt.4245},
  urldate = {2022-02-28},
  issue = {11},
  langid = {english}
}

@article{li2020,
  title = {Can {{CRISPR}} Gene Drive Work in Pest and Beneficial Haplodiploid Species?},
  author = {Li, Jun and Aidlin Harari, Ofer and Doss, Anna-Louise and Walling, Linda L. and Atkinson, Peter W. and Morin, Shai and Tabashnik, Bruce E.},
  date = {2020},
  journaltitle = {Evolutionary Applications},
  volume = {13},
  number = {9},
  pages = {2392--2403},
  issn = {1752-4571},
  doi = {10.1111/eva.13032},
  url = {https://onlinelibrary.wiley.com/doi/abs/10.1111/eva.13032},
  urldate = {2022-10-10},
  langid = {english}
}

@article{liu2022,
  title = {Modelling Homing Suppression Gene Drive in Haplodiploid Organisms},
  author = {Liu, Yiran and Champer, Jackson},
  date = {2022-04-13},
  journaltitle = {Proceedings of the Royal Society B: Biological Sciences},
  shortjournal = {Proc. R. Soc. B.},
  volume = {289},
  number = {1972},
  pages = {20220320},
  issn = {0962-8452, 1471-2954},
  doi = {10.1098/rspb.2022.0320},
  url = {https://royalsocietypublishing.org/doi/10.1098/rspb.2022.0320},
  urldate = {2022-05-18},
  langid = {english}
}

@article{nadin2018,
  title = {Hindrances to Bistable Front Propagation: Application to {{Wolbachia}} Invasion},
  shorttitle = {Hindrances to Bistable Front Propagation},
  author = {Nadin, Grégoire and Strugarek, Martin and Vauchelet, Nicolas},
  date = {2018-05-01},
  journaltitle = {Journal of Mathematical Biology},
  shortjournal = {J. Math. Biol.},
  volume = {76},
  number = {6},
  pages = {1489--1533},
  issn = {1432-1416},
  doi = {10.1007/s00285-017-1181-y},
  url = {https://doi.org/10.1007/s00285-017-1181-y},
  urldate = {2022-09-12},
  langid = {english}
}

@article{neve2018,
  title = {Gene Drive Systems: Do They Have a Place in Agricultural Weed Management?},
  shorttitle = {Gene Drive Systems},
  author = {Neve, Paul},
  date = {2018},
  journaltitle = {Pest Management Science},
  volume = {74},
  number = {12},
  pages = {2671--2679},
  issn = {1526-4998},
  doi = {10.1002/ps.5137},
  url = {https://onlinelibrary.wiley.com/doi/abs/10.1002/ps.5137},
  urldate = {2022-03-01},
  langid = {english}
}

@article{north2020,
  title = {Modelling the Suppression of a Malaria Vector Using a {{CRISPR-Cas9}} Gene Drive to Reduce Female Fertility},
  author = {North, Ace R. and Burt, Austin and Godfray, H. Charles J.},
  date = {2020-08-11},
  journaltitle = {BMC Biology},
  shortjournal = {BMC Biol},
  volume = {18},
  eprint = {32782000},
  eprinttype = {pmid},
  pages = {98},
  issn = {1741-7007},
  doi = {10.1186/s12915-020-00834-z},
  url = {https://www.ncbi.nlm.nih.gov/pmc/articles/PMC7422583/},
  urldate = {2022-02-28},
  pmcid = {PMC7422583}
}

@book{otto2011,
  title = {A {{Biologist}}'s {{Guide}} to {{Mathematical Modeling}} in {{Ecology}} and {{Evolution}}},
  author = {Otto, Sarah P. and Day, Troy},
  date = {2011-09-19},
  publisher = {{Princeton University Press}},
  doi = {10.1515/9781400840915},
  url = {https://www.degruyter.com/document/doi/10.1515/9781400840915/html},
  urldate = {2022-10-10},
  isbn = {978-1-4008-4091-5},
  langid = {english}
}

@article{peischl2013,
  title = {On the Accumulation of Deleterious Mutations during Range Expansions},
  author = {Peischl, S. and Dupanloup, I. and Kirkpatrick, M. and Excoffier, L.},
  date = {2013},
  journaltitle = {Molecular Ecology},
  volume = {22},
  number = {24},
  pages = {5972--5982},
  issn = {1365-294X},
  doi = {10.1111/mec.12524},
  url = {https://onlinelibrary.wiley.com/doi/abs/10.1111/mec.12524},
  urldate = {2022-09-30},
  langid = {english}
}

@article{rode2019,
  title = {Population Management Using Gene Drive: Molecular Design, Models of~Spread Dynamics and Assessment of Ecological Risks},
  shorttitle = {Population Management Using Gene Drive},
  author = {Rode, Nicolas O. and Estoup, Arnaud and Bourguet, Denis and Courtier-Orgogozo, Virginie and Débarre, Florence},
  date = {2019-08-01},
  journaltitle = {Conservation Genetics},
  shortjournal = {Conserv Genet},
  volume = {20},
  number = {4},
  pages = {671--690},
  issn = {1572-9737},
  doi = {10.1007/s10592-019-01165-5},
  url = {https://doi.org/10.1007/s10592-019-01165-5},
  urldate = {2021-04-29},
  langid = {english}
}

@article{roques2012,
  title = {Allee Effect Promotes Diversity in Traveling Waves of Colonization},
  author = {Roques, Lionel and Garnier, Jimmy and Hamel, François and Klein, Etienne K.},
  date = {2012-06-05},
  journaltitle = {Proceedings of the National Academy of Sciences},
  volume = {109},
  number = {23},
  pages = {8828--8833},
  publisher = {{Proceedings of the National Academy of Sciences}},
  doi = {10.1073/pnas.1201695109},
  url = {https://www.pnas.org/doi/abs/10.1073/pnas.1201695109},
  urldate = {2022-09-12}
}

@article{strugarek2016,
  title = {Reduction to a {{Single Closed Equation}} for 2-by-2 {{Reaction-Diffusion Systems}} of {{Lotka--Volterra Type}}},
  author = {Strugarek, Martin and Vauchelet, Nicolas},
  date = {2016-01-01},
  journaltitle = {SIAM Journal on Applied Mathematics},
  shortjournal = {SIAM J. Appl. Math.},
  volume = {76},
  number = {5},
  pages = {2060--2080},
  publisher = {{Society for Industrial and Applied Mathematics}},
  issn = {0036-1399},
  doi = {10.1137/16M1059217},
  url = {https://epubs.siam.org/doi/abs/10.1137/16M1059217},
  urldate = {2021-10-20}
}

@article{tanaka2017,
  title = {Spatial Gene Drives and Pushed Genetic Waves},
  author = {Tanaka, Hidenori and Stone, Howard A. and Nelson, David R.},
  date = {2017-08-08},
  journaltitle = {Proceedings of the National Academy of Sciences of the United States of America},
  shortjournal = {Proc Natl Acad Sci U S A},
  volume = {114},
  number = {32},
  eprint = {28743753},
  eprinttype = {pmid},
  pages = {8452--8457},
  issn = {1091-6490},
  doi = {10.1073/pnas.1705868114},
  url = {https://www.pnas.org/content/pnas/114/32/8452.full.pdf},
  langid = {english},
  pmcid = {PMC5559037}
}

@article{turelli2017,
  title = {Deploying Dengue-Suppressing {{Wolbachia}} : {{Robust}} Models Predict Slow but Effective Spatial Spread in {{Aedes}} Aegypti},
  shorttitle = {Deploying Dengue-Suppressing {{Wolbachia}}},
  author = {Turelli, Michael and Barton, Nicholas H.},
  date = {2017-06-01},
  journaltitle = {Theoretical Population Biology},
  shortjournal = {Theoretical Population Biology},
  volume = {115},
  pages = {45--60},
  issn = {0040-5809},
  doi = {10.1016/j.tpb.2017.03.003},
  url = {https://www.sciencedirect.com/science/article/pii/S0040580916301046},
  urldate = {2022-09-12},
  langid = {english}
}

@article{unckless2015,
  title = {Modeling the {{Manipulation}} of {{Natural Populations}} by the {{Mutagenic Chain Reaction}}},
  author = {Unckless, Robert L and Messer, Philipp W and Connallon, Tim and Clark, Andrew G},
  date = {2015-10-01},
  journaltitle = {Genetics},
  shortjournal = {Genetics},
  volume = {201},
  number = {2},
  pages = {425--431},
  issn = {1943-2631},
  doi = {10.1534/genetics.115.177592},
  url = {https://doi.org/10.1534/genetics.115.177592},
  urldate = {2021-04-28}
}

@article{zhou2019,
  title = {Critical Traveling Waves in a Diffusive Disease Model},
  author = {Zhou, Jiangbo and Song, Liyuan and Wei, Jingdong and Xu, Haimei},
  date = {2019-08-15},
  journaltitle = {Journal of Mathematical Analysis and Applications},
  shortjournal = {Journal of Mathematical Analysis and Applications},
  volume = {476},
  number = {2},
  pages = {522--538},
  issn = {0022-247X},
  doi = {10.1016/j.jmaa.2019.03.066},
  url = {https://www.sciencedirect.com/science/article/pii/S0022247X19302872},
  urldate = {2021-10-20},
  langid = {english}
}

@online{zotero-193,
  title = {Male {{Bias}} and {{Female Fertility}}},
  url = {https://targetmalaria.org/what-we-do/our-approach/male-bias-and-female-fertility/},
  urldate = {2022-02-28},
  langid = {english},
  organization = {{Target Malaria}}
}

\end{document}